\documentclass[svgnames, oneside]{compositio}

\usepackage[english]{babel}
\usepackage[utf8]{inputenc}
\usepackage[OT2,T1]{fontenc}
\usepackage{amsmath}
\usepackage{amsthm, ifthen, amsfonts, amssymb,
srcltx, amsopn, enumerate, mathabx, ulem}
\usepackage[usenames,dvipsnames]{xcolor}
\usepackage{mathtools}
\usepackage{enumitem}
\usepackage{overpic}
\usepackage{epigraph}

\usepackage[colorlinks=true, pdfstartview=FitV, linkcolor=blue, citecolor=blue, urlcolor=blue]{hyperref}
\usepackage[nameinlink]{cleveref}
\usepackage{url}
\usepackage{graphicx}
\usepackage{tikz}\usetikzlibrary{cd}
\usepackage{import}

\newtheorem{thm}{Theorem}[section]
\newtheorem{prop}[thm]{Proposition}
\newtheorem{lem}[thm]{Lemma}
\newtheorem{cor}[thm]{Corollary}
\newtheorem{claim}[thm]{Claim}

\newtheorem{thmi}{Theorem}
\newtheorem{cori}[thmi]{Corollary}

\theoremstyle{definition}
\newtheorem{defn}[thm]{Definition}
\newtheorem{remark}[thm]{Remark}

\newtheorem{notation}[thm]{Notation}
\newtheorem{construction}[thm]{Construction}
\newtheorem{hyp}[thm]{Hypothesis}

\newenvironment{claimproof}{\renewcommand{\qedsymbol}{$\blacksquare$}\begin{proof}}{\end{proof}}

\newcommand*\mf[1]{\mathfrak{#1}}
\newcommand*\mc[1]{\mathcal{#1}}
\newcommand{\ofS}{\overline{\mf S}}
\newcommand{\wfS}{\widetilde{\mf S}}

\newcommand{\N}{\mathbb{N}}

\newcommand{\R}{\mathbb{R}}
\newcommand{\Z}{\mathbb{Z}}

\newcommand*\nest{\sqsubseteq}

\newcommand*\propnest{\sqsubsetneq}

\newcommand*\trans{\pitchfork}
\newcommand*\orth{\bot}

\newcommand{\llangle}{\left\langle\!\left\langle}
\newcommand{\rrangle}{\right\rangle\!\right\rangle}

\renewcommand{\hat}{\widehat}

\newcommand{\dist}{\textup{\textsf{d}}}
\newcommand{\diam}{\textup{\textsf{diam}}}
\newcommand{\Cay}{\text{Cay}}

\newcommand{\Aut}{\operatorname{Aut}}

\newcommand{\qs}{\mf{S}/N}
\newcommand{\ol}{\overline}
\newcommand{\Stab}{\operatorname{Stab}}

\newcommand{\Supp}{\mathrm{Supp}}
\newcommand{\link}{\mathrm{Link}}

\newcommand{\tsh}[1]{\left\{\kern-.7ex\left\{#1\right\}\kern-.7ex\right\}}
\newcommand{\Tsh}[2]{\tsh{#2}_{#1}}
\newcommand{\ignore}[2]{\Tsh{#2}{#1}}

\renewcommand{\tilde}{\widetilde}
\DeclareSymbolFont{cyrletters}{OT2}{wncyr}{m}{n}
\DeclareMathSymbol{\Sha}{\mathalpha}{cyrletters}{"58}
\DeclareMathSymbol{\Zhe}{\mathalpha}{cyrletters}{"11}

\newcounter{acomments}

\newcounter{ccomments}

\newcounter{dcomments}

\newcounter{tcomments}

\newcounter{gcomments}

\title{Random quotients preserve acylindrical and hierarchical hyperbolicity}
\author{Carolyn Abbott}\email{carolynabbott@brandeis.edu}\address{Brandeis University, Waltham, Massachusetts, US}
\author{Daniel Berlyne}\email{danberlyne@gmail.com}\address{ThinkTank Maths, Edinburgh, UK}
\author{Giorgio Mangioni}\email{gm2070@hw.ac.uk}\address{Maxwell Institute and Department of Mathematics, Heriot-Watt University, Edinburgh, UK}
\author{Thomas Ng}\email{thomas.ng.math@gmail.com}\address{Brandeis University, Waltham, Massachusetts, US}
\author{Alexander J. Rasmussen}\email{ajrmath@gmail.com}\address{Tatari, US}
\shortauthors{Abbott, Berlyne, Mangioni, Ng, Rasmussen}


\begin{document}

\begin{abstract} 
We propose a new model for random quotients of groups using independent random walks. 
In this model, we show that random quotients of acylindrical hyperbolic groups asymptotically almost surely remain acylindrically hyperbolic. 
Our main tools relate the theories of spinning families and projection complexes to random walks. 

In the presence of a hierarchical hyperbolic structure on the group, we leverage the fine control of projections to show that this structure is preserved in the quotient asymptotically almost surely.  The same techniques yield that random quotients of a non-elementary hyperbolic group (relative to any finite collection of finitely generated peripheral subgroups) are asymptotically almost surely hyperbolic (relative to commensurable peripheral subgroups).

Finally, we also prove that any two groups that are both acylindrically and hierarchically hyperbolic have a common quotients which is itself acylindrically and hierarchically hyperbolic.  This produces ``exotic’’ hierarchically hyperbolic groups with strong fixed point properties, such as Kazhdan’s property (T).
\end{abstract}

\maketitle

\setcounter{tocdepth}{1}
\tableofcontents

\section{Introduction}
Two important classes of  groups with strong negative curvature features are acylindrically hyperbolic groups and hierarchically hyperbolic groups.  Acylindrically hyperbolic groups were introduced by Osin \cite{Osin_acyl_hyp} as a generalization of non-elementary hyperbolic and relatively hyperbolic groups.  This broad class includes non-exceptional mapping class groups, non-virtually cyclic CAT(0) groups that do not split as direct products, $\operatorname{Out}(\mathbb F_n)$ for $n\geq 2$, the Cremona group, and any group that admits a presentation with at least two more generators than relators.  Hierarchically hyperbolic groups (HHG) were introduced by Behrstock, Hagen, and Sisto in \cite{BHS_HHSI} to generalize the subsurface projection machinery of mapping class groups \cite{MMI,MMII} to a wider class of groups, including right-angled Artin and Coxeter groups \cite{BHS_HHSI} and, more generally, most cubulated groups, and most 3-manifold groups \cite{BHS_HHSII,HRSS-HHG_3mfld}.  There is a large overlap between these two classes of groups: any HHG that is not quasi-isometric to a non-trivial product or a line is acylindrically hyperbolic. However, several prominent examples of acylindrically hyperbolic groups, including outer automorphism groups of free groups, are not HHG.

While acylindrical hyperbolicity is known to be fairly common,  evidence has recently been mounting that hierarchical hyperbolicity is also widespread, as it has been shown for a large class of Artin groups, graph products, lattices, and group extensions \cite{BHS_HHS_AsDim,graph_prod,HagenMartinSisto_XL_Artin,Hughes,HughesValiunas,multicurve,central_ext}. 

In this direction, one can ask which quotients of an acylindrically hyperbolic group (resp., HHG) are themselves acylindrically hyperbolic (resp., HHG). Indeed, quotients are a common tool for constructing negatively curved groups. Random groups in both the few relators model and the density model (with small enough density) are hyperbolic with overwhelming probability \cite{gromov2,Olshanskii-random, Ollivier}. Moreover, Delzant showed that quotients of hyperbolic groups by elements with large translation length are again hyperbolic \cite{Delzant_Sous-groupes}.
In fact, Groves and Manning, and independently Osin, generalized Thurston's hyperbolic Dehn filling theorem to show that all peripheral quotients of relatively hyperbolic groups whose kernels avoid a finite set of elements are again relatively hyperbolic. These quotients are also typically hyperbolic; see \cite[Theorem~7.2]{Groves_Manning_Rel_Hyp_Dehn_Filling} and \cite[Theorem~1.1]{Osin-RelHypDehnFilling}.

This paper is a continuations of the above themes. More precisely, we consider quotients by the $n$th steps of finitely many independent random walks $w_{1,n},\ldots,w_{k,n}$ associated to \emph{permissible} probability measures $\mu_1,\ldots,\mu_k$.  We postpone the definition of a permissible probability measure to \Cref{def:permissible}, but the reader should have in mind the case when $\mu$ is supported on a finite, symmetric generating set for $G$. The quotient $G/\llangle  w_{1,n},\ldots,w_{k,n}\rrangle$ is a \emph{random quotient} of $G$.  Given a property $P$, we  say that a random quotient of $G$ has property $P$ \emph{asymptotically almost surely} (a.a.s.) if the probability that $G/\llangle  w_{1,n},\ldots,w_{k,n}\rrangle$ has property $P$ approaches 1 as $n$ tends to infinity. 

\begin{thmi}\label{thm:AHQuotient}
    Let $G$ be a 
    acylindrically hyperbolic group, and let $\mu_1,\dots, \mu_k$ be permissible probability measures on $G$.  A random quotient of $G$ is a.a.s.~acylindrically hyperbolic.
\end{thmi}

Moreover, when $G$ is hierarchically hyperbolic, we obtain a stronger structural result for a random quotient, providing  further evidence that hierarchically hyperbolic groups are common:

\begin{thmi}\label{thm:main_intro}
    Let $G$ be an acylindrically hyperbolic (relative) HHG, and consider $k$ permissible probability measures $\mu_1,\dots, \mu_k$ on $G$. A random quotient of $G$ is a.a.s.~an acylindrically hyperbolic (relative) HHG. 
\end{thmi}

Considering only acylindrically hyperbolic HHGs is necessary, as product HHGs include groups such as the Burger--Mozes group, which is simple and so has no non-trivial quotients. 

\Cref{thm:main_intro} is already new for random quotients of mapping class groups. In this setting, the explicit description of the HHG structure has already been used by the third author to prove that random quotients of mapping class groups are quasi-isometrically rigid, in the sense that if a finitely generated group is quasi-isometric to such a quotient, then it is weakly commensurable to it \cite{mangioni2023rigidityresultslargedisplacement}. Other hierarchically hyperbolic quotients of mapping class groups include quotients by suitable powers of all Dehn twists \cite{BHMS_Comb}; by suitable powers of a pseudo-Anosov element, for surfaces without boundary \cite{BHS_HHS_AsDim} and with one boundary component \cite{central_ext, Tao_ext}; and by deep enough subgroups of certain convex-cocompact subgroups \cite{BHS_HHS_AsDim}. 
\par\medskip

In the case of non-elementary hyperbolic groups, we can deduce stronger results about the random quotient. 
\begin{cori}\label{cor:hyp}
    Let $G$ be a non-elementary hyperbolic group, and let $\mu_1,\dots, \mu_k$ be permissible probability measures on $G$.  A random quotient of $G$ is a.a.s.~non-elementary hyperbolic.
\end{cori}

To the best of our knowledge, this result is not explicitly written down in the literature.  However, it follows quickly from a theorem of Delzant \cite[Th\'{e}or\`{e}me I]{Delzant_Sous-groupes} (originally stated by Gromov \cite[Theorem 5.5.D]{gromov1}), combined with the fact that the translation length of each $w_{i,n}$ is a.a.s.~linear in $n$.

\Cref{thm:main_intro} holds for the larger class of \emph{relative} HHGs, which includes relatively hyperbolic groups. Using this, we also show that a random quotients of a non-elementary relative hyperbolic group is again non-elementary hyperbolic, relative to commensurable peripherals:
\begin{cori}\label{cor:relhyp}
    Let $G$ be a group that is non-elementary hyperbolic relative to a collection $\{H_1,\ldots H_\ell\}$ of finitely generated, infinite subgroups. If $ G/\mc K$ is a random quotient of $G$, then the following hold a.a.s.
    \begin{enumerate}
        \item\label{item:peripherals} For every $i$, $|H_i\cap \mc K|<\infty$. If, moreover, the maximal finite normal subgroup of $G$ coincides with the center, then $H_i\cap \mc K=\{1\}$, so every $H_i$ injects in the quotient.
        \item The quotient $G/\mc K$ is non-elementary hyperbolic relative to $\{\ol H_1,\ldots, \ol H_\ell\}$, where $\ol H_i\coloneq H_i/(H_i\cap \mc K)$.
    \end{enumerate}
\end{cori}

It was recently shown that a version of \Cref{cor:relhyp} holds for random quotients of free products using a density model of randomness \cite[Theorem~1.2]{EMMNS}. Although the notions of randomness are different, both frameworks for random quotients yield relatively hyperbolic structures on the quotient with commensurable peripheral subgroups. In contrast, Dehn filling quotients of relatively hyperbolic groups are obtained by taking the quotient by sufficiently deep subgroups of the peripherals, and therefore never preserve the peripheral structure.

In \cite{shortHHG_I}, the third author introduced \emph{short HHGs}, a simple class of HHGs including extra-large type Artin groups, numerous RAAGs, and non-geometric graph manifolds groups. In the spirit of \Cref{cor:hyp} and \Cref{cor:relhyp}, we expect that one could deduce from \Cref{thm:main_intro} that random quotients of short HHGs are themselves short HHGs a.a.s. As a consequence, such quotients would be fully residually hyperbolic \cite[Corollary H]{shortHHG_II}.

While the main results in this paper are about random quotients, we also use the probabilistic techniques we develop to prove an existence result: any two acylindrically hyperbolic (relative) HHGs have a common quotient.

\begin{thmi}\label{thm:commonquot_intro}
    If $G_1$ and $G_2$ are acylindrically hyperbolic (relative) HHGs, then there exists an acylindrically hyperbolic (relative) HHG $H$ and surjections $G_1\twoheadrightarrow H$ and $G_2\twoheadrightarrow H$.  
\end{thmi}

\noindent Since there exists a non-elementary hyperbolic group with Kazhdan's property (T) by, e.g., \cite{Generalised_triangle}, and since the latter property passes to quotients by, e.g., \cite[Theorem 1.3.4]{(T)}, \Cref{thm:commonquot_intro} readily implies the following:
\begin{cori}
    Every acylindrically hyperbolic (relative) HHG has a quotient that is an acylindrically hyperbolic (relative) HHG with property (T).
\end{cori}

Our methods can be used to prove analogous results with property (T) replaced by any property satisfied by an acylindrically hyperbolic (relative) HHG that passes to quotients.

This corollary generalizes a similar result by Gromov for non-virtually cyclic hyperbolic groups \cite{gromov1} and by Hull for acylindrically hyperbolic groups \cite{H16}.  Combining the proof techniques of \Cref{thm:commonquot_intro} and \Cref{cor:relhyp}, we obtain the following corollary.

\begin{cori}\label{cor:common_rel_hyp_quotient_intro}
 Every non-elementary relatively hyperbolic group has a quotient that is non-elementary relatively hyperbolic with property (T). 
\end{cori}

Moreover, we can find a quotient as in \Cref{cor:common_rel_hyp_quotient_intro} so that every peripheral subgroup in the quotient is commensurable to a peripheral subgroup of the original group.  As in \Cref{cor:relhyp}, if the original relatively hyperbolic group has central maximal finite normal subgroup, then we can find such a quotient with isomorphic peripheral subgroups.

\subsection*{Outline of tools and arguments.}
\subsubsection*{Spinning families from geometrically separated subspaces}
For much of the paper, we work in the following general context. Let $G$ be a group with a non-elementary action on a hyperbolic space $X$. Assume there exist quasiconvex, geometrically separated subspaces $Y_1,\dots, Y_k\subseteq X$ with metrically proper, cobounded actions of subgroups $H_1,\dots, H_k$, and denote by $\hat X$ the \emph{cone-off} of $X$ with respect to the translates of the $Y_i$. These assumptions allow one to build a \textit{projection complex}, as defined by Bestvina, Bromberg, and Fujiwara \cite{BBF:quasitree}, with respect to the translates of $Y_i$; see \Cref{sec:projcplx} for generalities on projection complexes. In fact, we show in \Cref{cor:proj_complex_with_points} that one can build a projection complex with respect to $\hat X^{(0)}$, that is, with respect to the union of points in $X$ and translates of $Y_i$. We also require several additional assumptions to hold, namely \Cref{hyp:metric} and \Cref{hyp:spinning}, which ensure that the subgroups $H_1,\dots, H_k$ form a \emph{``sufficiently'' spinning family} with respect to the action of $G$ on $\hat X$. Spinning families are defined precisely in \Cref{sec:spinningQuotient}, but can be thought of as a geometric/dynamical generalization of satisfying a small cancellation condition.  

\subsubsection*{Acylindrically hyperbolic quotients from spinning families}
Clay and Mangahas \cite{ClayMangahas} and Clay, Mangahas, and Margalit \cite{ClayMangahasMargalit} studied spinning families in actions on projection complexes. In particular, Clay and Mangahas showed (among other things) that if a collection of subgroups $H_i\le G$ form an equivariant spinning family with respect to the action of $G$ on a projection complex $\mc P$, then the quotient $\mc P/\llangle H_1,\dots, H_k\rrangle$ is hyperbolic \cite[Theorem~1.1]{ClayMangahas}. In a similar fashion, in \Cref{sec:spinningQuotient} we prove hyperbolicity of $\hat X/N$ and establish a criterion for acylindrical hyperbolicity of the quotient; see \Cref{thm:hyperbolicity_olX_with_points} and \Cref{cor:preserveAH}, respectively. The core of the proof of \Cref{thm:hyperbolicity_olX_with_points} is to produce a closed lift $T\subseteq \hat X$ of a given geodesic triangle $\ol T\subseteq \hat X/N$. Since $\hat X$ is hyperbolic, the triangle $T$ is uniformly slim, and therefore so is $\ol T$, as the quotient map $\hat X\to\hat X/N$ is $1$--Lipschitz. In turn, in order to find a closed lift of $\ol T$, we start with an open lift of the triangle $\ol T$ and define a ``bending procedure'' that eventually yields a closed lift; see \Cref{lem:bendinggeodesics} and the surrounding discussion. The main ingredient in the above procedure is the existence of a \emph{shortening pair} (\Cref{prop:new_shortening_pair}), which can be thought of as a version of Greendlinger's Lemma for a spinning family \cite{Greendlinger, ClayMangahas}.  An analogous procedure shows that we can also lift quadrangles from $\hat X/N$ to $\hat X$, which we use to prove that acylindrical hyperbolicity is preserved under taking quotients by sufficiently spinning families. Similar ideas of lifting polygons from quotients by collections of subgroups satisfying similar small-cancellation-like conditions appear in \cite{Dahmani,DHS_dfdt, BHMS_Comb,rigidity_MCG/DT,shortHHG_II, ClayMangahas}.  Indeed, our arguments follow similar lines as those of Clay--Mangahas, with necessary adjustments to handle the fact that $\hat X$ itself is not a projection complex. 

\subsubsection*{Quotients of HHGs by spinning families}
In \Cref{sec:hierarchical}, we focus on the case that $G$ is a hierarchically hyperbolic group.  Hierarchically hyperbolic groups are defined precisely in \Cref{defn:rel_HHG}, but we describe a few key aspects here.  Roughly, an HHG structure on a group $G$ consists of a collection of projections from $G$ onto hyperbolic spaces $\{\mc CU\mid U\in \mf S\}$ indexed by a set $\mf S$.  There are three relations on $\mf S$: nesting, transversality, and orthogonality.  Intuitively, the projections onto hyperbolic spaces encode ``hyperbolic pieces'' of $G$, while the relations encode how the various pieces fit together to build the entire group.  The most relevant relation for this outline is nesting, which is a partial order on $\mf S$ with a unique largest element, typically denoted by $S$ and called the \textit{top-level domain}.  The hyperbolic space associated with the top-level domain is called the \emph{top-level space}.

The main technical result of  \Cref{sec:hierarchical} is \Cref{thm:quotientishhg}, which states that the quotient of a HHG by the normal closure of a sufficiently spinning family of subgroups is a HHG.  We present an informal version of this result and a brief description of the HHG structure of the quotient here; see \Cref{constr:HHGStructureQuotient} for further details on the structure.

\begin{thmi}
\label{thm:spinning_quotient}
    Let $(G,\mf S)$ be a (relative) HHG with top-level coordinate space $X$.  Suppose there exist quasiconvex, geometrically separated subspaces $Y_1,\dots, Y_k\subseteq X$ with metrically proper, cobounded actions of subgroups $H_1,\dots, H_k$. Let $\hat X$ be the cone-off of $X$ with respect to the translates of $\{Y_1,\ldots,Y_k\}$. If the subgroups $H_1,\dots, H_k$ form a sufficiently spinning family with respect to $\hat X$, then the quotient of $G$ by $N=\llangle H_1,\dots, H_k\rrangle$ is a (relative) HHG, with the following structure.
    \begin{itemize}
        \item The set of domains is $\mf S/N=\{\ol U\mid U\in \mf S\}$. 
        \item If $\ol U\neq \ol S$, then the associated hyperbolic space is isometric to $\mc CU$ for some (equivalently, any) representative $U\in \mf S$ of $\ol U$. 
        \item The top-level domain is $\hat X/N$, which is quasi-isometric to $X/N$.
        \item Two domains $\ol U, \ol V\in \mf S/N$ are orthogonal or nested if they admit orthogonal or nested representatives, respectively, and are transverse otherwise.
    \end{itemize}
\end{thmi}

The key idea in showing that the above candidate structure satisfies the axioms of an HHG is to define preferred representatives of elements of $G/N$, $\hat X/N$, and $\mf S/N$ in $G, \hat X$, and $\mf S$, respectively. For this, we introduce the notion of \textit{minimal representatives} in \Cref{def:minrep}. Briefly, if $\ol x, \ol y \in \hat X/N$, for example, then representatives $x, y$ of $\ol x, \ol y$, respectively, are minimal if $\dist_{\hat X}(x,y)=\dist_{\hat X/N}(\ol x, \ol y)$. For domains, the rough idea is that $U,V\in \mf S$ are minimal representatives if the distance between the images of $\rho^U_S$ and $\rho^V_S$ in $\hat X$ is minimal among all possible representatives. The relationship (i.e., nesting, transversality, or orthogonality) between domains in $\mf S/N$ is then defined to be the relationship between minimal representatives of those domains in $\mf S$, and we similarly use minimal representatives to define the various projections. To verify each axiom, the general strategy is to use minimal representatives to `lift' the setup of the axiom from $(G/N,\mf S/N)$ to $(G,\mf S)$, where the axiom is satisfied, and then push the result back down to $(G/N,\mf S/N)$. The technicalities involved in ``lifting'' the setup of the axioms are dealt with in \Cref{sec:minimal} and \Cref{sec:HHGStructure}, and the axioms themselves are verified in \Cref{sec:axiomproofs}.

\subsubsection*{Random walks form spinning families}
In \Cref{sec:RW}, we study random walks on an acylindrically hyperbolic group.  We show that subgroups generated by finitely many independent random elements a.a.s.~satisfy \Cref{hyp:metric},  \Cref{hyp:spinning}, and the assumptions of \Cref{cor:preserveAH}, thus proving \Cref{thm:AHQuotient}. Under then additional assumption that the group is an HHG, we  show that such subgroups a.a.s.~satisfy the assumptions of \Cref{thm:quotientishhg}, and this is then used to prove  \Cref{thm:main_intro}, \Cref{cor:hyp}, and \Cref{cor:relhyp} in \Cref{sec:RandomQuotients}.

\subsubsection*{Common quotients of HHGs}
Given two acylindrically hyperbolic HHGs $G_1$ and $G_2$, say generated by finite sets $S_1$ and $S_2$, one can build a common quotient by considering the free product $G_1*G_2$ and identifying every element $s_i\in S_1$ with a word $w_i$ in the alphabet $S_2$, and similarly every element $t_j\in S_2$ with some $v_j\in S_2^*$. To ensure the relations form a spinning family, one can choose $w_i$ and $v_j$ to be the $n$th steps of random walks on $G_1$ and $G_2$, respectively. We do just that in \Cref{sec:commonquot}. Our arguments are similar in spirit to a proof of Hull about the existence of common quotients of acylindrically hyperbolic groups \cite{H16}. However, while the elements $w_i$ and $v_j$ in Hull's paper are chosen to satisfy a suitable small cancellation condition, here we use random walks, as most of the requirements of our machinery are already verified in \Cref{sec:RW}.

\subsubsection*{Quotient by the finite radical}
Let $G$ be an acylindrically hyperbolic HHG, and let $\mc E(G)$ be its maximal finite normal subgroup, whose existence is granted by, e.g., \cite[Theorem 6.14]{DGO}. As explained in \Cref{rem:E_is_trouble}, if $\mc E(G)$ is larger than the center of $G$, then the conjugates of the $n$th step of a random walk $w$ might fail to form a spinning family. To circumvent this problem, we first prove our results under the assumption that $\mc E(G)$ is central in $G$. Then in the general setting we reduce to the above case using the following result from \Cref{sec:appendix}, which is a tool of independent interest in the study of hierarchically hyperbolic groups:
\begin{thmi}
    Let $(G,\mf S)$ be a (relative) HHG, and let $K\unlhd G$ be a finite normal subgroup. Then $G/K$ is a (relative) HHG, and if $G$ is  acylindrically hyperbolic, then so is $G/K$.
\end{thmi}

\subsection*{Comparison with very rotating quotients}
     \Cref{thm:quotientishhg} should be compared to \cite[Theorem~6.1]{BHS_HHS_AsDim}, which implies the same result for a \emph{very rotating family} of subgroups, rather than a sufficiently spinning family. While spinning families and rotating families capture similar behavior, it is not clear that \cite[Theorem~6.1]{BHS_HHS_AsDim} can be used to prove that random quotients of HHGs are HHGs. The key difference in the construction of a HHG structure on the quotient in each case is the relationship between the various constants involved. 

    In \cite{BHS_HHS_AsDim}, the first step is to modify the top-level hyperbolic space $\mc CS$ by gluing on ``hyperbolic cones'' as in \cite{DGO} to obtain a new hyperbolic space $\widehat{\mc CS}$. The collection of random subgroups $\langle w_{i,n}\rangle$ would then have to be an $r$--rotating family with respect to $\widehat{\mc CS}$ for sufficiently large $r$. The problem with using this construction is controlling the growth of $r$ as $n$ tends to infinity. Precisely how large the constant $r$ needs to be is not completely clear, but a careful reading of the proofs in \cite{BHS_HHS_AsDim} shows that it depends at least linearly on the geometric separation constant of the hyperbolic cones. Roughly, to be an $r$--rotating family, the translation length of the random walks must be sufficiently large with respect to $r$, and, in fact, exponential in $r$ \cite[Theorem~6.35]{DGO}. Work of Maher--Tiozzo \cite{MaherTiozzo18, MaherTiozzoCremona} and Maher--Sisto \cite{MaherSisto} show that the geometric separation constant of the random walk grows linearly in $n$, as discussed in \Cref{sec:RWBackground}, and so the translation length of the random walk would need to grow \textit{exponentially} in $n$ to be able to use this construction. However, the translation length grows only linearly in $n$ \cite{MaherTiozzo18}.

    If one could improve the geometric separation constants of the random walk to grow logarithmically in $n$, it might be possible to use \cite[Theorem~6.1]{BHS_HHS_AsDim} to obtain the results in this paper, though one would still need a better understanding of the precise relationship between $r$ and the geometric separation constant. Instead, in this paper we use the more straightforward cone-off procedure described above and replace rotating families with spinning families. With this construction, we show that the collection of random subgroups $\langle w_{i,n}\rangle$ needs to to be an $L$--spinning family for a constant $L$ that is again linear in the geometric separation constant of the random walks; see \Cref{rem:l_hyp}. To form an $L$--spinning family, however, we only need the translation length to be linear in $L$, which holds by \Cref{thm:RandomSubgroupIsSpinning}.

    Since the precise constants are critical in this paper, we have included a summary of their definitions and relative dependencies in \Cref{tableofconstants}. We suggest keeping it handy while reading through the paper.

\begin{acknowledgements}
We thank Alex Sisto for helpful conversations, and particularly for pointing us toward spinning families. We thank Yair Minsky for asking the question about random walks on mapping class groups which was the genesis of this project. The third author is grateful to the other authors for the opportunity to join this project.
We also thank Inhyeok Choi for comments on an early draft and encouragement to include \Cref{thm:AHQuotient}.

The first author is partially supported by NSF grants DMS-2106906 and DMS-2340341. 
The second author was supported by the Additional Funding Programme for Mathematical Sciences, delivered by EPSRC (EP/V521917/1) and the Heilbronn Institute for Mathematical Research. 
The fourth author was partially supported by ISF grant 660/20, 
and by a Zuckerman Fellowship at the Technion. The fifth author was partially supported by NSF grants DMS-2202986 and DMS-1840190.
\end{acknowledgements}


\section{Background on hyperbolicity and projection complexes} \label{sec:background}
\subsection{Hyperbolic spaces}\label{sec:HypBackground}
In this section we review some basic properties of metric spaces, including hyperbolicity. More details can be found in \cite{BH}. Most lemmas in this section are standard, but we provide proofs for completeness and to make explicit all constants, as they will play an important role later in the paper.

In what follows we consider subsets of a fixed space $X$ with metric $\dist = \dist_X$. Every metric space (possibly equipped with an action of a group $G$ by isometries) is ($G$-equivariantly) quasi-isometric to a simplicial graph, by e.g. \cite[Lemma 3.B.6]{cornulier2014metric}, so we can and will assume that all metric spaces we consider are simplicial graphs.

Given a subspace $A\subseteq X$, the closed $R$-neighborhood of $A$ in $X$ is denoted by $\mc N^X_R(A)$, or simply $\mc N_R(A)$ when the ambient space is understood.  Let $\lambda\geq 1$ and $c\geq 0$.  A map $f\colon (X,\dist_X)\to (Y,\dist_Y)$ of metric spaces is a \emph{$(\lambda,c)$--quasi-isometric embedding} if for all $x,y\in X$, we have that
$$\lambda^{-1}\dist_X(x,y)-c\le \dist_Y(f(x),f(y))\le \lambda \dist_X(x,y)+c.$$ 
A \emph{quasi-isometry} is a quasi-isometric embedding which is also \emph{coarsely surjective}, meaning that $Y\subseteq \mc N_R(f(X))$ for some constant $R\ge 0$. 

A \emph{$(\lambda,c)$--quasigeodesic} in $X$ is a $(\lambda,c)$--quasi-isometric embedding of an interval into $X$. When the constants $\lambda$ and $c$ are the same, we simply call such path a $\lambda$-quasigeodesic. A \emph{geodesic} is a $(1,0)$-quasigeodesic, that is, an isometric embedding of an interval. Given points $x,y\in X$, we denote a geodesic from $x$ to $y$ by $[x,y]^X$; if the space $X$ is clear from context we simply write $[x,y]$.

A metric space is \emph{geodesic} (resp. $(\lambda, c)$-quasigeodesic) if any two points are connected by a geodesic (resp. $(\lambda, c)$-quasigeodesic). For $\delta\geq 0$, a geodesic metric space $X$ is \emph{$\delta$--hyperbolic} if, for every three points $x,y,z\in X$, we have $[x,y] \subseteq \mc{N}_\delta( [x,z] \cup [z,y])$; we say that geodesic triangles in a $\delta$-hyperbolic space are $\delta$-\emph{slim}. If the particular choice of $\delta$ is not important, we simply say that $X$ is \emph{hyperbolic}. The \emph{boundary} $\partial X$ of a hyperbolic space is the set of quasigeodesic rays $[0,\infty)\to X$ up to bounded Hausdorff distance (see e.g. \cite[III.H.3]{BH}).

In this paper, all quasigeodesics $\gamma\colon I\to X$ will be continuous. In a hyperbolic space, this is no loss of generality by \cite[Lemma~III.H.1.11]{BH}. We denote the \emph{length} of the quasigeodesic $\gamma$ in $X$ by $\ell_X(\gamma)$.

Hyperbolic spaces satisfy the following \emph{Morse} property, which states that quasigeodesics with the same endpoints remain in a uniform neighborhood of each other (see, e.g., \cite[III.H.1.7]{BH}). This is also known as \emph{quasigeodesic stability}.
  \begin{lem}\label{lem:stability}
    For all $\delta,c\ge 0$, $\lambda\ge 1$, there is a constant $\Phi=\Phi(\lambda, c,\delta)\ge 0$ satisfying the following.  Let $X$ be a $\delta$--hyperbolic space, and let $\gamma_1$, $\gamma_2$ be $(\lambda,c)$--quasigeodesics with the same endpoints in $X\cup \partial X$.  Then the Hausdorff distance between $\gamma_1$ and $\gamma_2$ is at most $\Phi$.
  \end{lem}

\noindent A subspace $Y$ of a hyperbolic space $X$ is \emph{$K$-quasiconvex} if $[x,y]\subset \mc N_K(Y)$ whenever $x,y\in Y$. 

\begin{lem}\label{lem:n_A_unif_qc}
    Let $X$ be $\delta$-hyperbolic, and let $Z\subseteq X$ be a subset of diameter $E$. For every $A\ge 0$, the neighborhood $\mc N_A(Z)$ is $(2\delta+E)$-quasiconvex.
\end{lem}
\begin{proof}
    Let $x,y\in \mc N_A(Z)$, and let $x',y'\in Z$ be such that $\dist(x,x')\le A$ and $\dist(y,y')\le A$. From the definition of hyperbolicity, it readily follows that geodesic quadrangles in $X$ are $2\delta$-slim. Hence any geodesic $[x,y]$ is in the $2\delta$-neighborhood of geodesics $[x,x']\cup [x',y']\cup [y',y]$. Since $\diam (Z)\le E$, $[x,y]$ is actually in the $(2\delta+E)$-neighborhood of $[x,x']\cup [y',y]$, and these two geodesics lie in $\mc N_A(Z)$.
\end{proof}

From now on, we assume that $X$ is a $\delta$-hyperbolic \emph{graph}, and every point in $X$ is thought of as a vertex, so that distances between points are integer-valued. If $Y\subseteq X$ is quasiconvex, then for every $z\in X$ one can define a coarse closest point projection $\pi_Y\colon X\to Y$ by mapping every $z\in X$ to the collection
\begin{equation}\label{eqn:pi_Y} \pi_Y(z)=\{y\in Y\mid \dist(z,Y)=\dist(z,y)\}.\end{equation}
For any such $Y \subseteq X$ and every $A,B\subseteq X$ we set $\dist^\pi_Y(A,B)=\diam(\pi_Y(A)\cup \pi_Y(B))$.

\begin{lem}[Closest point projections are uniformly Lipschitz]\label{lem:lipschitzproj}
     Let $X$ be a $\delta$--hyperbolic graph, $Y\subseteq X$ a $K$--quasiconvex subspace, and $\pi_Y\colon X\to Y$ a coarse closest point projection. Let $x,y\in X$ be such that $\dist(x,y)\le 1$. Then $\dist^\pi_Y(x, y)\le J\coloneq 2K+10\delta+2$. In particular, for any $w,z\in X$, $\dist^\pi_Y(w,z)\le J\max\{\dist_X(w,z),1\}.$
\end{lem}

\begin{proof}
Let $x'\in \pi_Y(x)$ and $y'\in \pi_Y(y)$, and consider a geodesic quadrangle with vertices $\{x,x', y',y\}$. Towards a contradiction, if  $\dist(x', y')>J$, we can find a point $z\in [x',y']$ such that both $\dist(x', z)$ and $ \dist(y',z)$ are at least $K+5\delta$. Let $z'\in Y$ be such that $\dist(z,z')\le K$, which exists as $Y$ is $K$--quasiconvex. Since quadrangles in $X$ are $2\delta$--slim, there exists $w\in [x',x]\cup[x,y]\cup[y,y']$ within distance $2\delta$ from $z$. If $w\in [x',x]$ then 
$$\dist(x,x')=\dist(x, w)+\dist(w, x')\ge \dist(x, w)+\dist(z, x')-\dist(w, z)\ge \dist(x, w)+ K+3\delta,$$ 
while 
$$\dist(x, z')\le \dist(x, w)+\dist(w, z')\le \dist(x, w)+K+2\delta.$$ 
This would contradict the definition of $x'$ as a closest point to $x$ in $Y$. For the same reason, $w$ cannot lie on $[y, y']$, and so $w\in [x,y] - \{x,y\}$. However, this is impossible, as $\dist(x,y)\le 1$. 
\end{proof}

\begin{lem}[{\cite[Proposition 10.2.1]{CDP}}]\label{lem:NppGivesQgeo}
Given $\delta, K \geq 0$, there is a constant $\Omega = \Omega(\delta,K)$ such that the following holds. Let $Y$ be a $K$-quasiconvex subset of a $\delta$-hyperbolic graph $X$. For any pair of points $x,y\in X$, any $x'\in \pi_Y(x)$, and any $y'\in \pi_Y(y)$, if $\dist(x',y')\ge \Omega$, then the \emph{nearest point path} $[x, x'] \cup [x', y'] \cup [y', y]$ is a $(1,\Omega)$-quasigeodesic.
\end{lem}

A collection $\mc Y$  of quasiconvex subspaces of $X$ is \emph{geometrically separated} if for every $\varepsilon>0$, there exists $R>0$ such that   $\diam( \mc N_\varepsilon(Y')\cap Y)\le R$ whenever $Y\neq Y'\in \mc Y$.

\begin{lem}\label{lem:geomsep_for_qc}
Let $\mc Y$ be a collection of $K$-quasiconvex subspaces of the $\delta$-hyperbolic graph $X$. Suppose there exists $M_0>0$ such that $\diam(\mc N_{2K+2\delta}(Y')\cap Y)\le M_0$ whenever $Y,Y'\in\mc Y$ are distinct. Then $\mc Y$ is geometrically separated. More precisely, for every $t\geq 0$, 
\begin{equation}
    \label{eq:M(e)}
\diam(\mc N_{t}(Y')\cap Y)\le M(t)\coloneq M_0+2K+2t+4\delta+2.
\end{equation}
\end{lem}

\noindent If the assumption of \Cref{lem:geomsep_for_qc} holds, we  say that the collection $\mathcal Y$ is \emph{$M_0$-geometrically separated}.

\begin{proof}[Proof of \Cref{lem:geomsep_for_qc}]
    Let $x,y\in \mc N_{t}(Y') \cap Y$, and let $x', y'\in Y'$ realize the distance between $Y'$ and $x,y$, respectively. If $\dist(x,y)\le 2t+4\delta+1$ there is nothing to prove. Otherwise, consider a geodesic quadrangle $[x,y]\cup[y,y']\cup[y',x']\cup[x',x]$, and let $a\in [x,y]$ be such that $t +2\delta\le\dist(x,a)\le t +2\delta+1$. Notice that $\dist(a,y)\ge \dist(x,y)-\dist(x,a)\ge t +2\delta$ as well. Since this geodesic quadrangle is $2\delta$-slim, there exists $a'\in [y,y']\cup[y',x']\cup[x',x]$ such that $\dist(a,a')\le 2\delta$. Notice that $a'$ cannot lie in the interior of $[x,x']$, as otherwise the reverse triangle inequality would give that $\dist(x,x')> \dist(a',x)\ge \dist(x,a)-\dist(a,a')\ge t$. For the same reason, $a'$ does not lie in the interior of $[y,y']$, so we must have that $a'\in [x,y]\subseteq \mc N_K(Y')$. Hence $a\in \mc N_{K+2\delta}(Y')$. In turn, as $a$ lies on a geodesic $[x,y]$ with endpoints in $Y$, there is some $p\in Y$ such that $\dist(a,p)\le K$. Thus we have found a point $p\in Y\cap \mc N_{2K+2\delta} (Y')$ such that $\dist(x,p)\le \dist(x,a)+\dist(a,p)\le t +2\delta+K+1$. An analogous argument produces a point $q\in Y\cap \mc N_{2K+2\delta} (Y')$ such that $\dist(y,q)\le t +2\delta+K+1$. Since $p,q\in \mc N_{2K+2\delta}(Y')\cap Y$, which has diameter at most $M_0$, we have
    \[\dist(x,y)\le \dist(x,p)+\dist(p,q)+\dist(q,y)\le M_0+2K+2t+4\delta+2.\qedhere\]
\end{proof}

From now on, we shall say that a function $f$, depending on some constants $c_1,\ldots, c_n$, and $M_0$, is \emph{bounded linearly} in $M_0$ if there are positive functions $a(c_1,\ldots, c_n)$ and $b(c_1,\ldots, c_n)$ such that $|f(M_0,c_1,\ldots, c_n)|\le a(c_1,\ldots, c_n)M_0+b(c_1,\ldots, c_n)$.

\begin{lem}\label{lem:boundedproj_for_qc} 
Let $X$ be a $\delta$-hyperbolic graph, and let $\mc Y$ be an $M_0$-geometrically separated collection of $K$-quasiconvex subspaces. There exists $B=B(\delta, K,M_0)$ which is bounded linearly in $M_0$ and  such that $\diam \pi_{Y'}(Y) \le B$ for every $Y\neq Y'\in \mc Y$.
\end{lem}

\begin{proof} We claim that it suffices to take $B= M(2K+7\delta+1)$, where $M$ is the geometric separation function from \eqref{eq:M(e)}, which  is bounded linearly in $M_0$. Let $x,y\in Y$, and let $x'\in \pi_{Y'}(x)$ and $y'\in \pi_{Y'}(y)$. If $\dist(x', y')\le 2K+10\delta+2$ we are done, as the latter is less than $M(2K+7\delta+1)$. Otherwise consider a geodesic quadrangle with vertices $\{x,x',y',y\}$. Let $z'\in [x',y']$ be such that $K+5\delta\le \dist(x',z')\le K+5\delta+1$. By slimness of quadrangles, one can find a point $w\in [x',x]\cup [x,y]\cup[y,y']$ within distance $2\delta$ from $z$. Arguing exactly as in \Cref{lem:lipschitzproj}, one sees that $w$ must belong to $[x,y]$, or it would violate the fact that $x'$ (resp. $y'$) realizes the distance between $Y'$ and $x$ (resp. $y$). Therefore
    \[\dist(x', Y)\le \dist(x', z)+\dist(z,w)+\dist(w,Y)\le 2K+7\delta+1,\]
    where we used that $Y$ is $K$-quasiconvex to bound $\dist(w,Y)$. The same argument works for $y'$, so the two points lie in $Y'\cap \mc N_{2K+7\delta+1}(Y)$. By \Cref{lem:geomsep_for_qc}, the diameter of the latter is bounded by $M(2K+7\delta+1)$, as required.
\end{proof}

\subsection{Cone-offs of graphs}
Let $\mc{Y}$ be a collection of subgraphs of a connected hyperbolic graph $X$. We denote by $\hat{X}$ the graph obtained from $X$ as follows: add a vertex $v_Y$ for each $Y \in \mc Y$, and add edges from $v_Y$ to each vertex in $Y$. Note that $X$ is naturally a subspace of $\hat X$.  We say that $\hat X$ is the \textit{cone-off}, or \emph{electrification}, of $X$ with respect to $\mc Y$, and we call the vertices $v_Y$ \textit{cone vertices}.

In what follows, the endpoints of a path $\gamma\colon[0,1]\to X$ are denoted by $\gamma_-=\gamma(0)$ and $\gamma_+=\gamma(1)$, and $\alpha * \beta$ denotes the concatenation of two paths $\alpha$ and $\beta$ such that $\alpha_+=\beta_-$.

\begin{defn}\label{defn:coneoff_for_graph}
    Let $\hat X$ be the cone-off of $X$ with respect to a family of subgraphs $\mc Y$.  Let $\gamma=u_1 * e_1 * \cdots * e_n * u_{n+1}$ be a concatenation of geodesic segments where each $e_i$ is a concatenation of two edges sharing a common cone vertex and the $u_i$ are (possibly trivial) segments contained in $X$. A \textit{de-electrification} $\widetilde  \gamma$ of $\gamma$ is the concatenation $u_1 * \eta_1 * \cdots * \eta_n * u_{n+1}$, where each $\eta_i$ is a geodesic in $X$ connecting the endpoints of $e_i$. If $e_i$ connects points of $Y\in \mc Y$, then $\eta_i$ is a \emph{$Y$-component of $\widetilde  \gamma$.}
\end{defn}
\noindent We are particularly interested in the case that $\mc Y$ is a collection of uniformly quasiconvex subgraphs.  
\begin{lem}[{\cite[Proposition~2.6]{KapovichRafi}}]\label{lem:hyp_constant_cone_off}
    Let $X$ be a $\delta$-hyperbolic graph, and let $\mc Y$ be a collection of $K$-quasiconvex subspaces. There exists $\hat\delta=\hat\delta(\delta, K)$ such that $\hat X$ is $\hat\delta$-hyperbolic.
\end{lem}
\noindent Moreover, de-electrifications of geodesics in $\hat X$ are uniformly close to geodesics in $X$.

\begin{lem}[{\cite[Corollary~2.23]{Spriano_hyperbolic_I}}]\label{lem:DboundSpriano}
    Let $X$ be a $\delta$-hyperbolic graph, let $\mc Y$ be a collection of $K$-quasiconvex subsets of $X$, and let $\hat X$ be the cone-off of $X$ with respect to $\mc Y$.
    Then there exists a constant $D = D(\delta, K)$ such that, for any pair of points $x,y \in X$, every geodesic $[x,y]^X$, and every geodesic $\gamma$ from $x$ to $y$ in $\hat X$, we have $[x,y]\subseteq \mc N^X_D(\widetilde \gamma)$, where $\widetilde {\gamma}$ is a de-electrification of $\gamma$.
\end{lem}
\begin{remark}
    \cite[Corollary~2.23]{Spriano_hyperbolic_I} is stated for \emph{connected} subgraphs; however, the latter hypothesis is only used to ensure that de-electrifications of geodesics exist, which is true under the requirement that $X$ is hyperbolic (in particular geodesic, hence path-connected). Moreover, the cone-off $X'$ of $X$ used in \cite{Spriano_hyperbolic_I} is slightly different: there, an edge is added between any two vertices lying in a common $Y\in \mc Y$.  There is a map from $X'$ to $\hat X$ sending an edge to a concatenation of two edges between the same pair of points, and it is immediate that de-electrifications $\widetilde  \gamma$ of a geodesic in $X'$ agree with de-electrifications of the image of $\gamma$ in $\hat X$.  Hence \Cref{lem:DboundSpriano} still holds for $\hat X$. 
\end{remark}

\begin{cor}\label{lem:CloseInXIsCloseInHatX}
    In the setting of \Cref{lem:DboundSpriano}, let $x,y,w\in X$, and let $\gamma$  be an $\hat X$--geodesic from $x$ to $y$.  For any $t\geq 0$, if $\dist_X(w,[x,y]^X)\le t$, then $\dist_{\hat X}(w,\gamma)\le t+D+K+1$.
\end{cor}

\begin{proof}
    Let $\widetilde  \gamma$ be the de-electrification of $\gamma$.  Let $w'\in [x,y]^X$ satisfy $\dist_X(w,w')\le t$.  By \Cref{lem:DboundSpriano}, there is a point $z\in \widetilde  \gamma$ with $\dist_X(w',z)\le D$.  Either $z$ lies on $\gamma$, in which case we have proven the bound, or $z$ lies on a $Y$--component of $\widetilde \gamma$ for some $Y\in \mc Y$.  In the latter case, the $K$--quasiconvexity of $Y$ implies that there exists $z'\in Y$ such that $\dist_X(z,z')\le K$, and there is an edge from $z'$ to $v_Y$ in $\hat X$.  The cone vertex $v_Y$ lies on $\gamma$ by construction, and so $\dist_{\hat X}(w,\gamma)\le t + D + K + 1$, as desired.
\end{proof}

\begin{defn}\label{defn:extended_proj}
    For every $Y\in \mc Y$ there is a set-valued projection $\hat X-\{v_Y\}\to 2^Y$, which we still denote by $\pi_Y$, defined as follows. For every $x\in X$ the projection is $\pi_Y(x)$, and for every $U\in \mc Y$ other than $Y$ we set $\pi_Y(v_U)\coloneq \pi_Y(U)$, where $\pi_Y(U)$ is as defined in \eqref{eqn:pi_Y}, considering $U$ as a subspace of $X$.  For every $x,y\in \hat X-\{v_Y\}$, we set $\dist_Y^\pi(x,y)=\diam(\pi_Y(x)\cup \pi_Y(y)).$
\end{defn}

\noindent The following strengthening of the bounded geodesic image property describes how cone points relate to geodesics joining a pair of points. 

\begin{lem}[Strong bounded geodesic image]\label{lem:strongBGI}
    In the setting of \Cref{lem:DboundSpriano}, suppose further that the family $\mc Y$ is $M_0$-geometrically separated, in the sense of \Cref{lem:geomsep_for_qc}. There exists a constant $C=C(\delta, K, M_0)$  bounded linearly in $M_0$  such that for every $Y\in \mc Y$ and $x,y\in \hat X-\{v_Y\}$, if a $\hat{X}$--geodesic $\gamma$ does not pass through $v_Y$ then $\dist^\pi_Y(x,y) \le C$.
\end{lem}
\begin{proof} Assume first that $x,y\in X$. Set $\mc R=2\delta+2K+D+2$, where $D=D(\delta,K)$ is the constant from \Cref{lem:DboundSpriano}. Our first goal is to prove that if $\dist^\pi_Y(x,y) > J$, where $J=J(\delta, K)$ is the constant from \Cref{lem:lipschitzproj}, then $\gamma$ intersects the $\mc R$-neighborhood of $v_Y$ in $\hat X$. To that end, let $[x,y]$ be an $X$-geodesic between $x$ and $y$. Since $\dist^\pi_Y(x,y) > J$, we can argue as in \Cref{lem:lipschitzproj} to find some $z'\in Y$ such that $\dist_X(z',[x,y])\le K+2\delta$.  \Cref{lem:CloseInXIsCloseInHatX} then gives 
\[\dist_{\hat X}(v_Y,\gamma)\le 1+\dist_{\hat X}(z',\gamma)\le 1+\dist_X(z',[x,y])+D+K+1\le 2\delta+2K+D+2 = \mc R.\]
Next, let $a,b$ be the first and last vertices of $\gamma\cap X$ contained in the $(\mc R+2)$-neighborhood of $v_Y$. Notice that $\dist_{\hat X}(a,v_Y)\ge \mc R+1$, and similarly for $b$, since no two points on $\gamma - X$ are adjacent. Hence by the above argument both $\dist^\pi_Y(x,a)$ and $\dist^\pi_Y(b,y)$ are at most $J$. Moreover, if we let $a=a_0,a_1,\ldots,a_n=b$ be the subsegment of $\gamma$ between $a$ and $b$, then $n\le \dist_{\hat X}(a,b)\le 2\mc R+4$. Finally, for every $0\le i\le n-1$, we have $\dist^\pi_Y(a_{i},a_{i+1})\le \max\{J,B\}$, where $B=B(\delta,K,M)$ is the constant from \Cref{lem:boundedproj_for_qc}. Indeed, either $a_{i},a_{i+1}$ are $X$-adjacent, and therefore $\dist^\pi_Y(a_{i},a_{i+1})\le J$, or $a_i=v_U$ for some $U\neq Y$ and $a_{i+1}\in U$ (or vice versa), so $\diam(\pi_Y(a_{i+1})\cup \pi_Y(a_{i}))\le \diam(\pi_Y(U))\le B$. The triangle inequality thus yields
\[\dist^\pi_Y(x,y)\le \dist^\pi_Y(x,a)+\sum_{i=0}^{n-1} \dist_Y^\pi(a_i,a_{i+1})+\dist^\pi_Y(b,y)< C_0\coloneq 2J+(2\mc R+4)\max\{B,J\},\]
concluding the proof of \Cref{lem:strongBGI} for points in $X$. 

Finally, let $x,y\in \hat X$ be any two points, let $\gamma$ be an $\hat X$-geodesic between them which does not pass through $v_Y$, and let $x'$ and $y'$ be the first and last points of $\gamma\cap X$, respectively. In particular, either $x=x'$ or $x=v_U$ for some $U\in \mc Y$ containing $x'$, and similarly for $y$. Note that $\dist_Y^\pi(x',x),\, \dist_Y^\pi(y,y')\le B$, and $\dist_Y^\pi(x',y')\le C_0$ by the above argument, so the triangle inequality yields that $\dist_Y^\pi(x,y)\le C\coloneq C_0+2B$. The constant $C$ is bounded linearly in $M_0$, as so are $C_0$ and $B$.
\end{proof}

\subsection{Projection complexes}
\label{sec:projcplx}

In this section, we recall the machinery of projection complexes, first  introduced by Bestvina, Bromberg, and Fujiwara in \cite{BBF:quasitree}. Let $\mc Y$ be a collection of metric spaces. Suppose that for each $Y\in \mc Y$ there is a projection $\pi_Y$ from the elements of $\mc Y - \{Y\}$ to subsets of $Y$. Then, as in \Cref{defn:extended_proj}, we may define a ``distance function'' $\dist_Y^\pi\colon(\mc Y - \{Y\})^2 \to [0,\infty)$ by
\[
\dist_Y^\pi(U,V) = \diam\left(\pi_Y(U) \cup \pi_Y(V)\right).
\]
The map  $\dist_Y^\pi$ is usually not a true distance function since we may have $\dist_Y^\pi(U,U)>0$.

\begin{defn}[(Strong) projection axioms]\label{defn:Projection_axioms}
Given a collection $\mc Y$ and distance functions $ \{\dist_Y^\pi\}_{Y\in \mc Y}$ as above, as well as a constant $\theta\ge 0$, the \emph{projection axioms for $\mc Y$ with constant $\theta$} are
\begin{enumerate}[label=(\Roman*{})]
\item\label{projaxiom:symmetry} $\dist_Y^\pi(U,V) = \dist_Y^\pi(V,U)$;
\item\label{projaxiom:triangle} $\dist_Y^\pi(U,W) \le \dist_Y^\pi(U,V) + \dist_Y^\pi(V,W)$.
\item\label{projaxiom:boundedprojection} $\dist_Y^\pi(U,U)\le \theta$;
\item\label{projaxiom:bgi} $\min\{\dist_Y^\pi(U,V),\dist_V^\pi(U,Y)\}\le \theta$; and
\item\label{projaxiom:finitelymanybig} for all $U,V\in \mc Y$, the set $\{Y\in \mc Y \mid \dist^\pi_Y(U,V)\geq \theta\}$ is finite.
\end{enumerate}
The \emph{strong projection axioms for $\mc Y$ with constant $\theta$} were defined in \cite{BBFS}, and are obtained by replacing \Cref{projaxiom:bgi} by the following stronger statement:
\begin{enumerate}[label=(\Roman*{}'), start=4]
\item If $\dist_Y^\pi(U,V)>\theta$ then $\dist_V^\pi(U,W)=\dist_V^\pi(Y,W)$ for all $W\in \mc Y - \{V\}$.
\end{enumerate}
\end{defn}

\noindent Bestvina--Bromberg--Fujiwara--Sisto show that  functions $\dist_Y^\pi$ satisfying the projection axioms can be modified to satisfy the strong projection axioms:
\begin{thm}[{\cite[Theorem~4.1]{BBFS}}]\label{thm:modified_dist}
Assume that $\mc Y$ is a collection of metric spaces together with  functions $\{\dist_Y^\pi\}_{Y\in \mc Y}$ satisfying the projection axioms with constant $\theta$. Then there are functions $\dist_Y\colon (\mc Y - \{Y\})^2 \to [0,\infty)$ satisfying the strong projection axioms with constant $11\theta$, and such that for all $Y\in \mc Y$, $\dist_Y^\pi - 2\theta \le \dist_Y \le \dist_Y^\pi + 2\theta$.
\end{thm}

\begin{defn}\label{def:projcomplex}
    Suppose that $(\mc Y, \{\dist_Y^\pi\}_{Y\in \mc Y})$ satisfy the projection axioms with constant $\theta$, and let $\{\dist_Y\}_{Y\in \mc Y}$ be the modified distance functions. For any $\Zhe\geq 0$, the \emph{projection complex} $\mc P_\Zhe(\mc Y)$ is defined as follows. The vertices of $\mc P_\Zhe(\mc Y)$ are $\mc Y$, and two vertices $U,V\in \mc Y$ are joined by an edge if $\dist_Y(U,V)\le\Zhe$ for all $Y\in \mc Y-\{U,V\}$.
    When the collection $\mc Y$ is unimportant or clear from context we use the notation $\mc P_\Zhe$.
\end{defn}
\noindent We recall some facts about projection complexes.

\begin{lem}[Bounded path image, {\cite[Corollary 3.4]{BBFS}}]\label{lem:boundedpathimage}
    If $\Zhe\geq 33\theta$ and a path $U_1,\ldots,U_k$ in $\mc P_\Zhe$ does not intersect the $2$-neighborhood of a vertex $Y$, then $\dist_Y(U_1,U_k)\le 11\theta$.
\end{lem}

\begin{cor}[Strong bounded geodesic image]\label{cor:bgi_for_proj}
    If $\Zhe\geq 33\theta$ and a geodesic $U_1,\ldots,U_k$ in $\mc P_\Zhe$ does not contain a vertex $Y$, then $\dist_Y(U_1,U_k)\le 22\theta +6\Zhe$. 
\end{cor}

\begin{proof}
This is proven exactly as \cite[Lemma 3.18]{BBF:quasitree}, using \Cref{lem:boundedpathimage} to obtain the precise bound.
\end{proof}

\subsubsection{A projection complex from a separated family of quasiconvex subspaces}
Quasiconvex subsets of hyperbolic spaces naturally give rise to projection complexes; see, for example, \cite{BBF:quasitree}. For the rest of the section, we will work under the following hypothesis.  

\begin{hyp}
    \label{hyp:metric}
    Let $X$ be a connected graph and $\mc Y$ a collection of subsets of $X$. Assume there are constants $\delta, K, M_0>0$ such that the following hold.
    \begin{enumerate}
        \item 
    $X$ is $\delta$-hyperbolic.
    \item
    Each $Y \in \mc Y$ is $K$--quasiconvex in $X$.
        \item
    $\mc Y$ is $M_0$--geometrically separated, in the sense of \Cref{lem:geomsep_for_qc}.
    \end{enumerate}
    When the above are satisfied, we say that $(X,\mc Y)$ satisfies \Cref{hyp:metric} with respect to  $(\delta, K, M_0)$. We  omit the constants when they are unimportant or clear from  context. 
\end{hyp}

\noindent The goal of this subsection is to prove the following:
\begin{prop}
    \label{P:hyp-projcplx}
Suppose $(X,\mc Y)$  satisfies \Cref{hyp:metric} with respect to $(\delta, K, M_0)$.  
There exists $\theta=\theta(\delta, K,M_0)$ such that $(\mc Y, \{\dist^\pi_Y\}_{Y\in \mc Y})$ satisfies the projection complex axioms (\Cref{defn:Projection_axioms}) with constant $\theta$. Moreover, $\theta$ is bounded linearly in $M_0$. 
\end{prop}

\begin{proof}
    We determine lower bounds on $\theta$ one item at a time.
    \Cref{projaxiom:symmetry} and \Cref{projaxiom:triangle} follow immediately from the fact that $\dist_Y^\pi$ is defined in terms of diameters of projections, and \Cref{projaxiom:boundedprojection} holds with $\theta\ge B$ by \Cref{lem:boundedproj_for_qc}.

    We now move to \Cref{projaxiom:bgi}. Let $\theta_1=2B+J(2\delta+K+1)$, where $J$ is the Lipschitz constant from \Cref{lem:lipschitzproj}. Let $U,V,W\in \mc Y$ be such that $\dist_W^\pi(U,V)> \theta_1$. 
    Fix $a\in U$, and let $b\in\pi_V(a)$, $c_a\in\pi_W(a)$, and $c_b\in\pi_W(b)$.   Then $\dist(c_a,c_b)\ge \dist_W^\pi(U,V)-\diam\pi_W(U)-\diam\pi_W(V)\ge J$. As in \Cref{lem:lipschitzproj}, there thus exists $w\in[a,b]$ within distance at most $2\delta+K$ from $W$. In turn $\dist(\pi_V(w),\pi_V(W))\le J(2\delta+K)$ as projections are $J$-Lipschitz. Notice that $b\in \pi_V(w)$, as $w$ lies on $[a,b]$ and $b\in \pi_V(a)$, so 
\[\dist(\pi_V(U),\pi_V(W))\le\dist(b,\pi_V(W))\le\diam(\pi_V(w))+\dist(\pi_V(w),\pi_V(W)) \le J(2\delta+K+1).\]
Hence
    \[\dist_V^\pi(U,W)\le \diam\pi_V(U)+\dist(\pi_V(U),\pi_V(W))+\diam\pi_V(W)\le 2 B+J(2\delta+K+1)=: \theta_1,\]
    so \Cref{projaxiom:bgi} holds if $\theta\ge \theta_1$.

    Now let $D_0=2K+4\delta+M(2K+4\delta)+1$, which again is bounded linearly in $M_0$, and let $\theta=3JD_0+2B+2J(3\delta+1)$, which is greater than $\theta_1$. We are left to prove \Cref{projaxiom:finitelymanybig}, i.e., that for every $U\neq V\in \mc Y$ the set $\{Y\in \mc Y\mid \dist_Y^\pi(U,V)\ge \theta\}$ is finite.

    Let $\gamma$ be a geodesic from $u$ to $v$, where $u\in U$ and $v\in V$. For every $Y$ as above, $\dist^\pi_Y(u, v)\ge \theta-2B\ge 2J(3\delta+1)$, so let $a,b\in \gamma$ be the last point such that $\dist^\pi_Y(u, a)\le J(3\delta+1)$ and the first point such that $\dist^\pi_Y(b, v)\le J(3\delta+1)$, respectively.  See \Cref{fig:projcomplex}. Let $[a,b]$ be the subsegment of $\gamma$ between $a$ and $b$, and let $u'\in \pi_Y(u)$ and $v'\in \pi_Y(v)$. By slimness of quadrangles, every $w\in [a,b]$ is $2\delta$-close to some point $z\in[u,u']\cup [u',v']\cup[v',v]$.  If $z\in [u,u']$ then $u'\in \pi_Y(z)$, so $\dist^\pi_Y(u,w)\le J+\dist^\pi_Y(z,w)\le J(2\delta+1)$, contradicting the fact that $w$ is between $a$ and $b$. For the same reason, $z$ cannot lie on $[v,v']$, so $[a,b]\subseteq \mc N_{2\delta}([u', v'])\subseteq \mc N_{K+2\delta}(Y)$. Now,
    $$ \dist^\pi_Y(a,b)\ge \left( \dist^\pi_Y(u,v)-2J(3\delta+1) \right)\ge \left( \theta-2B-2J(3\delta+1) \right)=3JD_0>J, $$
    so by \Cref{lem:lipschitzproj} we have that $\dist(a,b)\ge \frac{1}{J}\dist^\pi_Y(a,b)\ge 3D_0$. Thus, if we cover $\gamma$ by finitely many segments $\gamma_1, \dots, \gamma_\ell$, each of length $D_0$, we must have that $\gamma_i\subseteq[a,b]\subseteq \mc N_{K+2\delta}(Y)$ for some $i\le l$. 
    
    \begin{figure}[htp]
    \centering
    \includegraphics[width=4in]{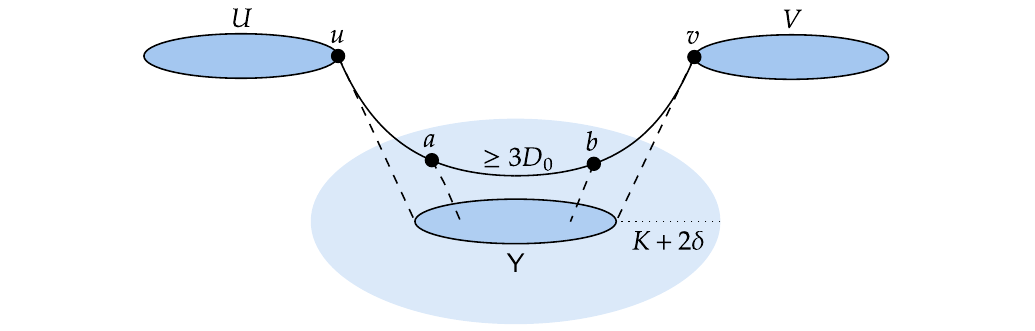} 
    \caption{The domains and geodesic involved in the proof of \Cref{projaxiom:finitelymanybig} in \Cref{P:hyp-projcplx}. The geodesic $[u,v]$ has a long subsegment in a uniform neighborhood of $Y$, and the endpoints $a$ and $b$ of this subsegment project to $Y$ close to the endpoints of the whole geodesic.}
    \label{fig:projcomplex}
    \end{figure}
   
    It is enough to show that if $\gamma_i\subseteq \mc N_{K+2\delta}(Y)\cap \mc N_{K+2\delta}(Y') $ for some $Y,Y'\in \mc Y$, then $Y=Y'$, because then the set $\{Y\in \mc Y\mid \dist_Y^\pi(U,V)\ge \theta\}$ will have cardinality at most $l$. To see this, notice that $\gamma_i$ has length $D_0$, so there exist $p,q\in Y\cap N_{K+2\delta}(\gamma_i)\subseteq Y\cap N_{2K+4\delta}(Y')$ such that $\dist(p,q)\ge D_0-2(K+2\delta)=M(2K+4\delta)+1$. Since $\mc Y$ is $M_0$-separated, this means that $Y=Y'$, as required.
\end{proof}

For later purposes, we can define a projection complex structure on $\hat X^{(0)}$ itself, and this will serve as a unified framework for studying projections between spaces in $\mathcal Y$ and from $X$ to spaces in $\mathcal Y$.

For every $x\in X$ let $\pi_x\colon X\to \{x\}$ be the constant map, and for every cone point $v_Y$, let $\pi_{v_Y}$ be the map  $\pi_Y\colon X\to 2^Y$. For every $Y\in \mc Y$ and $x\in \hat X$, set $\pi_{x}(v_{Y})=\pi_x(Y)$, as in \Cref{defn:extended_proj}. Equip the set $\hat X^{(0)}$ with distance functions $\dist^\pi_x(y,z)=\diam(\pi_x(y)\cup \pi_x(z))$.

\begin{cor}\label{cor:proj_complex_with_points}
    Suppose $(X,\mc Y)$ satisfies \Cref{hyp:metric} with respect to $(\delta, K,M_0)$. There exists $\Theta=\Theta(\delta, K,M_0)$ such that $(\hat X^{(0)}, \{\dist^\pi_x\}_{x\in \hat X})$ satisfies the projection complex axioms (\Cref{defn:Projection_axioms}) with constant $\Theta$. Moreover, $\Theta$ is bounded linearly in $M_0$. 
\end{cor}

\begin{proof}
    $\hat X^{(0)}$, which we can identify with $\mc Y\cup X^{(0)}$, is a $(M_0+4K+4\delta)$-geometrically separated collection of $K$-quasiconvex subsets, because points are $0$-quasiconvex and balls of radius $2K+2\delta$ have diameter $4K+4\delta$. Then the result follows from \Cref{P:hyp-projcplx}, applied to $(X, \hat X^{(0)})$.
\end{proof}


\section{Hyperbolic quotients from spinning families}

\label{sec:spinningQuotient}
Recall that an isometric action of a group $G$ on a space $X$ is \emph{$R$-cobounded} for some constant $R\ge 0$ if $X=\mc N_R(G\cdot x)$ for every $x\in X$. In this section, we work under the following standing assumption, which introduces a group action into the framework of \Cref{sec:background}.

\begin{hyp} \label{hyp:spinning}
Let $(X,\mc Y)$ satisfy \Cref{hyp:metric} with respect to $(\delta, K, M_0)$. Let $M$ be the geometric separation function from \Cref{eq:M(e)}. Let $J=J(\delta,K)$ be the Lipschitz constant of $\{\pi_Y\}_{Y\in \mc Y}$ from \Cref{lem:lipschitzproj}, let $B=B(\delta, K, M_0)$ be the bound on the diameter of projections from \Cref{lem:boundedproj_for_qc}, and let $C=C(\delta, K,M_0)$ be the bounded geodesic image constant from \Cref{lem:strongBGI}.  

\noindent Let $G$ be a group acting by isometries on $X$. Suppose that the following hold.
\begin{enumerate}[start=4]
    \item \label{I:cobounded} The $G$-action on $X$ is $R$-cobounded, for some $R\ge0$; and
    \item 
    \label{I:invariance}
    the collection $\mc Y$ is $G$--invariant, that is, for each $Y \in \mc Y$ every translate $gY$ is in $\mc Y$.
\end{enumerate}
Let $\Theta=\Theta(\delta,K,M_0)$ be the constant from \Cref{cor:proj_complex_with_points}. Set $\widetilde\Theta=\max\{\Theta, 2(JR+B+\Theta)/33\}$ and $\Zhe=33\widetilde\Theta$, and let $\mc P\coloneq \mc P_{\Zhe}(\hat X^{(0)})$ be the associated projection complex, as in \Cref{def:projcomplex}. Let $L_{hyp}=L_{hyp}(\mc P)$ be the threshold from \cite[Theorem~1.1]{ClayMangahas}, and set \begin{equation}
    \label{eq:bound_on_l_spinning}
\ol L=\max\{ L_{hyp}, 40C, 10(2(B+JR)+2J+1)\}.\end{equation}
Finally, suppose that the following hold.
\begin{enumerate}[start=6]
    \item
    \label{I:equivariance}
    For each $Y \in \mc Y$ there is a non-trivial subgroup $H_Y\le \Stab_G(Y)$ such that $gH_Yg^{-1} = H_{gY}$ for each $g\in G$; and
    \item
    \label{I:spinning_bound}
    There exists $L>\ol L$ such that, for any $Y\in \mc Y$, any $x \neq v_Y \in  \hat X$, and any nontrivial $h \in H_Y$, we have $\dist^\pi_Y(x, hx) > L$.
\end{enumerate}
If the above are satisfied, we say $(X,\mc Y, G,\{H_Y\}_{Y\in \mc Y})$ satisfies \Cref{hyp:spinning} with respect to $(\delta, K, M_0, R, L)$. We omit the constants when they are unimportant or clear from context.
\end{hyp}

\noindent We emphasize that the projection complex $\mc P$ in \Cref{hyp:spinning} is formed using \textit{all} of the points in $\hat X^{(0)}$, not just the set $\mc Y$.

\begin{remark}[Linear dependence of $L_{hyp}$ on $M_0$]\label{rem:l_hyp}
    By inspection of \cite{ClayMangahas}, one sees that the constant $L_{hyp}$ is bounded linearly in $M_0$. Indeed, Clay and Mangahas first introduce constants $C_e,C_p,C_g$ such that the following hold.
    \begin{itemize}
        \item Whenever $x,y$ are adjacent in $\mc P$, then $\dist^\pi_z(x,y)\le  C_e$ for every $x\neq z\neq y$. The definition of $\mc P$ and \Cref{thm:modified_dist}, which states that $\dist^\pi_z$ and $\dist_z$ differ by at most $2\Theta$,  ensures that we can choose $C_e=\Zhe+2\Theta$.
        \item Whenever $x,y$ are joined by a path in $\mc P$ which does not pass through the $2$-neighborhood of $z$, then $\dist^\pi_z(x,y)\le C_p$. By \Cref{lem:boundedpathimage} we can choose $C_p=11\widetilde\Theta+2\Theta$.
        \item Whenever $x,y$ are joined by a geodesic in $\mc P$ which does not pass through $z$, then $\dist^\pi_z(x,y)\le C_g$. By \Cref{cor:bgi_for_proj} we can use $C_g=22\widetilde\Theta+6\Zhe+2\Theta$.
    \end{itemize}
    From here, they set:
    \begin{itemize}
        \item $m=11C_e+6C_g+5C_p$ and $L_0=4(m+\Theta)+1$ \cite[Lemma 2.1]{ClayMangahas}; 
        \item $L_{short}= \max\{L_0, 5m, 14\Theta\}$ \cite[Proposition 3.2]{ClayMangahas};
        \item $L_{lift}(0)=\max\{L_{short},40 C_g\}$ \cite[Proposition 4.3]{ClayMangahas};
        \item $L_{hyp}=L_{lift}(0)$ \cite[Theorem 1.1]{ClayMangahas}.
    \end{itemize}
    In the end, $L_{hyp}$ is a piecewise linear function of $\widetilde\Theta$ and $\Theta$, both of which are bounded linearly in $M_0$. We also stress that, as pointed out in \cite[Lemma 2.1]{ClayMangahas}, $L_{hyp}\ge L_0$ is also greater than the constant $L(\mc P)$ from \cite[Theorem 1.6]{ClayMangahasMargalit}, so all results of that paper concerning only the action $G\curvearrowright\mc P$ still hold in our setting.
\end{remark}

\subsection{Properties of spinning families}
For each $x\in \hat X$, define $H_x=H_Y$ if $x=v_Y$ for some $Y\in \mc Y$, and define $H_x=\{1\}$ if $x\in X$. The collection of subgroups $\{H_x\}_{x\in \hat X}$ form a \textit{spinning family}, a notion introduced by Clay, Mangahas, and Margalit \cite[Section~1.7]{ClayMangahasMargalit}.

\begin{defn}[Spinning family]
    Let $\mc P$ be a projection complex, and let $G$ be a group that acts on $\mc P$ by isometries.  For each vertex $x\in \mc P$, let $H_x$ be a subgroup of the stabilizer of $x$ in $G$. Let $L>0$.  The family of subgroups $\{H_x\}$ is an \textit{equivariant $L$--spinning family} of subgroups of $G$ if it satisfies the following two conditions:
    \begin{itemize}
        \item \textit{Equivariance:}  If $g\in G$ and $x$ is a vertex of $\mc P$, then $gH_xg^{-1}=H_{gx}$.
        \item \textit{Spinning:}  For any $x\neq y\in \mc P$ and any non-trivial $h\in H_y$, we have $\dist_y^\pi(x,hx)>L$.
        \end{itemize}
\end{defn}
\noindent In our setting, $G$ acts on $\mc P=\mc P_{\Zhe}(\hat X^{(0)})$ by isometries, as the projections $\{\pi_Y\}_{Y\in \mc Y}$ and $\{\pi_x\}_{x\in X}$ are $G$-equivariant.  The following lemma is immediate.
\begin{lem}
    If $(X, \mc Y, G, \{H_Y\}_{Y\in\mc Y})$ satisfies \Cref{hyp:spinning} with respect to $(\delta, K, M_0, R, L)$,  then the family of subgroups $\{H_x\}_{x\in \hat X}$ forms an equivariant $L$--spinning family, with respect to the action on $\mc P$.
\end{lem}
\noindent We gather here some facts about spinning families from \cite{ClayMangahasMargalit, ClayMangahas}. First, the main theorem of \cite{ClayMangahasMargalit} states that the subgroup $N\coloneq \llangle H_x\rrangle_{x\in \hat X}=\llangle H_Y\rrangle_{Y\in \mc Y}$ normally generated by the spinning family is a (generally infinite) free product:
\begin{thm}\label{thm:N_is_freeprod}
    Let $(X, \mc Y, G,\{H_Y\}_{Y\in \mc Y})$ be as in \Cref{hyp:spinning}, and let $\mc O$ be any set of orbit representatives for the action of $N$ on $\mc Y$. Then $N\cong \bigast_{Z\in \mc O} H_Z$.
\end{thm}

\noindent We can use \Cref{thm:N_is_freeprod} to characterize stabilizers in $N$:
\begin{cor}\label{cor:stabilisers_in_n}
    Let $(X, \mc Y, G,\{H_Y\}_{Y\in \mc Y})$ be as in \Cref{hyp:spinning}. For $Y\in \mc Y$ we have that $N\cap \Stab_{G}(Y)=H_Y$.
\end{cor}

\begin{proof}
    Let $\mc O$ be a collection of $N$-orbit representatives including $Y$. Since $H_Y$ is a non-trivial free factor of $N\cong \bigast_{Z\in \mc O} H_Z$, it is malnormal in $N$, i.e., it intersects  any of its conjugates trivially. Furthermore, if $n\in N$ fixes $Y$, then it normalizes $H_Y$ by \Cref{hyp:spinning}.\eqref{I:equivariance}, and so  $n\in H_Y$ by malnormality, as required.
\end{proof}

\begin{remark}[Complexity]\label{rem:complexity}
    There is a partial order $\prec$ on elements of $N$, called \emph{complexity}, which is invariant under conjugation by elements of $N$. The only facts we will need about this partial order is that it has the identity as its unique minimal element, and that descending chains have finite length; this will allow for inductive arguments on the complexity of an element. We omit the full definition and refer the reader to \cite[Section~3]{ClayMangahas} for details.
\end{remark}

\begin{prop}[{\cite[Lemma 3.2]{ClayMangahas}}]\label{prop:new_shortening_pair}
    Assume that $(X, \mc Y,  G, \{H_Y\}_{Y\in \mc Y})$ satisfies \Cref{hyp:spinning}. 
    Let $x\in \hat X$ and $ h \in N$ be such that $hx\neq x$.
    Then there exist $Y \in \mc{Y}$ and $h_Y\in H_Y$ such that the following hold.
    \begin{enumerate}
        \item Either $v_Y\in\{x,hx\}$ or $\dist_Y^\pi(x, hx)>L/10$; and 
        \item $h_Yh \prec h$.
    \end{enumerate}
    We say that $(Y,h_Y)$ is a \emph{shortening pair} for $(x,h)$.
\end{prop}

\begin{cor}[Large injectivity radius]\label{cor:tau}
Assume that $(X, \mc Y,  G, \{H_Y\}_{Y\in \mc Y})$ satisfies \Cref{hyp:spinning}, and let $\tau=(L/10-2(B+JR))/J\ge 2$. For every $x\in X$ and every $h\in N-\{1\}$,  we have $\dist_X(x,hx)>\tau$. In particular, $N$ acts freely on $X$. 
\end{cor}

\begin{proof}Let $\mc V_R(x) = \{ Y \in \mc{Y} \mid \dist_{X}(x, Y) \le R \} $ be the collection of nearest subspaces to $x$. This set is non-empty as the action is $R$-cobounded, so let $Y\in\mc V(x)$. If $hY=Y$, then $h\in H_Y$ by \Cref{cor:stabilisers_in_n}. Hence $\dist_Y(x,hx)>L>J$ by \Cref{hyp:spinning}, and so $\dist_X(x,hx)>L/J$ by \Cref{lem:lipschitzproj}.

Suppose instead that $hY\neq Y$. Then by \Cref{prop:new_shortening_pair} there exists $U\in\mc Y$ such that $\dist_U(Y,hY)>L/10$. Since projections are $J$-Lipschitz and $x$ is $R$-close to $Y$, we have that $\dist_U(x,hx)>L/10-2(B+JR)>J$, since $L$ is greater than the constant $\ol L$ from \Cref{eq:bound_on_l_spinning}. Hence again \Cref{lem:lipschitzproj} yields that $\dist_X(x,hx)>(L/10-2(B+JR))/J$.
\end{proof}

\noindent Clay and Mangahas further investigated the properties of spinning families and showed that $\mc P/N$ is hyperbolic, provided that $L>L_{hyp}$ \cite[Theorem~1.1]{ClayMangahas}. In \Cref{S:quotient}, we use these properties to prove that $\hat X/N$, rather than the quotient of $\mc P$, is uniformly hyperbolic, regardless of the choice of $L$ (as long as $L$ is bigger than the constant $\ol L$ from \Cref{eq:bound_on_l_spinning}).

\subsection{Hyperbolicity of the quotient graph}\label{S:quotient}
We now turn our attention to the quotient $\ol X \coloneq  \hat X/N$.  Let $q\colon \hat X \to \ol X$ be the quotient map. We say that a subgraph $T\subseteq \hat X$ \emph{lifts} a subgraph $\ol T\subseteq \ol X$ if $q$ restricts to an isometry between $T$ and $\ol T$. To avoid confusion, we will denote points in $\ol X$ by $\ol x$ and points in $\hat X$ (and $X$) simply by $x$.  By an abuse of notation, we consider points in $\ol X$ to be equivalence classes of points in $\hat X$:  if $q(x)=\ol x$, we will write $x\in \ol x$ and say $x$ is a \textit{representative} of $\ol x$. We first define a certain class of representatives.

\begin{defn}\label{def:minrep}
Given  $\ol{x}, \ol y \in \ol{X}$, two points $x \in \ol{x}$ and $y\in \ol y$ are \emph{minimal distance representatives}, or simply \emph{minimal}, if 
$\dist_{\hat{X}}(x,y) = \inf_{x'\in\ol x,\,y'\in \ol y} \dist_{\hat{X}}(x',y')$.
\end{defn}

\begin{lem}\label{lem:bendinggeodesics}
    Let $\ol x, \ol y\in \ol X$, and let $x\in \ol x$, $y\in \ol y$ be minimal representatives. Suppose $\alpha=\alpha_1*\alpha_2$ is a geodesic in $\hat X$ from $x$ to $y$ with endpoints $(\alpha_1)_+=v_Y=(\alpha_2)_-$ for some $Y\in \mc Y$.  For any $h_Y \in H_Y$, the path $\alpha'=\alpha_1*h_Y\alpha_2$ is also a geodesic in $\hat X$. 
\end{lem}

\begin{proof}
    Since $h_Y\in \Stab(Y)$, we have $h_Yv_Y=v_Y$.  Therefore $\alpha'$ has the same length as $\alpha$, and so $\dist_{\hat X}(x,h_Yy)\le \dist_{\hat X}(x,y)$.  On the other hand, since $x$ and $y$ are minimal and $h_Yy$ is also a representative of $\ol y$, the distance from $h_Yy$ to $x$ is at least the distance from $y$ to $x$.  Together, this implies that $\dist_{\hat X}(x,y)=\dist_{\hat X}(x,h_Yy)$, and so $\alpha'$ is a geodesic in $\hat X$.
\end{proof}

\noindent We say that we \textit{bend the path $\alpha$ along $v_Y$} to obtain the path $\alpha'=\alpha_1*h_Y\alpha_2$, and we call $\alpha'$ a \textit{bent path}. Notice  that, since the endpoints of $\alpha$ are minimal, the image in the quotient $q(\alpha)$ is a geodesic between $\ol x$ and $\ol y$, and both $\alpha$ and $\alpha'$ are lifts of $q(\alpha)$.

\par\medskip
The leitmotiv of many arguments throughout this paper is that, if one combines \Cref{prop:new_shortening_pair} with the strong bounded geodesic image \Cref{lem:strongBGI}, then many combinatorial configurations lift from $\ol X$ to $\hat X$. We showcase this in the next Proposition, where we prove that geodesic quadrangles admit lifts. The ``moreover'' part will be relevant in \Cref{sec:preserving_loxo} to ensure that the image of certain WPD elements of $G$ remain WPD elements of $G/N$.

\begin{prop}\label{prop:liftingQuadrilaterals}
    Let $\ol Q\subseteq \ol X$ be a geodesic quadrangle with vertices $\ol v_1,\ol v_2, \ol v_3, \ol v_4$. Then there exists a geodesic quadrangle $Q\subseteq \hat X$ which lifts $\ol Q$. 

    Moreover, if the geodesics $[\ol v_1, \ol v_2]$ and $[\ol v_3, \ol v_4]$ of $\ol Q$ have lifts $[v_1', v_2']$ and $[v_3', v_4']$, respectively, such that $\sup_{Y\in \mc Y}\dist_Y^\pi(v_1',v_2')\le L/40$ and $\sup_{Y\in \mc Y}\dist_Y^\pi(v_3',v_4')\le L/40$, then the lifts of $[\ol v_1, \ol v_2]$ and $[\ol v_3, \ol v_4]$ contained in $Q$ can be chosen to be $N$--translates of $[v_1', v_2']$ and $[v_3', v_4']$, respectively.
\end{prop}

\begin{proof}
    Lift each geodesic side of $\ol Q$ to a geodesic in $\hat X$. Up to replacing each lift by some $N$-translate, this produces a concatenation of four geodesics $\gamma_1=[v_1,v_2]$, $\gamma_2=[v_2,v_3]$, $\gamma_3=[v_3, v_4]$, and $\gamma_4=[v_4,v_5]$, where $v_5=hv_1$ for some $h\in N$.  Under the hypothesis of the ``moreover'' part, we can choose $\gamma_1$ and $\gamma_3$ to be $N$--translates of $[v_1', v_2']$ and $[v_3', v_4']$, respectively. We call this configuration an \emph{open lift} of $\ol T$.

    Recall that, by \Cref{rem:complexity}, descending chains in $\prec$ have finite length, so we  proceed by induction on the complexity of $h$. If $h=1$, or more generally if $hv_1=v_1$, then the union $\gamma_1\cup\ldots\cup\gamma_4$ is already a geodesic quadrangle $Q$ lifting $\ol Q$, as required. Otherwise, we show that we can find another open lift of $\ol Q$ where the new endpoints differ by some $h'\prec h$, reducing the complexity. Indeed, since $hv_1\neq v_1$, \Cref{prop:new_shortening_pair} provides a shortening pair $(Y,h_Y)$ for $(v_1,h)$. There are two cases to consider.
    \begin{itemize}
        \item Suppose first that $v_Y=v_j$ for some $1\le j\le 5$. For every $i\ge j$, replace $\gamma_i$ by $h_Y\gamma_i$.  This results in another open lift of $\ol T$, made of $N$-translates of the original $\gamma_i$s, but now $v_1$ and $h_Y hv_1$ differ by $h_Yh$, which has lower complexity than $h$ by definition of the shortening pair. 
        \item Now suppose that $v_Y\neq v_j$ for all $j$, so that all projections from $v_j$ to $Y$ are defined. By \Cref{prop:new_shortening_pair} we have that $\dist_Y^\pi(v_1, hv_1)> L/10$, and the triangle inequality yields that at least one of $\dist_Y^\pi(v_1, v_2)$, $\dist_Y^\pi(v_2, v_3)$, $\dist_Y^\pi(v_3, v_4)$, and $\dist_Y^\pi(v_4, hv_1)$ is larger than $L/40$. We assume that $\dist_Y^\pi(v_4, hv_1)> L/40$; an analogous argument holds in the other cases. Since we chose $L$ greater than the constant $\ol L$ from \Cref{eq:bound_on_l_spinning}, the quantity $L/40$ is greater than the constant $C$ from the bounded geodesic image \Cref{lem:strongBGI}. It follows that $v_Y$ lies on the geodesic $\gamma_4$. Bend $\gamma_4$ at $v_Y$ by $h_Y$; in other words, apply $h_Y$ to every vertex of the open lift between $v_Y$ and $hv_1$. See \Cref{fig:openlift}. Since the bent path is still a geodesic lift of $[\ol v_4,\ol v_1]$, this operation produces a new open lift of $\ol Q$. Moreover, as before, $v_1$ and $h_Yhv_1$ now differ by $h_Yh\prec h$. Notice that, in the setting of the ``moreover'' part, both $\dist_Y^\pi(v_1, v_2)$ and $\dist_Y^\pi(v_3, v_4)$ are at most $L/40$, hence the bending procedure replaces each of $\gamma_1$ and $\gamma_3$ by an $N$--translate.
    \end{itemize}
    In both cases, we conclude by induction: after finitely many steps, we have obtained a geodesic quadrangle $Q$ that lifts $\ol Q$, and, in the ``moreover'' setting, the lifts of $[\ol v_1, \ol v_2]$ and $[\ol v_3, \ol v_4]$ contained in $Q$ are $N$-translates of $[v_1', v_2']$ and $[v_3', v_4']$, respectively.
    \end{proof}

    \begin{figure}[htp]
        \centering
        \includegraphics[width=4in]{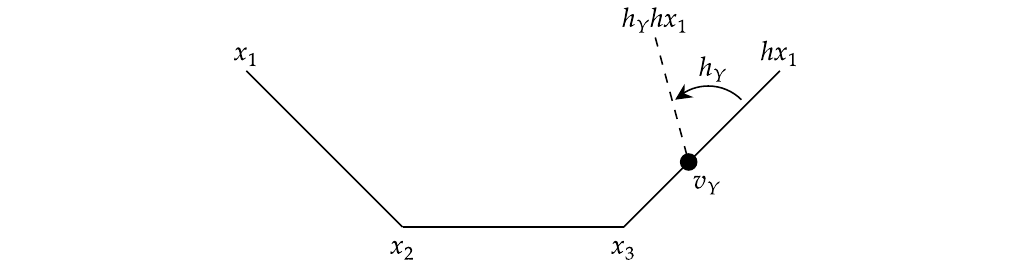}
        \caption{The open lift from the proof of \Cref{prop:liftingQuadrilaterals}, and its bending at the shortening vertex (here, the dashed line).}
        \label{fig:openlift}
    \end{figure}

\begin{thm}\label{thm:hyperbolicity_olX_with_points}
    If $(X, \mc Y, G,\{H_Y\}_{Y\in \mc Y})$ satisfies \Cref{hyp:spinning}, then $\ol X$ is $\hat\delta$-hyperbolic, where $\hat\delta=\hat\delta(\delta, K)$ is the hyperbolicity constant of $\hat X$ from \Cref{lem:hyp_constant_cone_off}.
\end{thm}

\begin{proof}
    Let $\ol T\subseteq \ol X$ be a geodesic triangle, which we see as a degenerate quadrangle, and let $T\subseteq \hat X$ be a lift of $\ol T$. Given any point $\ol p\in \ol T$ let $p\in T$ be its lift. By slimness of triangles in $\hat X$, there exists $q\in T$ on one of the other sides such that $\dist_{\hat X}(p,q)\le \hat\delta$. But then, since the projection map $q\colon \hat X\to \ol X$ is $1$-Lipschitz, we have that $\dist_{\ol X}(\ol p, \ol q)\le\dist_{\hat X}(p,q)\le \hat\delta $, where $\ol q$ is the projection of $q$. This proves that $\ol T$ is $\hat\delta$-slim, as required. 
\end{proof}

\subsection{Preserving acylindrical hyperbolicity}\label{sec:preserving_loxo}
In this subsection, we show that quotients of acylindrically hyperbolic groups by spinning families are acylindrically hyperbolic.  We go about this by showing that if $G\curvearrowright X$ admits independent loxodromic WPD elements whose axes are ``transverse'' to the collection $\mc Y$, then the quotient $G/N$ does, as well.  

Recall that, given an isometric action of a group $G$ on a hyperbolic space $X$, an element $g\in G$ is \emph{loxodromic} if for every $x\in X$ the map $\Z\to X$ mapping $n$ to $g^n(x)$ is a quasi-isometric embedding. In this case, the \emph{limit set} of $\langle g \rangle$ consists of the two endpoints of the quasigeodesic $\{g^n(x)\}_{n\in \Z}$ inside $\partial X$, which do not depend on $x$. Two loxodromic isometries of a hyperbolic space are \textit{independent} if their limit sets are disjoint.
 Furthermore, following \cite{bestvina_fujiwara},  a loxodromic element $g\in G$ is \emph{weakly properly discontinuous}, or \emph{WPD}, if for every $\varepsilon>0$ and every $x\in X$ there exists $M\in\N$ such that
$$\left | \left\{h\in G \mid \dist_{X}\left(x, hx\right)\le\varepsilon,\, \dist_{X}\left(g^Mx, hg^Mx\right)\le \varepsilon\right\}\right |<\infty.$$
The core of the arguments in this subsection is the following lemma.
\begin{lem}\label{lem:isoproj}
    Let $x,y\in X$ be such that $\sup_{Y\in \mc Y}\dist^\pi_Y(x,y)\le L/20$. Then
    \begin{equation}
        \label{eq:isoproj}
    \frac{1}{L/20+2C}\dist_X(x,y)\le \dist_{\hat X}(x,y)=\dist_{\ol X}(\ol x, \ol y).
    \end{equation}
\end{lem}

\begin{proof}
    Let $\gamma$ be an $\hat X$-geodesic connecting $x$ to $y$, and let $\widetilde\gamma$ be its de-electrification in $X$. Clearly the length of $\widetilde\gamma$ is at least $\dist_{X}(x,y)$; hence the inequality on the left of \eqref{eq:isoproj} follows if we bound the length of $\widetilde\gamma$ by $(L/20+2C)$-times the length of $\gamma$. In turn, by the definition of the de-electrification, it suffices to show that whenever $\gamma$ has a two-edge subsegment of the form $\{c,v_Y,d\}$, where $Y\in \mc Y$ and $c,d\in Y$, then $\dist_X(c,d)\le L/20+2C$. To see this, assume that $c$ is closer to $x$ than $d$. Since $\gamma|_{[x,c]}$ does not pass through $v_Y$, $\dist^\pi_Y(c,x)\le C$ by \Cref{lem:strongBGI}, and symmetrically $\dist^\pi_Y(d,y)\le C$. Therefore, since $c,d\in Y$, we have that
    \[\dist_X(c,d)=\dist^\pi_Y(c,d)\le \dist^\pi_Y(c,x)+\dist^\pi_Y(x,y)+\dist^\pi_Y(y,d)\le L/20+2C.\]
    For the equality on the right of \eqref{eq:isoproj}, suppose $x$ and some $y'\in \ol y$ are minimal and let $\eta$ be an $\hat X$-geodesic connecting $x$ to $y'$. We claim that $\gamma$ and $\eta$ have the same length. To see this, let $h\in N$ be such that $hy=y'$. We proceed by induction on the complexity of $h$. If $h$ fixes $y$ we are done. Otherwise, let $(Y,h_Y)$ be a shortening pair for $(y,h)$, as in \Cref{prop:new_shortening_pair}. We have that
    $$L/10<\dist^\pi_Y(y,y')\le \dist^\pi_Y(y,x)+\dist^\pi_Y(x,y')\le L/20+ \dist^\pi_Y(x,y').$$
    Since $L/10-L/20=L/20>C$, \Cref{lem:strongBGI} implies that $v_Y$ lies on $\eta$. Bending $\eta$ at $v_Y$,  we conclude by induction.
\end{proof}

\begin{prop}\label{prop:independent_in_quotient}
    Let $f,g\in G$ be independent loxodromic isometries for the action on $X$, and let $x\in X$ be such that $\sup_{Y\in \mc Y}\sup_{m,n\in\mathbb{Z}}\dist^\pi_Y(f^m x, g^n x)\le L/40.$ Then $\ol f, \ol g\in G/N$ are independent loxodromics for the action on $\ol X$. Moreover, if $f $ is a WPD element, then so is $\ol f$.  
\end{prop}

\begin{proof}
    The hypotheses, together with \Cref{lem:isoproj}, imply that $\dist_{\ol X}(\ol x, \ol f^n \ol x)$ and $\dist_{\ol X}(\ol x, \ol g^n \ol x)$ both grow linearly in $n$, so $\ol f$ and $\ol g$ are loxodromic. Moreover, the same \Cref{lem:isoproj} yields that
    \[\dist_{Haus}(\langle\ol f\rangle\ol x, \langle \ol g\rangle \ol x)\ge\frac{1}{L/20+2C}\dist_{Haus}(\langle f\rangle x, \langle g\rangle x)=\infty,\]
    so $\ol f$ and $\ol g$ are indeed independent.

    Now suppose that $f$ is a  WPD element with respect to the action of $G$ on $X$.  By hypothesis, a quasiaxis for $f$ can fellow travel along a subspace $Y\in \mc Y$ only for a uniformly bounded amount, so \cite[Theorem~2.4]{MaherMasaiSchleimer} implies that $f$ is a WPD element with respect to the action of $G$ on $\hat X$. Now fix $\varepsilon >0$ and $\ol x\in \ol X$, let $x\in \hat X$ be a lift of $\ol x$, and let $M>0$ be such that the set
    \[
    \Delta = \{h \in G \mid \dist_{\hat X}(x, hx) \leq \varepsilon,\, \dist_{\hat X}(f^Mx,fg^Mx)\leq \varepsilon\} \]
    is finite.  Let $D$ denote the size of $\Delta$.

    Suppose there exist distinct elements $\ol g_1,\dots, \ol g_p\in G/N$ such that 
    \[ 
    \dist_{\ol X}(\ol x, \ol g_i\ol x)\leq \varepsilon \quad \textrm{and} \quad \dist_{\ol X}(\ol f^M\ol x, \ol g_i\ol f^M \ol x)\leq \varepsilon.
    \]
    Consider the geodesic quadrilateral $\ol Q$ in $\ol X$ with vertices $\ol x, \ol f^M \ol x, \ol g_i\ol f^M\ol x, $ and $\ol f^M \ol x$.  By the ``moreover'' statement of \Cref{prop:liftingQuadrilaterals}, there exist $g_i\in \ol g_i$ such that the quadrilateral $\ol Q$ lifts to a geodesic quadrilateral $Q$ of $\hat X$ with vertices $x, f^M x, g_if^Mx$, and $g_ix$. In particular, we have 
    \[ 
    \dist_{\hat X}(x,  g_i x)=\dist_{\ol X}(\ol x, \ol g_i\ol x)\leq \varepsilon \quad \textrm{and, similarly,} \quad \dist_{\hat X}( f^M x,  g_i f^M  x)\leq \varepsilon.
    \]
    It follows that $g_i\in \Delta$ for each $1\leq i \leq p$, and hence $p<D$.  This shows that the element $\ol f$ is WPD with respect to the action of $G/N$ on $\ol X$, as required.
\end{proof}

\begin{cor}\label{cor:preserveAH}
    Let $(X, \mc Y, G,\{H_Y\}_{Y\in \mc Y})$ satisfy \Cref{hyp:spinning}. Suppose there exist independent loxodromics $f,g\in G$ as in \Cref{prop:independent_in_quotient}, and that $f$ is a WPD element. Then $G/N$ is acylindrically hyperbolic.
\end{cor}

\begin{proof}
    By \Cref{prop:independent_in_quotient}, $G/N$ acts with independent loxodromic isometries on $\ol X$ and so is not virtually cyclic.  Furthermore, $\ol f$ is a loxodromic WPD element for this action. Therefore $G/N$ is acylindrically hyperbolic by \cite[Theorem 1.2]{Osin_acyl_hyp}.
\end{proof}

\section{Hierarchical hyperbolicity of spinning quotients}
\label{sec:hierarchical}
Up until this point, we have worked with an arbitrary hyperbolic graph $X$ with a group action, along with an equivariant collection of uniformly quasiconvex subspaces. In this section, we specialize to the case when $X$ arises as the top level hyperbolic space in a (relative) hierarchically hyperbolic group structure.

\subsection{Background on hierarchical hyperbolicity}
This subsection gathers definitions and properties of (relative) hierarchically hyperbolic spaces and groups introduced in \cite{BHS_HHSII}; see \cite[Definition~2.8]{Russell_Rel_Hyp} for this formulation of some of the axioms.

\begin{defn}[Relative hierarchically hyperbolic space]\label{defn:HHS}
    Let $E>0$, and let $\mc{X}$ be an $(E,E)$--quasigeodesic space.  A \emph{relatively hierarchically hyperbolic space (relative HHS) structure with constant $E$} for $\mc{X}$ is an index set $\mf S$ and a set $\{ \mc{C}W \mid W\in\mf S\}$ of geodesic spaces $(\mc{C}W,\dist_W)$ such that the following axioms are satisfied. 
    \begin{enumerate}[label=(\arabic*{})]
        \item\label{axiom:projections} \textbf{(Projections.)} For each $W \in \mf{S}$, there exists a \emph{projection} $\pi_W \colon \mc{X} \rightarrow 2^{\mc{C}W}$  that is a $(E,E)$--coarsely Lipschitz, $E$--coarsely onto, $E$--coarse map.
	
        \item\label{axiom:nesting} \textbf{(Nesting.)} If $\mf S \neq \emptyset$, then $\mf{S}$ is equipped with a  partial order $\nest$ and contains a unique $\nest$--maximal element. When $V\nest W$, we say $V$ is \emph{nested} in $W$. For each $W\in\mf S$, we denote by $\mf S_W$ the set of all $V\in\mf S$ with $V\nest W$.  Moreover, for all $V,W\in\mf S$ with $V\propnest W$ there is a specified non-empty subset $\rho^V_W\subseteq \mc{C}W$ with $\diam(\rho^V_W)\le E$.
            
        \item\label{axiom:finite_complexity}  \textbf{(Finite complexity.)} Any set of pairwise $\nest$--comparable elements has cardinality at most $E$.
            
        \item\label{axiom:orthogonal} \textbf{(Orthogonality.)} The set $\mf S$ has a symmetric relation called \emph{orthogonality}. If $V$ and $W$ are orthogonal, we write $V\perp W$ and require that $V$ and $W$ are not $\nest$--comparable. Further, whenever $V\nest W$ and $W\perp U$, we require that $V\perp U$. We denote by $\mf{S}_W^\perp$ the set of all $V\in \mf{S}$ with $V\perp W$.
            
        \item\label{axiom:containers}\textbf{(Containers.)} For each $W \in \mf{S}$ and $U \in \mf{S}_W$ with $ \mf{S}_W\cap \mf{S}_U^\perp \neq \emptyset$, there exists $Q \in\mf{S}_W-\{W\}$ such that $V \nest Q$ whenever $V \in\mf{S}_W \cap \mf{S}_U^\perp$. We call $Q$ the \emph{container for $U$ inside $W$}.
		
        \item\label{axiom:transversality} \textbf{(Transversality.)} If $V,W\in\mf S$ are neither orthogonal nor $\nest$-comparable, we say $V$ and $W$ are \emph{transverse}, denoted $V\trans W$.  Moreover, for all $V,W \in \mf{S}$ with $V\trans W$, there are non-empty sets $\rho^V_W\subseteq \mc{C}W$ and $\rho^W_V\subseteq \mc{C} V$, each of diameter at most $E$.
		
        \item\label{axiom:consistency} \textbf{(Consistency.)} For all $x \in\mc X$ and $V,W,U \in\mf{S}$:
	\begin{itemize}
            \item if $V\trans W$, then $\min\left\{\dist_{W}(\pi_W(x),\rho^V_W),\dist_{V}(\pi_V(x),\rho^W_V)\right\}\le E$, and
            \item if $U\nest V$ and either $V\propnest W$ or $V\trans W$ and $W\not\perp U$, then $\dist_W(\rho^U_W,\rho^V_W)\le E$.
	\end{itemize}
		
        \item\label{axiom:hyperbolicity} \textbf{(Hyperbolicity)} For each $W\in \mf S$, either $W$ is $\nest$--minimal or $\mc CW$ is $E$--hyperbolic.
		
        \item\label{axiom:bounded_geodesic_image} \textbf{(Bounded geodesic image.)} For all $V,W\in\mf S$ and for all $x,y \in \mc{X}$, if  $V\propnest W$ and $\dist_V(\pi_V(x),\pi_V(y)) \geq E$, then every $\mc{C}W$--geodesic from $\pi_W(x)$ to $\pi_W(y)$ must intersect $\mc{N}_E(\rho_W^V)$.
		
        \item\label{axiom:partial_realisation} \textbf{(Partial realization.)} If $\{V_i\}$ is a finite collection of pairwise orthogonal elements of $\mf S$ and $p_i\in \mc{C}V_i$ for each $i$, then there exists $x\in \mc X$ so that:
	\begin{itemize}
		\item $\dist_{V_i}(\pi_{V_{i}}(x),p_i)\le E$ for all $i$;
            \item for each $i$ and each $W\in\mf S$, if $V_i\propnest W$ or $W\trans V_i$, we have $\dist_{W}(\pi_W(x),\rho^{V_i}_W)\le E$.
	\end{itemize}

        \item\label{axiom:uniqueness} \textbf{(Uniqueness.)} There is a function $\theta \colon [0,\infty) \to [0,\infty)$ so that for all $r \geq 0$, if $x,y\in\mc X$ and $\dist_\mc{X}(x,y)\geq\theta(r)$, then there exists $W\in\mf S$ such that $\dist_W(\pi_W(x),\pi_W(y))\geq r$. 

        \item\label{axiom:large_link_lemma} \textbf{(Large links.)} For all $W\in\mf S$ and  $x,y\in\mc X$, there exists $\{V_1,\dots,V_m\}\subseteq\mf S_W -\{W\}$ such that $m$ is at most $E \dist_{W}(\pi_W(x),\pi_W(y))+E$, and for all $U\in\mf S_W - \{W\}$, either $U\in\mf S_{V_i}$ for some $i$, or $\dist_{U}(\pi_U(x),\pi_U(y)) \le E$.  
    \end{enumerate}
    We use $\mf{S}$ to denote the relative HHS structure, including the index set $\mf{S}$, spaces $\{\mc{C}W : W \in \mf{S}\}$, projections $\{\pi_W : W \in \mf{S}\}$, and relations $\nest$, $\perp$, $\trans$.  We call an element $U\in \mf{S}$ a \emph{domain}, the associated space $\mc CU$ its \emph{coordinate space}, and call the maps $\rho_W^V$ the \emph{relative projections} from $V$ to $W$. The number $E$ is called the \emph{hierarchy constant} for $\mf{S}$; notice that every $E'\ge E$ is again a hierarchy constant for $\mf S$, so we are often free to enlarge $E$ by a bounded amount. 

    A relative HHS is called a \emph{hierarchically hyperbolic space} (HHS) if for every $W\in \mf S$ the space $\mc CW$ is $E$--hyperbolic, in this case we say $\mf S$ is a \textit{HHS structure} on $\mc X$.

    A quasigeodesic space $\mc{X}$ is a \emph{(relative) HHS with constant $E$} if there exists a (relative) HHS structure on $\mc{X}$ with constant $E$. The pair $(\mc{X},\mf{S})$ denotes a (relatively) HHS equipped with the specific HHS structure $\mf{S}$.
\end{defn}	
\noindent The large links axiom (\Cref{defn:HHS}.\ref{axiom:large_link_lemma}) can be replaced with the following, which is traditionally called \emph{passing up}. 
\begin{enumerate}[label=(\arabic*{}'), start=12]
    \item\label{axiom:passing_up} \textbf{(Passing up).} For every $t>0$, there exists an integer $P=P(t)>0$ such that if $V\in \mf S$ and $x,y\in \mc X$ satisfy $\dist_{U_i}(x,y)>E$ for a collection of domains $\{U_i\}_{i=1}^P$ with $U_i\in \mf S_V$, then there exists $W\in \mf S_V$ with $U_i\propnest W$ for some $i$ such that $\dist_W(x,y)>t$. 
\end{enumerate}

\noindent It was shown in \cite[Lemma~2.5]{BHS_HHSII} that every HHS satisfies the passing up axiom.  The following lemma, which is the converse implication, is stated explicitly in \cite[Section~4.8]{Durham_Cubulating_Infinity}, but the strategy behind the proof appears in \cite[Lemma~5.3]{PetytSpriano_Eyries}.  The proof is written for a HHS, but it does not use the hyperbolicity of the spaces $\mc CW$, and so it also applies to a relative HHS.

\begin{lem}\label{lem:passingupisenough}
   If $(\mc X,\mf S)$ satisfies axioms \ref{axiom:projections}--\ref{axiom:uniqueness} from \Cref{defn:HHS}, as well as the passing up axiom \ref{axiom:passing_up}, then $(\mc X,\mf S)$ is a relative HHS. If moreover the spaces $\mc CW$ are $E$--hyperbolic for all $W\in \mf S$, then $(\mc X,\mf S)$ is a HHS.
\end{lem}

\noindent The following is a combination of \cite[Lemma~1.8]{BHS_HHS_AsDim} and the consistency axiom \ref{axiom:consistency}.
\begin{lem}\label{lem:close_proj_of_orthogonals}
    Let $(\mc X,\mf S)$ be a relative HHS. If $U,V\in \mf S_W$ are not transverse, then $\dist_{S}(\rho^U_W,\rho^V_W)$ is at most $2E$.
\end{lem}

\noindent A hallmark of hierarchically hyperbolic spaces is that every pair of points can be joined by a  special family of quasigeodesics called \emph{hierarchy paths}, each of which projects to a quasigeodesic in each of the spaces $\mc{C}W$.

\begin{defn}
	A \emph{$\lambda$--hierarchy path} $\gamma$ in a HHS $(\mc{X},\mf{S})$
	is a $\lambda$--quasigeodesic with the property that $\pi_W\circ
	\gamma$ is an unparametrized $\lambda$--quasigeodesic
	for each $W \in \mf{S}$.
\end{defn}

\begin{thm}[{\cite[Theorem 6.11]{BHS_HHSII}}]\label{lem:hierarchypath}
	Let $(\mc{X},\mf{S})$ be a relative HHS with constant $E$. There exist $\lambda\geq 1$
	depending only on $E$ so that every pair of
	points in $\mc{X}$ is joined by a $\lambda$--hierarchy
	path.
\end{thm}

\noindent Given $A,B,C,D\in \R$,  write $A\preceq_{C,D}B$ to mean that $A\le BC+D$, and $A\asymp_{C,D} B$ if $B\preceq_{C,D}A\preceq_{C,D}B$. Let $\ignore{A}{B}$ be the quantity which is $A$ if $A\ge B$, and is $0$ otherwise. 
\begin{thm}[{Distance formula, \cite[Theorem~6.10]{BHS_HHSII}}]\label{thm:distance_formula}
Let $(\mc X,\mf S)$ be a relative HHS. There exists $s_0$ such that for all $s\geq s_0$, there exist $k_1,k_2$ so
that for all $x,y\in\mc X$,
$$\dist_{\mc X}(x,y)\asymp_{k_1,k_2}\sum_{U\in\mf S}\ignore{\dist_U(x,y)}{s}.$$
\end{thm}

\begin{defn}[Product region]\label{defn:prodreg}
    Given $U\in\mf S$, the \emph{product region} associated to $U$ is the subspace
$$\mathbf{P}_U=\{x\in \mc {X} \mid \dist_W(x,\rho^U_W)\le E\mbox{ for all } U\trans W \mbox{ or }U\propnest W\}.$$
\end{defn}

\noindent We next introduce the notion of a hierarchically hyperbolic group, first introduced in \cite{BHS_HHSII} and then reformulated as follows in \cite[Section~2.1]{DHS_corrigendum}.
\begin{defn}[Hierarchically hyperbolic group]\label{defn:rel_HHG}
Let $G$ be a finitely generated group and $\mc X$ be the Cayley graph of $G$ with respect to some finite generating set. We say $G$ is a (relatively) \textit{hierarchically hyperbolic group} (HHG) if the following hold.
    \begin{enumerate}[label=(\roman*{})]
        \item\label{HHG_structure} The space $\mc X$ admits a (relative) HHS structure $\mf S$ with hierarchy constant $E$.
        \item\label{HHG_action} There is a $\nest$--, $\perp$--, and $\trans$--preserving action of $G$ on $\mf S$ by bijections such that $\mf S$ contains finitely many $G$--orbits.
        \item\label{HHG_isometries} For each $W\in \mf S$ and $g\in G$, there exists an isometry $g_W\colon \mc CW \to \mc C(gW)$ satisfying the following for all $V,W\in\mf S$, $h\in G$, and $x\in \mc X$.
            \begin{itemize}
                \item The maps $(gh)_W\colon \mc CW \to \mc C(ghW)$ and $g_{hW}\circ h_W\colon \mc CW\to \mc C(ghW)$ coincide.
                \item $g_W(\pi_W(x))=\pi_{gW}(g\cdot x)$.
                \item If $V\trans W$ or $V\propnest W$, then $g_W(\rho^V_W)=\rho^{gV}_{gW}$.
            \end{itemize}
    \end{enumerate}
The structure $\mf S$ satisfying \ref{HHG_structure}---\ref{HHG_isometries} is called a (relative) hierarchically hyperbolic group structure on $G$. We use $(G,\mf S)$ to denote $G$ equipped with a specific (relative) HHG structure on $G$. In \Cref{HHG_isometries} we often drop the subscript on the isometry $g_W$ and simply write $g$ when the domain $W$ is clear from context. 
\end{defn}

\begin{remark}[Convention on $\mc C S$]\label{rem:X_is_CayG} In a HHS structure, one can assume that $\mc C W$ is a graph for every $W\in \mf S$, as it can always be $\Stab_{G}({W})$-equivariantly replaced by a graph (see, e.g., \cite[Lemma 3.B.6]{cornulier2014metric}). Furthermore, if $G$ is a HHG, the projection  $\pi_S\colon G \to \mc C S$ can be assumed to be a bijection. This is because, since $\pi_S$ is $E$-coarsely onto and $G$-equivariant, $G$ acts coboundedly on $\mc C S$, so $\mc C S$ is $G$-equivariantly quasi-isometric to a Cayley graph of $G$ with respect to some possibly infinite generating set. Hence we can (and will) identify points of $\mc C S$ with elements of $G$. Finally, if $\mc C S$ is bounded, we will always assume that the HHG constant $E$ is larger than $\diam (\mc C S)$.
\end{remark}

\noindent For the next lemma, recall that the action of a group $G$ on a metric space $X$ is \textit{acylindrical} if for all $\varepsilon>0$ there exist constants $R=R(\varepsilon)\geq 0$ and $N=N(\varepsilon)\geq 0$ such that for every $x,y\in X$ with $\dist(x,y)\geq R$, we have 
\[
\#\{g\in G \mid \dist(x,gx)\le \varepsilon \textrm{ and } \dist(y,gy)\le \varepsilon\} \le N.
\]
A group is \textit{acylindrically hyperbolic} if it admits a non-elementary acylindrical action on a hyperbolic space, that is, an acylindrical action that contains two independent loxodromic isometries.

\begin{lem} \label{lem:acyl_action_on_CS} If $(G,\mf S)$ is a (relative) HHG whose top-level coordinate space $X\coloneq \mc CS$ is hyperbolic, then $G$ acts acylindrically on $X$. As a consequence, if $G$ is not virtually cyclic and $X$ is unbounded, then $G$ is acylindrically hyperbolic, and we say $(G,\mf S)$ is an \emph{acylindrically hyperbolic (relative) HHG}.
\end{lem}
\begin{proof}
    This is  \cite[Corollary 14.4]{BHS_HHSI}, which is stated for HHGs but whose proof does not use the hyperbolicity of non-maximal elements of $\mf S$.
\end{proof}

\noindent The following lemma follows immediately from the distance formula; see also \cite{ABD} in the case of a HHG.
\begin{lem}\label{lem:Dbddprojs}
    Let $(G,\mf S)$ be a relative HHG and suppose $H\le G$ is $(\lambda,c)$--quasi-isometrically embedded by the orbit map $G \to \mc CS$. There is a constant $\aleph=\aleph(\mf S, \lambda, c)$  such that the diameter of $\pi_U(H)$ is at most $\aleph$ for all $U\in \mf S-\{S\}$.  
\end{lem}

\subsection{Statement of the main result}
Given a metric space $X$ and a group $G$, we say that an isometric action $G\curvearrowright X$ is \emph{geometric} if it is cobounded and \emph{metrically proper}, i.e., for every $x\in X$ and every $R>0$ the set $\{g\in G\mid \dist(x,gx)\le R\}$ is finite. For the rest of the section, we shall work under the following strengthening of \Cref{hyp:spinning}:
\begin{hyp}\label{hyp:complete} Let $G$ be a (relative) HHG, with top-level coordinate space $X$ and HHG constant $E>0$. Assume that $(X,\mc Y,G,\{H_Y\}_{Y\in \mc Y})$ satisfy \Cref{hyp:spinning} with respect to $(E,K,M_0,R,L)$. Notice that we can always choose $R=0$ by \Cref{rem:X_is_CayG}. Furthermore, suppose that:
\begin{enumerate}[start=7]
    \item \label{I:finitelymany_GY} $G$ acts cofinitely on $\mc Y$;
    \item \label{I:stabilizers} For each $Y\in \mc Y$, the subgroup $H_Y$ acts geometrically on $Y$; and
    \item\label{I:def_of_widetildeL} $L>\widetilde{L}$, where $\widetilde{L}=\widetilde{L}(E,K,M_0)$ is defined in \Cref{eq:tildeL} below and is bounded linearly in $M_0$. 
\end{enumerate}
 We say that the collection $(X,\mc Y,G,\{H_Y\}_{Y\in \mc Y})$ satisfies \Cref{hyp:complete} with respect to $(E,K,M_0,L)$. We drop the constants when unimportant or clear from context. Let $N\coloneq \llangle H_Y\rrangle_{Y\in \mc Y}$ be the normal subgroup generated by the spinning family. 
\end{hyp}

\noindent Our main result is a technical formulation of \Cref{thm:spinning_quotient} that quotients of (relative) HHGs by subgroups forming a sufficiently spinning family are (relative) HHGs. 

\begin{thm}\label{thm:quotientishhg} If $G$ is a (relative) HHG such that $(X,\mc Y,G,\{H_Y\}_{Y\in \mc Y})$ satisfies \Cref{hyp:complete}, then $G/N$ is a (relative) HHG.
\end{thm}

\begin{remark}
    Notice that \Cref{thm:quotientishhg} holds trivially when $G$ is \emph{not} an acylindrically hyperbolic (relative) HHG. Indeed, by \Cref{lem:acyl_action_on_CS}, either such a $G$ is virtually cyclic, or its top-level coordinate space $X$ is bounded. In the first case, any quotient of $G$ is still virtually cyclic, hence hierarchically hyperbolic. In the second case, since $\widetilde{L}\ge E$ by \Cref{eq:tildeL}, and since we are assuming that $E\ge \diam (X)$ by \Cref{rem:X_is_CayG}, then the only $L$-spinning family of subgroups of $G$ is the trivial family, so that $N=\{1\}$ and $G/N=G$. In light of this, the bulk of work is to deal with the case that $X$ is non-elementary hyperbolic, which we shall assume for the remainder of the section.

    Assuming that $G$ is acylindrically hyperbolic does not guarantee that the quotient is again acylindrically hyperbolic. For example, if $G$ is a surface group, $X$ is a Cayley graph of $G$, $Y=X$, and $H$ is a normal subgroup of sufficiently large finite index, then $(X,Y,G,H)$ satisfies \Cref{hyp:complete}, but the quotient is finite. However, under the additional assumption that there exist two independent loxodromics whose axes have uniformly bounded projections to all $Y\in \mc Y$, \Cref{prop:independent_in_quotient} will ensure that $G/N$ is again an acylindrically hyperbolic HHG. We note that this condition is equivalent to $G$ having two independent loxodromic elements for the action on $G\curvearrowright \hat X$, where $\hat X$, defined in the next subsection, is a slight modification of the cone-off from \Cref{defn:coneoff_for_graph}.
    \end{remark}

\subsection{A modified cone-off}\label{sec:basicassumption} 
In describing the quotient hierarchy structure, we will use both projections coming from the relative HHS structure and the projection complex structure described in \Cref{cor:proj_complex_with_points}.  
It will be convenient to first modify the space $X$ by a quasi-isometry that introduces new cone points for each domain $U \in \mf S - \{S\}$ that will serve as an anchor for canonically defining projections and the quotient action.
We first set
\begin{equation}\label{eqn:A}
    A\coloneq  \max\{K,3E\}+4E+D
\end{equation}
where $D=D(E,\max\{K,3E\})$ is defined as in \Cref{lem:DboundSpriano}. Let $X'$ be formed from $X$ by coning off the collection of subspaces $\{\mc N_A(\rho^U_S)\mid U\in \mf S-\{S\}\}$ such that the cone point over $\mc N_A(\rho^U_S)$ is $v_U$. Then $G$ acts acylindrically, $1$-coboundedly, and by isometries on $X'$. 

Furthermore, the natural inclusion map $X\to X'$ is a uniform quality quasi-isometry, since the sets we are coning off have diameter which is bounded in terms of $E$ and $K$; in particular $X'$ is $E'(E,K)$--hyperbolic. For the same reason, each $Y\in \mc Y$, identified with its image under the inclusion, is $K'(E,K)$--quasiconvex in $X'$, so there is a projection $\pi'_Y\colon X'\to Y$ which differs from $\pi_Y\colon X\to Y$ by a uniformly bounded amount, depending on $E$ and $K$. As a consequence, using that $\pi'_Y$ is uniformly Lipschitz with respect to $E$ and $K$, for every $U\in \mf S$, the set $\pi'_Y(v_U)$ is within uniformly bounded Hausdorff distance from $\pi_Y(\rho^U_S)$. We thus expand the domain of the projection $\pi_Y$ to the whole $X'$ by setting $$\pi_Y(v_U)\coloneq \pi_Y(\rho^U_S).$$
Notice also that $\mc Y$ is $M_0'$--geometrically separated in $ X'$, for some $M_0'$ which differs from $M_0$ by a uniform amount; in particular, $M_0'$ is bounded linearly in $M_0$. Hence $(X',\mc Y)$ again satisfies \Cref{hyp:metric} with respect to the constants $(E',K',M_0')$. With a little abuse of notation, we once and for all replace $E$ with the maximum of $E'$ and the original $E$, and similarly for $K$ and $M_0$, so that $(X',\mc Y)$ satisfies \Cref{hyp:metric} with respect to $(E,K,M_0)$. This allows us to define the constants $B,J$, etc. from \Cref{sec:HypBackground}, all of which are bounded linearly in $M_0'$ and therefore in $M_0$.

Let $\hat X$ be formed by coning off $\mc Y$ in $X'$ (not in $X$), and let $v_{Y}$ denote the cone points over $Y\in \mc Y$. Then $\hat X$ satisfies the strong bounded geodesic image property: applying \Cref{lem:strongBGI} in $X'$, there is a constant $C' = C'(E,K,M_0)$ which is bounded linearly in $M_0$, and such that if $x,y\in X'$ satisfy $\dist^{\pi'}_Y(x,y)>C'$, where the projection to $Y$ is measured in $X'$, then any $\hat X$ geodesic $[x,y]$ passes through the cone point $v_Y$. However, as the projections in $X'$ and $X$ differ by a uniform amount (depending on $E$ and $K$), we immediately obtain the following version of strong bounded geodesic image property, no longer mentioning $X'$ and $\pi_Y'$:

\begin{lem}\label{lem:SBGIinHHS}
There is a constant $C=C(E, K,M_0)$ which is bounded linearly in $M_0$, such that for every $Y\in \mc Y$ and $x,y\in \hat X-\{v_Y\}$, if $\dist^\pi_Y(x,y)\geq C$ (where the projection distance is measured in $X$), then any $\hat X$--geodesic $[x,y]$ passes through the cone point $v_Y$.
\end{lem}
With a slight abuse of notation we still call the above constant $C$, as the one from \Cref{lem:strongBGI}, since we can just take $C$ to be the maximum of the two constants and ensure that both lemmas hold.

Notice that the collection $\{H_Y\}_{Y\in \mc Y}$ is still an $L'$--spinning family on $X'$ (with respect to the original projections $\pi_Y$), for some constant $L'$ which differs from $L$ by a bounded amount $\Sha=\Sha(E,K)$. In particular, there exists some constant $L_{1}(E,K,M_0)$, which is bounded linearly in $M_0$, such that if $L>L_1$ then $L'$ satisfies \Cref{hyp:spinning}, so that all the consequences from \Cref{sec:spinningQuotient} still hold for $X'$ and $\hat X$ (with respect to the original projections $\pi_Y$).  Set
\begin{equation}\label{eq:tildeL}
    \widetilde{L}=\max\{L_1,100C+\Sha, 20(C+EJ)+\Sha\},
\end{equation} which is still bounded linearly in $M_0$. From now on we assume that $L>\widetilde{L}$, and in particular
\begin{equation}\label{eq:bound_on_L'} 
    L'\ge L-\Sha>\max\{100C, 20(C+EJ)\}.
\end{equation}

\noindent We conclude this subsection with some remarks about the construction.
\begin{remark}[Projection to domains]
    Recall that, by \Cref{rem:X_is_CayG}, we are assuming that the vertex set of $X$ coincides with $G$. In particular, for every $x\in X$, the projection $\pi_U(x)$ is well-defined for every $U\in \mf S$. 
\end{remark}

\begin{remark}[$H_Y$--orbits have bounded projection]\label{rem:YtoU_is_bounded}
    Since each $H_Y$ acts geometrically on the corresponding $Y\subseteq X$, and since $G$ acts on $\mc Y$ with finitely many orbits, \Cref{lem:Dbddprojs} implies the existence of a constant $\aleph$ such that  $\diam(\pi_U(H_Y\cdot y))\le \aleph$ for every $Y\in \mc Y$, every $y\in Y$, and every $U\in \mf S-\{S\}$.
\end{remark}

\begin{remark}\label{rem:newsprianobound}
    For every $U\in \mf S-\{S\}$, the set $\mc N_A(\rho^U_S)$ is $3E$--quasiconvex by \Cref{lem:n_A_unif_qc}, regardless of the value of $A$. Hence $\hat X$ can be seen as the cone-off of $X$ with respect to the family $\mc H=\{\mc N_A(\rho^U_S)\}_{U\in \mf S-\{S\}}\cup \mc Y$, whose elements are $\max\{K,3E\}$-quasiconvex. Hence \Cref{lem:DboundSpriano} implies that, for every $x,y\in X$, every $X$-geodesic $[x,y]$, and every $\hat X$-geodesic $\gamma$ with the same endpoints, we have $[x,y]\subseteq \mc N_D^X(\widetilde \gamma)$, where $\widetilde \gamma$ is the de-electrification of $\gamma$ with respect to $\mc H$ and $D=D(E,\max\{K,3E\})$ is as in the definition of $A$.
\end{remark}

\subsection{Minimal representatives for points in the quotient}\label{sec:minimal}
 The main goal of this subsection is to define canonical lifts of pairs of points in the quotient and to verify that they are well-behaved with respect to projections; see  \Cref{prop:LiftsProjCloseInU} and \Cref{prop:projections_are_welldef}. 
 
 Let $\ol X\coloneq \hat X/N$. The composition $X  \xhookrightarrow{i} \hat X \xrightarrow{q} \ol X$, where $i$ is the inclusion map and $q$ is the quotient map, is 1--Lipschitz as both $i$ and $q$ are. For each point $x\in \hat X$, including the case $x=v_Y,v_U\in \hat X$, let $\ol x\coloneq  q(x)$. For each $U\in \mf S$ we denote its image in $\mf S/N$ by $\ol U$; as we will see in \Cref{sec:HHGStructure}, $\mf S/N$ will be the index set for the (relative) HHG structure on $G/N$. The following is the analogue of \Cref{def:minrep}:

\begin{defn}[Minimal representatives]
    Let $x,y\in X$. We say $\{x,y\}$ are \textit{minimal distance representatives}, or simply \emph{minimal}, if $\dist_{\hat X}(x,y)=\min_{x'\in \ol x,\,y'\in \ol y}\dist_{\hat X}(x',y')$.
    
    If in the above definition we replace $x$ by $v_U$, for some $U\in \mf S-\{S\}$, we say that $\{U,y\}$ are minimal. If we also replace $y$ by $v_V$ for some $V\in \mf S-\{S\}$, we say that $\{U,V\}$ are minimal.
\end{defn}

\noindent Our first goal is to show that if $\{x,U\}$ and $\{x',U\}$ are minimal, then there is a uniform bound on $\dist_U(x,x')$.  We begin with some preparatory lemmas which investigate the uniqueness of certain lifts in $\hat X$.

\begin{lem}\label{lem:orbit_of_edge_is_unique}
    Let $x, y\in \hat X$ be adjacent vertices. For every $x'\in \ol x$ and $y'\in \ol y$ which are adjacent in $\hat X$, there exists $n\in N$ mapping $x$ to $x'$ and $y$ to $y'$. In other words, an edge of $\ol X$ admits a unique $N$--orbit of lifts.
\end{lem}

\begin{proof}
    Up to the action of $N$, we assume that $y=y'$. Let $h\in N$ be such that $x'=hx$, and let $\gamma$ be the path $x,y,x'$. Suppose that $hx\neq x$ and let $(Y,h_Y)$ be a shortening pair, as in \Cref{prop:new_shortening_pair}, so that $L'/10<\dist^\pi_Y(x,x')$. If $v_Y\not\in\{x,y,x'\}$ then using the bounded geodesic image \Cref{lem:SBGIinHHS}, we have 
    $$L'/10<\dist^\pi_Y(x,x')\le \dist^\pi_Y(x,y)+\dist^\pi_Y(y,x')\le 2C.$$
     However this contradicts the lower bound on $L'$ from \Cref{eq:bound_on_L'}. Hence it must be that $v_Y\in\{x,y,x'\}$. If we apply $h_Y$ to all vertices of $\gamma$ between $v_Y$ and $x'$, we obtain a new configuration of $N$-translates of $\{x,y\}$ and $\{y,x'\}$,  where the new $x$ and $x'$ differ by $h_Yh\prec h$.  We conclude by induction on the complexity of $h$. 
\end{proof}

\begin{lem}\label{lem:bgi_with_U}
    Let $x,x'\in X$, let $\gamma$ be an $\hat X$--geodesic between them, and let $V\in \mf S-\{S\}$ be such that $\dist_V(x,x')>E$. Then $\dist_{\hat X}(v_V,\gamma)\le 2$.
\end{lem}
\begin{proof}
    Since $\dist_V(x,x')> E$, the bounded geodesic image in $(G, \mf S)$ \Cref{defn:HHS}.\ref{axiom:bounded_geodesic_image} implies that, for every $X$--geodesic $\alpha$ from $x$ to $x'$, there exists $w\in \alpha$ with $\dist_X(w, \rho^V_S)\le E$. In turn, by \Cref{rem:newsprianobound}, $\alpha$ is contained in the $D$--neighborhood of $\widetilde \gamma$ the de-electrification of $\gamma$, so we can find $z'\in \widetilde \gamma$ such that $\dist_X(z', \rho^V_S)\le E+D$. If $z'$ belongs to $\gamma\cap \widetilde {\gamma}$, then there is an edge $[z',v_V]$ in $\hat X$ connecting $z'$ to $v_V$, since we chose $A\ge D+E$ in \Cref{eqn:A}. If instead $z' \not\in \gamma$, then $z'$ lies along a de-electrified segment for some subspace $Z \in \mc H$ with $v_Z\in \gamma$; here $v_Z=v_{W}$ if $Z=\mc N_A(\rho^W_S)$ for some $W\in \mf S-\{S\}$.  Since $Z$ is $\max\{K,3E\}$--quasiconvex, there exists $z'' \in Z$ with $\dist_X(z',z'')\le \max\{K,3E\}$. The situation is depicted in \Cref{fig:ShorterPath}, and again by our choice of $A$ in \Cref{eqn:A} there are edges $[v_Z,z'']$ and $[z'',v_V]$ in $\hat X$.
\end{proof}
\begin{figure}[htp]
        \centering
        \begin{overpic}[width=2in, trim={2in 4.5in 4.25in 4in}, clip]{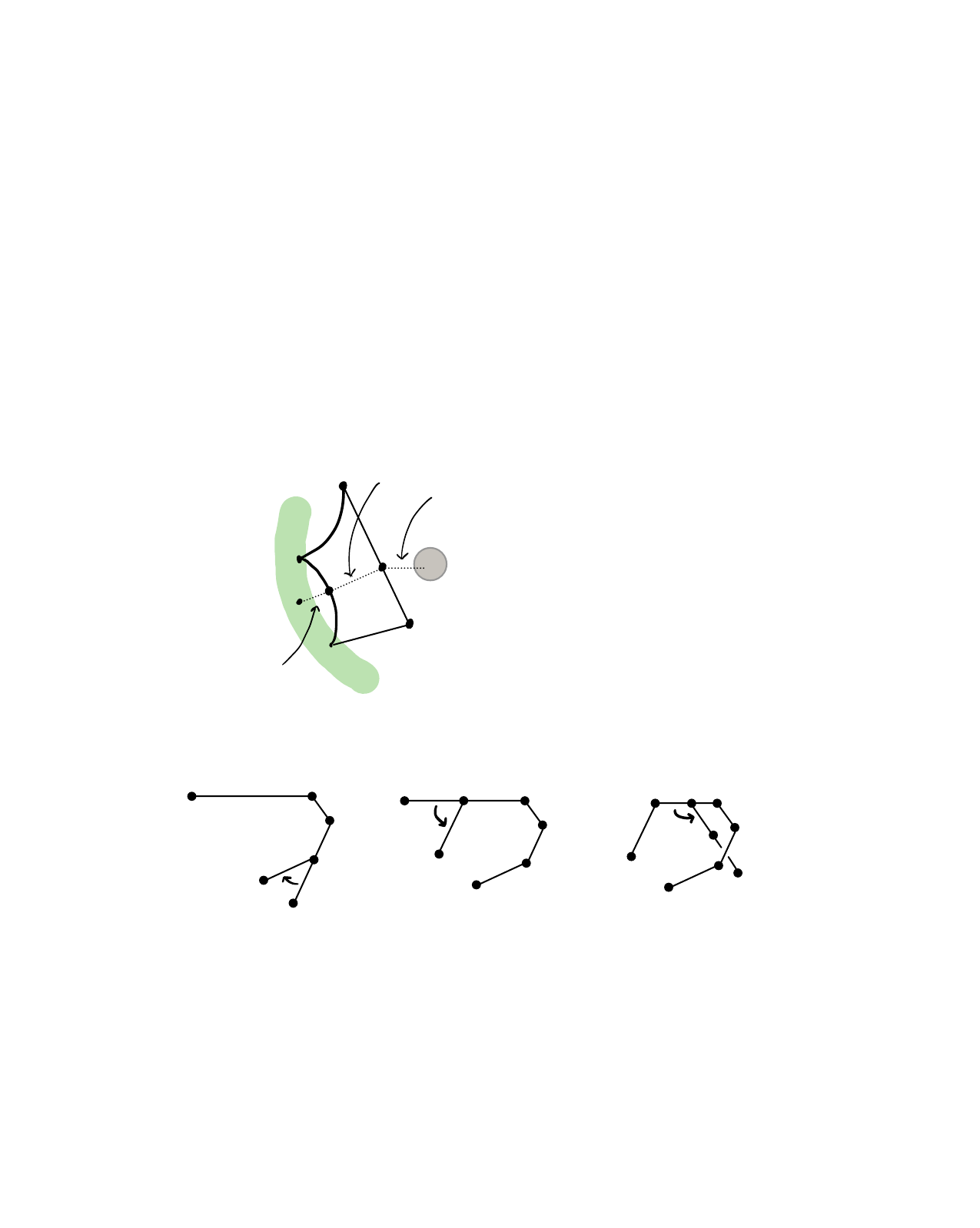}
            \put(42,90){$x$}
            \put(10,70){{\color{ForestGreen}$Z$}}
            \put(19,39){$z''$}
            \put(41,42){$z'$}
            \put(54,48){$w$}
            \put(81,47){$\rho^V_S$}
            \put(71,31){$x'$}
            \put(34,70){$\widetilde\gamma$}
            \put(-2,10){$\le \max\{K, 3E\}$}
            \put(59,89){$\le D$}
            \put(79,80){$\le E$}
            \put(56,63){$\alpha$}
        \end{overpic}
        \caption{Proof of \Cref{lem:bgi_with_U}, in the case $z'$ does not lie on $\gamma$.}
        \label{fig:ShorterPath}
    \end{figure}
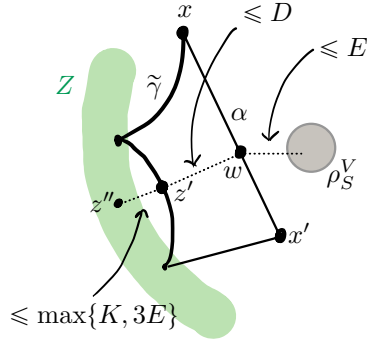

\begin{lem}\label{lem:Unotnearleg_NEW}
    Let $x\in G$ and $U\in \mf S-\{S\}$ be such that $\{x,U\}$ are minimal, and let $V\in \mf S-\{S\}$ be such that $\dist_X(\rho^U_S,\rho^V_S)\le 2E$; we include the case  $U=V$. Let $Y\in \mc Y$ be such that $v_Y$ lies on a geodesic $\gamma$ from $x$ to $v_U$ in $\hat X$, and let $x'\in \gamma$ be the vertex of $\gamma$ before $v_Y$. Then $\dist_V(x,x')\le E$. 
\end{lem}
\begin{proof} Apply \Cref{lem:bgi_with_U} to $x, x'$, and the subsegment of $\gamma$ between them. As in the proof, this produces some $z$, belonging to either $\gamma \cap \widetilde  \gamma$ or some $Z\in\mc H$ where $v_Z\in  \gamma$, such that $\dist_X(z, \rho^V_S)\le D+E+\max\{K,3E\}$. By assumption, we have $\dist_X(\rho^U_S, \rho^V_S)\le 2E$, so $\dist_X(z, \rho^U_S)\le D+4E+\max\{K,3E\}\le A$. As above, it follows that $v_U$ is in the $2$-neighborhood of the subsegment of $\gamma$ between $x$ and $x'$. However, the subpath of $\gamma$ from $x'$ to $v_U$ has length at least $3$, since it must contain $v_Y$ and $v_U$, which are distance at least $2$ apart. Hence we contradict the fact that $\gamma$ is a geodesic between $x$ and $v_U$. 
\end{proof}

\begin{prop}\label{prop:LiftsProjCloseInU}
    Let $\ol x\in G/N$, $x,x' \in \ol{x}$ be distinct representatives, and $U\in \mf S-\{S\}$ be such that $\{x,U\}$ and $\{x',U\}$ are minimal, and let $V\in \mf S-\{S\}$ be such that $\dist_X(\rho^U_S,\rho^V_S)\le 2E$, where we allow  $U=V$. Then $\dist_V(x, x') \le 2\aleph+9E$, where $\aleph$ is the constant from \Cref{rem:YtoU_is_bounded}.
\end{prop}

\begin{proof}
    Let $\gamma, \gamma'$ be $\hat X$--geodesics from $x$ and $x'$ to $v_U$, respectively.  Since $x,x'\in \ol x$, there is some $h\in N$ such that $hx=x'\neq x$. Let $x_0=x$ and $x_0'=x'$, and similarly let $\gamma_0=\gamma$ and $\gamma_0'=\gamma'$. By \Cref{prop:new_shortening_pair} there exists a shortening pair $(Y,h_Y)$ such that $h_Yh\prec h$ and $\dist^\pi_Y(x_0,x_0')>L'/10$. By the triangle inequality, one of $\dist^\pi_Y(x_0,v_U)$ and $\dist^\pi_Y(v_U, x_0')$ is at least $L'/20$; we focus on the case $\dist^\pi_Y(x_0,v_U)>L'/20$, as the other is dealt with analogously. Since $L'/20$ is greater than the constant $C$ from \Cref{lem:SBGIinHHS}, the cone point $v_Y$ lies on $\gamma_0$. Now bend $\gamma_0$ at $v_Y$ by $h_Y^{-1}$, and call this $\gamma_1$ with endpoint $x_1=h_Y^{-1}x_0$. Then set $\gamma_1'=\gamma_0'$ and $x_1'=x_0'$. Notice that  $x_1$ and $x_1'$ differ by $hh_Y$, which is an $N$-conjugate of $h_Yh$ and so still has  complexity strictly less than that of $h$. Moreover, both geodesics still have $v_U$ as an endpoint.
    
    Repeat the argument with $x_1, x_1'$ and the associated geodesics. We proceed inductively until $x_k=x_k'$ for some $k\in \mathbb N$; we call this point $x''$. In other words, we eventually produce two geodesics $\gamma_k$ and $\gamma_k'$, both with endpoints $v_U$ and $x''\in \ol x$. Notice that $x''$ and $v_U$ are minimal, as their distance is still the length of $\gamma$.  Since $\gamma_k$ is obtained by successively bending $\gamma$ while fixing the endpoint $v_U$, there exists $v_Y\in \gamma\cap\gamma_k$. Moreover, if $z\in \gamma$ is the vertex before $v_Y$, then there exists $n\in N$ such that the vertex of $\gamma_k$ before $v_Y$ is $nz$. By \Cref{lem:orbit_of_edge_is_unique}, we can actually choose $n\in \Stab_G{v_Y}$, and the latter is $H_Y$ by \Cref{cor:stabilisers_in_n}; hence $\dist_V(z, nz)\le \aleph$ by \Cref{rem:YtoU_is_bounded}. The situation is therefore as in \Cref{fig:minimals_project_close}.
    
    \begin{figure}[htp]
        \centering
        \includegraphics[width=4in]{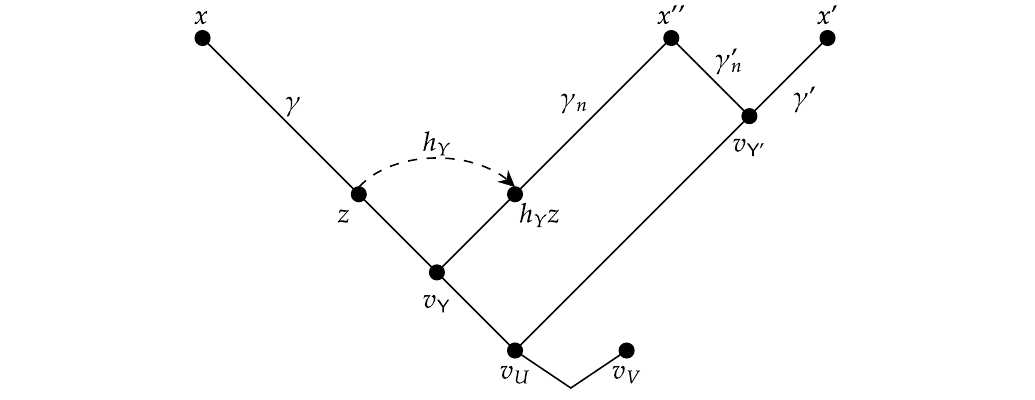}
        \caption{The geodesics from the proof of \Cref{prop:LiftsProjCloseInU}. After bending $\gamma$ and $\gamma'$ a finite number of times, their endpoints coincide. Notice that $\gamma$ and $\gamma_k$ must overlap on the segment of $\gamma$ connecting $v_U$ to the first vertex of the form $v_Y$ for some $Y\in \mc Y$, as bending can occur only at cone points.}
        \label{fig:minimals_project_close}
    \end{figure}
    
    By \Cref{lem:Unotnearleg_NEW} applied to the minimal pairs $\{x,U\}$ and $\{x'',U\}$, $V$, and $Y$, we obtain
    \[
    \dist_V(x,x'') \le \dist_V(x,z) + \diam\pi_V(z)+ \dist_V(z,nz) + \diam\pi_V(nz)+\dist_V(nz, x'') \le \aleph+4E.
    \]
    Symmetrically, we obtain $\dist_V(x'',x')\le \aleph+4E$, completing the proof as $\diam\pi_V(x'')\le E$.
\end{proof}

\noindent A similar statement holds for pairs of minimal domains:

\begin{prop}\label{prop:projections_are_welldef}
    Let $V\in \mf S-\{S\}$ and $U,U'\in\ol U\in \mf S/N-\{\ol S\}$ be such that $\{U,V\}$ and $\{U',V\}$ are minimal.
    \begin{itemize}
        \item If $U$ and $V$ are not transverse then $U=U'$.
        \item Otherwise $\dist_V(\rho^U_V, \rho^{U'}_V)\le 2\aleph+25E.$
    \end{itemize}
\end{prop}

\begin{proof}
    Suppose first that $\dist_{\hat X}(v_U,v_V)=2$. This occurs, in particular, if $U$ and $V$ are not transverse, since $\dist_{X}(\rho^U_S,\rho^V_S)\le 2E\le 2A$ by \Cref{lem:close_proj_of_orthogonals}. Then the union of any two geodesics $[v_U,v_V]\cup [v_V,v_{U'}]$ does not contain any cone point $v_Y$ for $Y\in \mc Y$, as $v_V$ is only $\hat X$-adjacent to points in $X$. Towards a contradiction, suppose that $U'\neq U$. Let $h\in N$ be such that $U'=hU$, and let $(Y,h_Y)$ be a shortening pair, as in \Cref{prop:new_shortening_pair}. Since the geodesics $[v_U,v_V]$ and $[v_V,v_{U'}]$ do not pass through $v_Y$, the diameter of the $Y$-projection of each geodesic is at most $C$ by \Cref{lem:SBGIinHHS}.  Thus $L'/10\le 2C$, contradicting \Cref{eq:bound_on_L'}. This proves the first bullet point.

For the second bullet, the above argument lets us assume that $\dist_{\hat X}(v_U,v_V)\ge 3$, which in particular means that $\dist_X(\rho^U_S, \rho^V_S)>2A$. First, we prove that $\diam \pi_V(\mc N_A(\rho^U_S))\le 5E$. To see this, let $p\in \mc N_A(\rho^U_S)$, and fix $q\in \rho^U_S$, so that $\dist(p,q)\le A+E$. If a geodesic $[p,q]$ in $X$ passed $E$-close to $\rho^V_S$, then $\dist_X(\rho^U_S, \rho^V_S)\le A+2E$, and this is at most $2A$ by \Cref{eqn:A}, contradicting our assumption. Thus the bounded geodesic image axiom for $\mf S$ yields that $\dist_V(p,q)\le E$. Hence $\pi_V(p)\subseteq \mc N_{2E}(\pi_V(q))$, from which the claim follows.

 Now let $\gamma$ and $\lambda$ be geodesics connecting $v_U$ (resp. $v_{U'}$) to $v_V$, and let $p$ (resp. $q$) be the first point of $\gamma$ after $v_U$ (resp. the first point of $\lambda$ after $v_{U'}$). 
 Since $U, U' \in \ol{U}$, we may shorten the concatenation as in \Cref{prop:LiftsProjCloseInU} to produce two geodesics, $\gamma''$ and $\lambda''$, each joining $v_V$ to a single point $v_{U''}\in \ol{v_U}$, and such that $\gamma''$ is obtained by bending $\gamma$ while $\lambda''$ is obtained by bending $\lambda$. If $p''$ is the first point of $\gamma''$ after $v_{U''}$, then $p''\in \ol p$, and, as in the proof of \Cref{prop:LiftsProjCloseInU}, we obtain $\dist_V(p,p'')\le \aleph+4E$. 
 
 Now let $r\in X$ be such that $\dist_X(r,\rho^U_S)\le E$ and $\dist_V(r,\rho^U_V)\le E$, which exists by the partial realization axiom \Cref{defn:HHS}.\ref{axiom:partial_realisation}. Since $r,p\in \mc N_A(\rho^U_S)$, we have that
 \[\dist_V(\rho^U_V,p'')\le E+\dist_V (r,p'')\le 
 E+\diam\pi_V(N_A(\rho^U_S))+\dist_V(p,p'')\le \aleph+10E.\] 
 Symmetrically, there exists $q''\in \mc N_A(\rho^{U''}_S)$ such that $\dist_V(\rho^{U'}_V,q'')\le \aleph+10E$, thus concluding the proof as $\diam\pi_V(\mc N_A(\rho^{U''}_S))\le 5E$.
\end{proof}

\subsection{The quotient structure and projections}\label{sec:HHGStructure}
We are now ready to define the relative HHG structure on $G/N$.  Recall that, by \Cref{rem:X_is_CayG}, we are assuming that $X$ is a Cayley graph for $G$. In particular, if we fix a word metric $\dist_G$ on $G$ with respect to a finite generating set, the projection map $\pi_S\colon G\to X$ is bijective at the level of vertices.

\begin{construction} \label{constr:HHGStructureQuotient} Fix a finite generating set $\mc T$ for $G$ which induces a word metric $\dist_G$, let $\ol{\mc T}$ be its image in $G/N$, and let $\dist_{G/N}$ be the word metric on $G/N$ induced by $\ol {\mc T}$. The relative HHG structure on $G/N$ has the following components.
\begin{itemize}
    \item \textit{Index set:} $\qs$.  Recall that we denote the image of $U\in \mf S$ by $\ol U\in \mf S/N$.
    
    \item \textit{Hyperbolic spaces:} The top-level space $\mc C \ol S$ is $\ol X$ as defined in \Cref{sec:minimal}. For each $\ol U\in \mf S/N -\{\ol S\}$, we set 
    $$\mc C \ol U=\left(\bigcup_{U\in \ol U}\mc C U\right) /N.$$ 
    Notice that, for every $U\in \ol U$, the projection map $\mc C U\to \mc C \ol U$ is an isometry. Indeed, by \Cref{cor:tau} $N$ acts freely on $X'$, and in particular no non-trivial $n\in N$ can fix $v_U$ (hence $U$).  

     \item \textit{Relations:} Every domain in $\qs$ nests into $\ol  S$.  The relation between $\ol U,\ol V\in \mf S/N - \{\ol S\}$ is the same as the relation between any minimal pair $\{U,V\}$ with $U\in \ol U$ and $V\in \ol V$.
     In particular, $\ol U$ is $\nest$-minimal in $\mf S/N$ if and only if every representative $U\in \ol U$ is $\nest$-minimal in $\mf S$.
    
     \item \textit{Projection maps:} We identify $G$ with $X$ via the projection $\pi_S$, and we consider it as a subspace of $\hat X$ via the inclusion $i\colon X\hookrightarrow \hat X$. This way, $G/N$ can be seen as a subgraph of $\ol X$. For every $\ol g\in G/N$ define $\pi_{\ol S}\colon G/N \to \ol X$ by $\pi_{\ol S} (\ol g)\coloneq \ol{g}$; moreover, given $\ol  U\in \qs-\{S\}$ set 
     \[
      \pi_{\ol  U}(\ol  g)\coloneq  \left( \bigsqcup_{\substack{g\in \ol {g},\, U\in \ol U \\\{U,g\} \textrm{ minimal}}} \pi_U(g)\right)/N.
     \]

    \item \textit{Relative projections:}  Set $\rho^{\ol  V}_{\ol  S}=\ol{v_V}$ for some (equivalently, any) $V\in\ol V$. Furthermore, given $\ol  U, \ol  V\in \mf S/N-\{\ol S\}$ with $\ol  U\trans \ol  V$ or $\ol  V\propnest \ol  U$, define  
        \[
            \rho^{\ol  V}_{\ol  U} \coloneq  \left( \bigsqcup_{\substack{U\in \ol U, V\in\ol V \\\{U,V\} \textrm{ minimal}}} \rho^V_U\right)/N.
        \]
\end{itemize}
\end{construction}

\begin{remark}
\label{rem:closeProj} Let $U\propnest S$ and $g\in G$ be minimal. An immediate consequence of how the projections are defined is that $\pi_U(g) \subseteq \pi_{\ol{U}}(\ol{g})$. Similarly, if $U,V\propnest S$ are minimal and non-orthogonal, then $\rho^V_U \subseteq \rho^{\ol{V}}_{\ol{U}}$. These convenient facts are useful for comparing the hierarchical structure for $(G, \mf S)$ and that of $(G/N, \mf S/N)$.
\end{remark}

\begin{remark}\label{rem:UniqueReps}
    Let $\ol U, \ol V\in \mf S/N-\{\ol S\}$, and fix a representative $U\in \ol U$. If there exists $V\in\ol V$ such that $\dist_{X}(\rho^U_S, \rho^V_S)\le 2A$ (for example if $U$ and $V$ are not transverse, by \Cref{lem:close_proj_of_orthogonals}), then there can be at most one such domain $V\in \ol V$, as we argued in the proof of \Cref{prop:projections_are_welldef}. As a consequence, for every other $V'\in \ol V$ we have that $\dist_{X}(\rho^U_S, \rho^{V'}_S)>2A>E$, so $V'\trans U$. This proves that the relations $\nest$ and $\orth$ in $\mf S/N$ are well-defined, as either there is a unique minimal pair including $U$, or every pair $\{U,V\}$ is transverse. This also shows that $\ol U$ and $\ol V$ are orthogonal (resp. nested) if and only if they admit orthogonal (resp. nested) representatives $U$ and $V$, as such representatives are minimal.
\end{remark}

\subsubsection{Lifting minimal configurations}
\noindent Before we prove that the above construction yields a hierarchy structure on $G/N$, we record a few technical lemmas about projections and minimal collections. We first argue that certain configurations admit pairwise minimal representatives.
\begin{lem}[Lifting triples]\label{lem:LiftingTriples}
    Given $\ol x,\ol y,\ol z\in G/N$, there exist representatives $x\in \ol x, y\in \ol y$, and $z\in \ol z$ such that $\{x,y,z\}$ are pairwise minimal. Furthermore, the same holds when any of the elements of $G/N$ are replaced with domains in $\mf S/N-\{\ol S\}$.
\end{lem}

\begin{proof}
   Fix $x\in \ol x$.  Considering $\ol x, \ol y, \ol z$ as vertices of $\ol X$, pick geodesics in $\ol X$ to form a geodesic triangle. By \Cref{prop:liftingQuadrilaterals} we can lift this to a geodesic triangle in $\hat X$ based at $x\in \ol x$.  The vertices of this triangle are pairwise minimal by construction. The same argument holds if we replace any vertex of the triangle in $\ol X$ with the image of the cone point over a domain in $\mf S-\{ S\}$, thus proving the ``furthermore'' part of the statement.
\end{proof}

\begin{lem}\label{lem:lifting_trans_and_points}
    Let $\ol x, \ol y\in G/N$ and $\ol U_1,\ldots, \ol U_k\in \mf S/N-\{\ol S \}$. There exist representatives $\{x,y,U_1,\ldots, U_k\}$ such that $\{x,y,U_i\}$ are pairwise minimal for every $i\le k$.
\end{lem}

\begin{proof}
    We proceed by induction on $k$. The base case $k=1$ is \Cref{lem:LiftingTriples}.  Suppose that $\{x,y,U_1,\ldots, U_{k-1}\}$ are as in the statement. There are $\{x',y',U_k\}$ pairwise minimal representatives by \Cref{rem:UniqueReps}, where $x'\in \ol x$ and $y'\in \ol y$. Up to the action of $N$, we can assume that $x=x'$. Let $h\in N$ map $y$ to $y'$. For every $i\le k-1$, consider a geodesic triangle with vertices $\{x,y, v_{U_i}\}$, and also consider a geodesic triangle with vertices $\{x,y', v_{U_k}\}$, as in \Cref{fig:lifting_trans_and_points}.
    
    \begin{figure}[htp]
        \centering
        \includegraphics[width=4in]{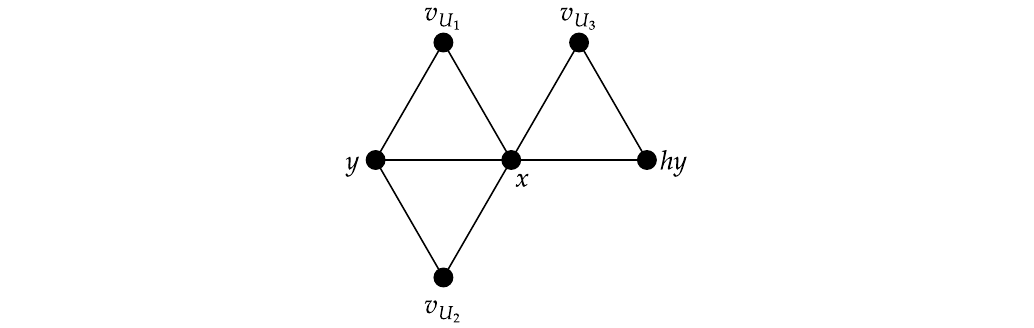}
        \caption{The various lifts from the proof of \Cref{lem:lifting_trans_and_points} (for $k=3$).}
        \label{fig:lifting_trans_and_points}
    \end{figure}
    
    If $hy=y$, we are done. Otherwise, by \Cref{prop:new_shortening_pair} there exists a shortening pair $(Y,h_Y)$, so that $\dist^\pi_Y(y,y')>L'/10$.  By triangle inequality, one of $\dist^\pi_Y(y,x)$ and $\dist^\pi_Y(x,y')$ is at least $L'/20$.

    If $\dist^\pi_Y(y,x)>L'/20$, then for every $i\le k-1$, one of $\dist^\pi_Y(y,v_{U_i})$ and $\dist^\pi_Y(v_{U_i},x)$ is at least $L'/40>C$.  \Cref{lem:SBGIinHHS} then implies that $v_Y$ lies both on every geodesic $[x, v_{U_i}]$ and every geodesic $[x, y]$. In particular, $v_Y$ is a cut vertex of the triangle with vertices $\{x,y, v_{U_i}\}$. For the same reason, if $\dist^\pi_Y(x,y')>L'/20$, then $v_Y$ is a cut vertex of the triangle with vertices $\{x,y', v_{U_k}\}$. In either case, applying $h_Y$ to the every vertex in the configuration from \Cref{fig:lifting_trans_and_points} between $v_Y$ and $y'$ reduces the complexity of $h$, and we conclude by induction.
\end{proof}

\noindent If the given domains are all non-transverse to some domain, the lifts from \Cref{lem:lifting_trans_and_points} can be chosen to be pairwise minimal:
\begin{lem}\label{non_trans_dom_and_two_points_lift}
    Let $\{\ol U_0,\dots, \ol U_k\}$ be a collection of domains in $\mf S/N-\{\ol S\}$ such that $\ol U_i$ and $\ol U_0$ are not transverse for any $i\neq 0$, and let $\ol x, \ol y\in G/N$. Then there exist pairwise minimal representatives $\{x,y,U_0,\ldots, U_k\}$.
\end{lem}

\begin{proof}
    We first prove the existence of pairwise minimal representatives of the domains alone. Fix $U_0\in \ol U_0$. Since $\ol U_i$ is not transverse to $\ol U_0$ for any $i$, by \Cref{rem:UniqueReps} there is a unique representative $U_i\in \ol U_i$ such that $\{U_1,U_0\}$ are minimal, and $\rho^{U_i}_S$ at distance at most $2E$ from $\rho^{U_0}_S$ by \Cref{lem:close_proj_of_orthogonals}. Since $\dist_X(\rho^{U_i}_S,\rho^{U_j}_S)\le 5E\le 2A$, it follows from \Cref{rem:UniqueReps} that $\{U_i,U_j\}$ are minimal for all $i\neq j$. 

    Next, we prove by induction on $k$ that there exist minimal representatives $\{x,U_0,\ldots, U_k\}$. If $k=0$, there is nothing to prove. Assume now that there exist pairwise minimal representatives $\{x,U_0',\ldots, U_{k-1}'\}$, and let $\{U_0, \ldots, U_k\}$ be pairwise minimal representatives, whose existence is guaranteed by the above argument. Up to the action of $N$ we can assume that $U_0=U_0'$, and by \Cref{prop:projections_are_welldef} this implies that $U_i=U_i'$ for every $i\le k-1$ (but possibly $x$ and $U_k$ are not yet minimal). Let $x'\in \ol x$ be minimal with respect to $U_k$. Consider geodesics from $x$ to every $v_{U_i}$ for $i\le k-1$, from $v_{U_i}$ to $v_{U_j}$ for every $i,j$, and from  $v_{U_k}$ to $x'$, as in \Cref{fig:x_and_nontrans}. Notice that each such geodesic lifts a geodesic of $\ol X$, as its endpoints are minimal.
    
    \begin{figure}[htp]
        \centering
        \includegraphics[width=4in]{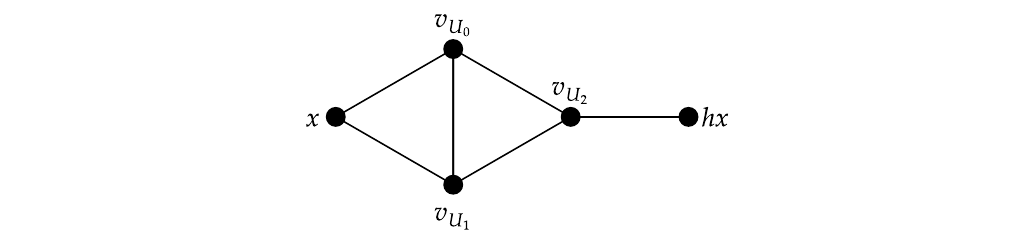}
        \caption{The various lifts from the first part of the proof of \Cref{non_trans_dom_and_two_points_lift} (with $k=2$). Each edge in the picture is a geodesic with whose endpoints are minimal.}
        \label{fig:x_and_nontrans}
    \end{figure}
    
    Let $h\in N$ be such that $hx=x'$. We proceed by induction on the complexity of $h$. If $hx=x$ we are done, as then $x$ and $U_k$ are already minimal. Otherwise, there is $(Y,h_Y)$ a shortening pair by \Cref{prop:new_shortening_pair}. If $\dist^\pi_Y(v_{U_k}, x')>C$ then $v_Y$ lies on the geodesic between $v_{U_k}$ and $x'$. In this case, we bend this geodesic at $v_Y$, and conclude by induction since $h_Yx'$ differs from $x$ by $h_Yh\prec h$. If instead $\dist^\pi_Y(v_{U_k}, x')\le C$, then for every $i\le k-1$ we have that
    \[\dist^\pi_Y(x, v_{U_i})\ge \dist^\pi_Y(x, x')-\dist^\pi_Y(v_{U_i}, v_{U_k})-\dist^\pi_Y(v_{U_k}, x') \ge L'/10-2C>C.\]
    Here, we used that, since $U_i$ and $U_k$ are not transverse, they lie at distance $2$ in $\hat X$, so any $\hat X$--geodesic connecting them belongs to $X'$ and cannot pass through $v_Y$. Hence $v_Y$ lies on the geodesic between $x$ and each $v_{U_i}$, so we bend each of these geodesics at $v_Y$ by applying $h_Y$ to the geodesic connecting $v_Y$ to $x'$. Again $h_Yx'$ differs from $x$ by $h_Yh\prec h$, and we conclude by induction.

    Finally, let $\{x,U_1,\ldots, U_k\}$ and $\{y,U_1',\ldots, U_k'\}$ be pairwise minimal collections, which exist by the above argument. As before, up to $N$-translation we can actually assume that $U_i=U_i'$ for every $i$. Consider geodesics from $x$ to $v_{U_i}$ for every $i$, from $v_{U_i}$ to $v_{U_j}$ for every $i\neq j$, from $v_{U_i}$ to $y$ for every $i$, and from $y$ to some $x'\in \ol X$ such that $\{y,x'\}$ are minimal. The situation is as in \Cref{fig:x_y_and_nontrans}.
    
    \begin{figure}[htp]
        \centering
        \includegraphics[width=4in]{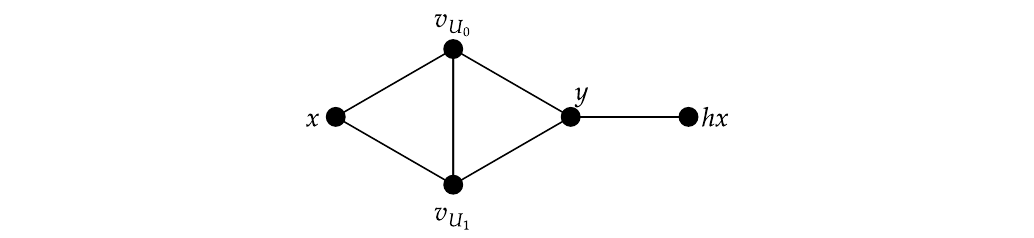}
         \caption{The various lifts from the final part of the proof of \Cref{non_trans_dom_and_two_points_lift} (with $k=1$).}
        \label{fig:x_y_and_nontrans}
    \end{figure}

    The argument to find pairwise minimal representatives is now very similar to the one with $x$ alone. Let $h\in N$ map $x$ to $x'$. If $hx=x$ we are done; otherwise let $(Y,h_Y)$ be a shortening pair. If $\dist^\pi_Y(y, x')>C$ then $v_Y$ lies on the geodesic between $y$ and $x'$; in this case we bend this geodesic at $v_Y$ and conclude by induction. If instead $\dist^\pi_Y(y, x')\le C$ then $\dist^\pi_Y(x,y)\ge L'/10-C>3C$, where we used \Cref{eq:bound_on_L'}. In turn, for every $i,j\le k$ one of $\dist^\pi_Y(x,v_{U_i})$ and $\dist^\pi_Y(v_{U_j},y)$ is greater than $C$, as otherwise we would have that 
    $$\dist^\pi_Y(x,y)\le\dist^\pi_Y(x,v_{U_i})+\dist^\pi_Y(v_{U_i},v_{U_j})+ \dist^\pi_Y(v_{U_j},y)\le 3C.$$
    In particular, this means that either $\dist^\pi_Y(x,v_{U_i})>C$ for every $i$, or the same holds with $y$ replacing $x$. In other words, by \Cref{lem:SBGIinHHS}, $v_Y$ lies on either every geodesic coming out of $x$, or every geodesic ending at $y$. In either case $v_Y$ is a cut point of the configuration from \Cref{fig:x_y_and_nontrans}, so we can apply $h_Y$ to every point between $v_Y$ and $x'$ and again conclude by induction.
\end{proof}

\subsubsection{Bounded projections}
\noindent Our next goal is to provide some bounds on the projections from \Cref{constr:HHGStructureQuotient}. We first show that the new maps $\pi_{*}$ and $\rho^{*}_{*}$ send points to uniformly bounded sets.

\begin{prop}\label{prop:ProjsBdd} Let $\beth=2\aleph+27E$, where $\aleph$ is the constant from \Cref{rem:YtoU_is_bounded}. For any $\ol x\in G/N$ and $\ol U\in \mf S/N$, the projection $\pi_{\ol U}(\ol x)$ has diameter at most $\beth$ in $\mc C\ol U$.  Analogously, if $\ol U,\ol V\in \mf S/N$ satisfy $\ol U\trans \ol V$ or $\ol V\nest \ol U$, then $\rho^{\ol V}_{\ol U}$ has diameter bounded by $\beth$ in $\mc C\ol U$.
\end{prop}

\begin{proof}
    If $\ol U=\ol S$, then $\pi_{\ol S}(\ol x)=\ol x$, seen as a point in $\ol X$, and we have nothing to prove. Thus suppose $\ol U\neq \ol S$. Since being minimal is an $N$--equivariant relation, it suffices to consider a fixed $U\in \ol U$, so that we can identify $\mc C \ol U$ with $\mc C U$.  By \Cref{prop:LiftsProjCloseInU}, if $x,x'\in \ol x$ are both minimal with $U$,  then $\dist_U(x,x')\le 2\aleph+9E$, so $\diam(\pi_U(x)\cup\pi_U(x'))\le 2\aleph+11E\le \beth$.

    For the second statement, if $\ol U=\ol S$ then $\rho^{\ol V}_{\ol S}=\ol {v_V}$ is a point. Otherwise fix $U\in \ol U$, and suppose $\{V',U\}$ and $\{V,U\}$ are both minimal, with $V,V'\in \ol V$. The result follows by applying \Cref{prop:projections_are_welldef} and using that $\rho^V_U$, $\rho^{V'}_U$ have diameter at most $E$.
\end{proof}

\begin{lem}\label{lem:DistInQuotient}
    Let $x,y\in X$, and suppose $U \in \mf S - \{S\}$. If $\{U,x,y\}$ are pairwise minimal, then \[\dist_U(x,y)-2\beth \le \dist_{\ol U}(\ol x, \ol y)\le  \dist_U(x,y) \quad\mbox{ and }\quad
    \dist_{\ol X}(\ol x, \ol y)\le \dist_X(x,y).
    \]
     Moreover, if $\ol V \in \mf S / N-\{\ol S\}$ satisfies $\ol U\trans \ol V$ or $\ol V\propnest \ol U$ and  $\{x,U,V\}$ are pairwise minimal representatives, then
    \[
    \dist_U(x,\rho^V_U) -2\beth\le \dist_{\ol U}(\ol x, \rho^{\ol V}_{\ol U})\le \dist_U(x,\rho^V_U)
    \quad\mbox{ and }\quad
    \dist_{\ol X}(\ol x,\ol{v_V})\le \dist_X(x,\rho^V_S) + 1.
    \]
\end{lem}

\begin{proof}
    For any pair $\{x,y\}$, we have that $\dist_{\ol X}(\ol x, \ol y) \le \dist_{\hat X}(x,y)\le \dist_X(x,y)$.  When $y\in \ol y$ is replaced with $V\in \ol V\neq \ol S$, we have $\dist_{\ol X}(\ol x,\ol{v_V})\le \dist_{\hat X}(x,v_V) \le \dist_X(x,\rho^V_S) + 1$.

    Furthermore, by definition of the projections $\pi_{\ol U}$ and the fact that $\mc C\ol U$ is identified with $\mc CU$, we have $\pi_{\ol U}(\ol x)\supseteq \pi_U(x)$ by \Cref{rem:closeProj}, so
    $\dist_U(x,y)\ge \dist_{\ol U}(\ol x,\ol y)$. On the other hand, by \Cref{prop:ProjsBdd}, $\pi_{\ol U}(\ol x)$ and $\pi_{\ol U}(\ol y)$ have diameter bounded by $\beth$, so $\dist_{\ol U}(\ol x, \ol y)\ge \dist_{U}(x, y)-2\beth$. Analogously, $\rho^{\ol V}_{\ol U}$ contains $\rho^V_U$ and has diameter at most $\beth$, so the second statement follows analogously.
\end{proof}

Finally, we prove that if $x\in G$ and  $U\in \mf S$ are ``almost'' minimal, then the projection of $x$ to $U$ is almost the projection of a minimal element:

\begin{lem}\label{lem:almost_minimal_proj}
    Let $x\in X$, $y\in \hat X$ and $U\in \mf S-\{S\}$. Suppose that $\{x,y\}$ are minimal and $\dist_{\hat X}(v_U,y)\le 2$. Then there exists $x^*\in \ol x$ such that $\{x^*,U\}$ are minimal and $$\dist_U(x,x^*)\le 9\aleph+28E.$$
\end{lem}

\begin{proof} If $\{x,U\}$ are already minimal, there is nothing to prove, so we henceforth assume the contrary. Let $\gamma=[y,x]$ be an $\hat X$--geodesic, and up to replacing $y$ by some point on $\gamma$, we can assume that $\gamma\cap \mc N_2(v_U)=\{y\}$ (notice that $\dist_{\hat X}(x,v_U)\ge 2$, or they would be minimal). If $\eta=[v_U,y]$ is an $\hat X$--geodesic, then $y$ is the only point of $\eta$ which can be of the form $v_Y$ for some $Y\in \mc Y$, because the link of $v_U$ in $\hat X$ belongs to $X$. 

We now describe an algorithm to find the required $x^*\in \ol x$. To initialize the procedure, set $\gamma_0=\gamma$ and $x_0=x$; furthermore, let $\widetilde  x_0\in \ol x$ be minimal with $U$, let $\sigma_0$ be an $\hat X$-geodesic connecting $v_U$ to $\widetilde  x_0$, and let $h_0\in N$ map $x_0$ to $\widetilde  x_0$. 

Suppose we are given geodesics $\gamma_i=[y, x_i]$ and $\sigma_i=[v_U,\widetilde  x_i]$, where $x_i, \widetilde  x_i\in \ol x$ and the endpoints of both geodesics are minimal. Suppose  $h_i\in N$ maps $x_i$ to $\widetilde  x_i$. If $\widetilde  x_i=x_i$, then  set $x^*=x_i$ and  stop the algorithm. Otherwise, let $(Y,h_Y)$ be a shortening pair, as in \Cref{prop:new_shortening_pair}. Then $v_Y$ must lie on one of $\sigma_i$ or $\gamma_i$ since $L'/10>3C$; notice that it cannot lie in the interior of $\eta$, which is a single point in $X$. 
\begin{enumerate}
    \item If $v_Y\in \sigma_i$, let $\widetilde  x_{i+1}=h_Y\widetilde  x_i$, and let $\sigma_{i+1}=[v_U,\widetilde  x_{i+1}]$ be obtained by bending $\sigma_i$ at $v_Y$ by $h_Y$. Set $\gamma_{i+1}=\gamma_i$, $x_{i+1}=x_i$, and $h_{i+1}=h_Yh_i\prec h_i$. We now repeat the procedure with the data indexed by $i+1$.
    \item Suppose instead that $v_Y\in \gamma_i$. Let $x_{i+1}=h_Y^{-1}x_i$, and let $\gamma_{i+1}=[y, x_{i+1}]$ be obtained by bending $\gamma_i$ at $v_Y$ by $h_Y^{-1}$. There are two sub-cases to consider.
    \begin{subequations}
    \begin{enumerate}
        \item If $\gamma_{i+1}\cap \mc N_2(v_U)=\{y\}$, we set $\eta_{i+1}=\eta_i$, $\widetilde  x_{i+1}=\widetilde  x_i$, and $h_{i+1}=h_ih_Y$, which is conjugate to $h_Yh_i\prec h_i$ and therefore has the same complexity. We now repeat the procedure with the data indexed by $i+1$.
        \item\label{item:bad_ending} If instead  $\gamma_{i+1}\cap \mc N_2(v_U)$ contains some other point $z$, we  stop the algorithm.
    \end{enumerate}
    \end{subequations}
\end{enumerate}

    \begin{figure}[htp]
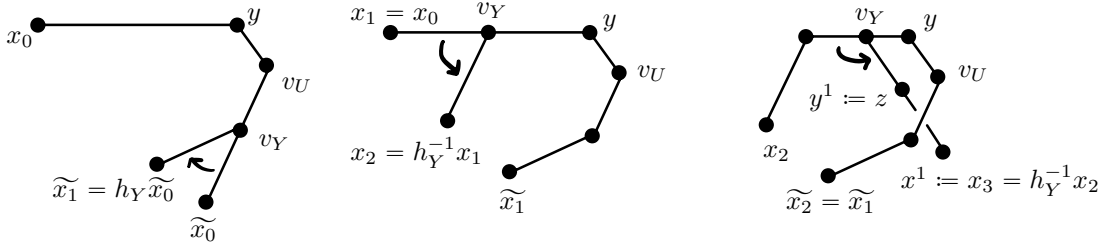

        \centering
        \begin{overpic}[width=6in, trim={1.5in 2.75in 1in 6.75in}, clip]{img/Lemma433.pdf}
            \put(0,15){$x_0$}
            \put(21, 17){$y$}
            \put(52, 16.5){$y$}
            \put(80, 16.5){$y$}
            \put(24, 11){$v_U$}
            \put(55, 12){$v_U$}
            \put(83, 12){$v_U$}
            \put(16,-2){$\widetilde{x_0}$}
            \put(4, 1.5){$\widetilde{x_1}=h_Y\widetilde{x_0}$}
            \put(22, 6){$v_Y$}
            \put(30, 17){$x_1=x_0$}
            \put(41, 17.5){$v_Y$}
            \put(30, 4.75){$x_2=h_Y^{-1}x_1$}
            \put(43, 0.5){$\widetilde{x_1}$}
            \put(74, 17){$v_Y$}
            \put(68, 0.5){$\widetilde{x_2}=\widetilde{x_1}$}
            \put(78, 2.3){$x^1\coloneq x_3=h_Y^{-1}x_2$}
            \put(70, 9.5){$y^1\coloneq z$}
            \put(66, 5){$x_2$}
        \end{overpic}
        \caption{An example of the algorithm in the proof of \Cref{lem:almost_minimal_proj}. From left to right, we applied step 1, then step 2(a), and finally step 2(b), at which point we reached a bad ending and stopped the algorithm.  We would next apply the algorithm  to the points $x^1$ and $y^1$.}
        \label{fig:enter-label}
    \end{figure}

\textbf{Good ending.} When running the above algorithm, if we never encounter the termination condition from \Cref{item:bad_ending}, then each step reduces the complexity of $h_i$. Since the complexity is a good ordering, we must eventually find some $n\in\mathbb{N}$ such that $x^*=x_n=\widetilde  x_n$, which is minimal with $U$. Since $\gamma_n$ is obtained by successively bending $\gamma_0$, there exists $Y\in \mc Y$ such that $v_Y\in \gamma_0\cap \gamma_n$. Let $t\in \gamma_0$ be the vertex after $v_Y$, so that the first vertex of $\gamma_n$ after $v_Y$ is of the form $h_Y t$ for some $h_Y\in H_Y$ (to see this, we can argue exactly as in the proof of \Cref{prop:LiftsProjCloseInU}). Then
    $$\dist_U(x,x^*)\le \dist_U(x,t)+\diam\pi_U(t)+\dist_U(t,h_Yt)+\diam\pi_U(h_Yt)+\dist_U(h_Yt,x^*).$$
    The third term is at most $\aleph$ by \Cref{rem:YtoU_is_bounded}. Moreover the segment of $\gamma_0$ between $t$ and $x$ does not pass $2$-close to $v_U$, so \Cref{lem:bgi_with_U} gives that the first term is bounded by $E$. Similarly, since we never encountered the termination condition from \Cref{item:bad_ending}, the last term is also bounded by $E$. Therefore 
    \begin{equation}\label{eqn:dUxx*}
    \dist_U(x,x^*)\le \aleph+4E,
    \end{equation}
    which satisfies the requirement of the statement.

\par\medskip
\textbf{Bad ending.} Suppose instead that the algorithm terminated because some $z\in \gamma_{i+1}-\{y\}$ is within distance $2$ from $v_U$. We can assume that $\dist_{\hat X}(z,v_U)=2$. Notice that the above argument applies to $x$ and $x_i$, giving that $\dist_U(x, x_{i})\le \aleph+4E$. Moreover, since $x_{i+1}$ differs from $x_{i}$ by an element in some $H_Y$, we have that 
\begin{equation}\label{eqn:dUxixi+1}
    \dist_U(x, x_{i+1})\le \dist_U(x, x_{i})+\diam\pi_U(x_i)+\dist_U(x_i, x_{i+1})\le 2\aleph+5E.
\end{equation}
Now set $x^1=x_{i+1}$ and $y^1=z$, and notice that $\dist_{\hat X}(y^1, x^1)< \dist_{\hat X}(y,x)$. If we repeat the whole procedure with $\{x^1,y^1,U\}$, we either obtain a good ending or find some $\{x^2,y^2\}$ with $\dist_{\hat X}(y^2, x^2)< \dist_{\hat X}(y^1,x^1)$, and so on. We cannot keep falling into the termination condition from \Cref{item:bad_ending} more than four times in total. Indeed, if so, then we would have 
$$\dist_{\hat X}(y^5, x^5)< \dist_{\hat X}(y^4,x^4)<\ldots<\dist_{\hat X}(x,y),$$ and
so $\dist_{\hat X}(y^5, x^5)\le \dist_{\hat X}(x,y)-5$. In turn, $$\dist_{\ol X}(\ol{v_U},\ol x)\le \dist_{\hat X}(v_U, x^5)\le 2+\dist_{\hat X}(y^5, x^5)\le \dist_{\hat X}(x,y)-3.$$ However this would be a contradiction, since
$$\dist_{\ol X}(\ol{v_U}, \ol x)\ge \dist_{\ol X}(\ol x, \ol y)-\dist_{\ol X}(\ol{v_U}, \ol y)\ge \dist_{\hat X}(x,y)-2,$$
where we used that $\{x,y\}$ are minimal and that $\dist_{\ol X}(\ol{v_U}, \ol y)\le \dist_{\hat X}(v_U,y)\le 2$. 

Since the process must have a good ending after at most four bad endings, there exists $j\le 4$ and some $x^*\in \ol x$ which is minimal with $U$, such that $\dist_U(x^j,x^*)\le \aleph+4E$. We then use \eqref{eqn:dUxx*} and \eqref{eqn:dUxixi+1} to conclude that
$$\dist_U(x,x^*)\le \dist_U(x,x^1)+\diam\pi_U(x^1)+\ldots+\dist_U(x^{j-1}, x^j)+\diam\pi_U(x^j)+\dist_U(x^j,x^*)\le$$
$$\le4(2\aleph+5E+E)+\aleph+4E=9\aleph+28E,$$
as required.
\end{proof}

\subsection{The quotient satisfies the HHG axioms}\label{sec:axiomproofs}
We are now ready to show that $G/N$ is a relative HHG. More precisely, we shall check that each axiom from \Cref{defn:HHS} is satisfied by the structure described in \Cref{constr:HHGStructureQuotient} for some choice of the hierarchy constant. We shall then set $\ol E$ to be the maximum of these constants, thus ensuring that all axioms hold. We advise the reader to keep the list of all constants from \Cref{tableofconstants} at hand throughout the proof.

\begin{proof}[Proof of \Cref{thm:quotientishhg}] 
We check each axiom in turn.
\par\medskip\textbf{\ref{axiom:projections} Projections:} The maps $\pi_{\ol U}$ for $\ol U\in \mf S/N$ send points to uniformly bounded diameter sets by \Cref{prop:ProjsBdd}. We now check that $\pi_{\ol U}$ is coarsely Lipschitz with respect to the fixed word metric $\dist_{G/N}$ induced by $\dist_G$. Given $\ol g, \ol h\in G/N$ such that $\dist_{G/N}(\ol g, \ol h)=1$, we will uniformly bound $\dist_{\ol U}(\ol g, \ol h)$. Let $g, h\in G$ be such that $\dist_{G}(g,h)=\dist_{G/N}(\ol g,\ol h)=1$, so that for every $V\in \mf S$ we have that $\dist_V(g,h)\le 2E$, as projections in $(G,\mf S)$ are $E$-coarsely-Lipschitz. If $\ol U=\ol S$ then
$$\dist_{\ol X}(\pi_{\ol S}(\ol g),\pi_{\ol S}(\ol g))=\dist_{\ol X}(\ol g, \ol h)\le \dist_{X}(g,h)\le 2E.$$

Now suppose $\ol U\neq \ol S$, and let $\ol T\subseteq \ol X$ be a geodesic triangle with vertices $\{\ol g, \ol h, \ol {v_U}\}$. Lift $\ol T$ to a geodesic triangle $T$ in $\hat X$ with vertices $\{g,h',v_U\}$, where $h'\in \ol h$, and consider an $X$-geodesic $\lambda$ from $g$ to $h$, whose length (in $X$, hence also in $\hat X$) is at  most $2E$. The situation is as in \Cref{fig:coarselylip_proj}.

\begin{figure}[htp]
    \centering
    \includegraphics[width=\linewidth]{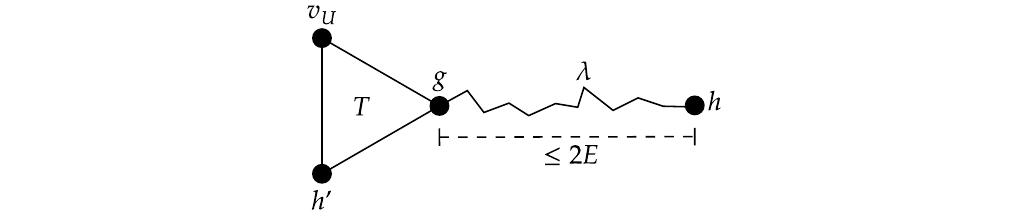}
    \caption{The configuration in the proof of the Projection axiom. In the picture, straight lines are $\hat X$-geodesics, while the zig-zag path $\lambda$ is an $X$--geodesic of length at most $2E$.}
    \label{fig:coarselylip_proj}
\end{figure}

Notice that ${g, h', v_U}$ are minimal, so $\dist_{\ol U}(\ol g, \ol h)\le \dist_U(\pi_U(g), \pi_U(h'))$. If $h=h'$, then $\dist_U(\pi_U(g), \pi_U(h'))=\dist_U(\pi_U(g), \pi_U(h))\le 2E$, and we are done. Otherwise we proceed by induction on the complexity of the element $n\in N$ mapping $h$ to $h'$. By \Cref{prop:new_shortening_pair}, there exists a shortening pair $(Y,h_Y)$. If $\dist^\pi_Y(h',g)\le 2C$ then $\dist^\pi_Y(h',h)\le 2C+2EJ$, as every two consecutive points of $\lambda$ are $X$--adjacent and $\pi_Y$ is $J$ Lipschitz. But then \Cref{eq:bound_on_L'}  gives that $\dist^\pi_Y(h',h)\le L/10$, contradicting the definition of a shortening pair. Thus we must have that $\dist^\pi_Y(h',g)> 2C$, so one of $\dist^\pi_Y(h',v_U)$ and $\dist^\pi_Y(v_U, g)$ is at least $C$ by the triangle inequality. In other words, $v_Y$ is a cut vertex for $T$ by \Cref{lem:SBGIinHHS}.  Apply $h_Y$ to the connected component of $T-\{v_Y\}$ containing $h'$. After this procedure, $g$, the image of $v_U$, and $h_Yh'$ are again minimal, and now $h_Yh'$ differs from $h$ by $h_Yn\prec n$. Proceeding inductively, we eventually find some $U'\in \ol U$ such that $\{g, h, v_{U'}\}$ are minimal and $\dist_{U'}(g,h)=\dist_U(g,h')\le 2E$, so $\dist_{\ol U}(\ol g, \ol h)\le\dist_{U'}(g,h) \le 2E $.

\par\medskip\textbf{\ref{axiom:nesting} Nesting:}  The nesting relation $\nest$ and the subsets $\rho^{\ol V}_{\ol W}$ for $\ol V \propnest \ol W$ are defined in \Cref{constr:HHGStructureQuotient}.  These sets have uniformly bounded diameter by \Cref{prop:ProjsBdd}.

\par\medskip\textbf{\ref{axiom:finite_complexity} Finite complexity:}
If $\ol U_1\nest \ol U_2\nest \ldots$ then by \Cref{non_trans_dom_and_two_points_lift} we may choose pairwise minimal representatives $U_1\nest U_2\nest \ldots$ of the $\ol U_i$.  Finite complexity in $(G,\mc S)$ then implies that both sequences are eventually constant. 

\par\medskip\textbf{\ref{axiom:orthogonal} Orthogonality:} The relation $\perp$ is defined in \Cref{constr:HHGStructureQuotient}, and by construction has the same properties as $\perp$ in $(G,\mf S)$.  Thus the axiom holds by the orthogonality axiom in $(G,\mf S)$.

\par\medskip\textbf{\ref{axiom:containers} Containers:} To show containers exist, let $\ol T\in \mf S/N$, and consider $\ol U\in (\mf S/N)_{\ol T}$  such that $\ol{\mc V} \vcentcolon = \{\ol V\in (\mf S/N)_{\ol T} \mid \ol V \perp \ol U\}\neq \emptyset$.

 If $\ol T=\ol S$, let $U\in \ol U$, and let $W$ be the container for $U$ in $S$. Now every $\ol V$ which is orthogonal to $\ol U$ has a representative $V$ which is orthogonal to $U$ and therefore nested in $W$. Then $\ol V\nest \ol W$ by \Cref{rem:UniqueReps}, proving that $\ol W$ is the container for $\ol U$ in $\ol S$.

Now assume that $\ol T\neq \ol S$, and choose $\{U,T\}$ minimal representatives of $\ol U$ and $\ol T$.  Since $U\nest T$, we have $\dist_X(\rho^U_S,\rho^T_S)\le E$ by consistency in $\mf S$. By \Cref{rem:UniqueReps}, each $\ol V\in \ol{\mc V}$ has a unique representative $V$ which is orthogonal to $U$. Let $\mc V$ be the collection of such representatives. Notice that $\dist_X(\rho^V_S,\rho^U_S)\le 2E$ by \Cref{lem:close_proj_of_orthogonals}.  Combining this with the previous inequality yields that $\dist_X(\rho^V_S,\rho^T_S)\le 4E\le 2A$, so $V$ and $T$ are minimal. Since $\ol V\nest \ol T$, the minimal representatives $V$ and $T$ must be nested. Now let $W$ be the container for $U$ in $T$, which contains every $V\in \mc V$. It then follows from \Cref{rem:UniqueReps} that $\ol W\propnest \ol T$ and $\ol W$ contains every $\ol V\in \ol{\mc V}$, as required.

\par\medskip\textbf{\ref{axiom:transversality} Transversality:}  The relation $\trans$ and the sets $\rho^{\ol U}_{\ol W}\subseteq \mc C\ol W$ and $\rho^{\ol W}_{\ol U}\subseteq \mc C\ol U$ are defined in \Cref{constr:HHGStructureQuotient}.  These sets  have uniformly bounded diameter by \Cref{prop:ProjsBdd}. 

\par\medskip\textbf{\ref{axiom:consistency} Consistency:} Let $\ol x\in G/N$, and let $\ol U,\ol V\in \mf S/N$ satisfy $\ol U\trans \ol V$.  By \Cref{lem:LiftingTriples}, we may choose $x\in \ol x$, $U\in \ol U$, and $V\in \ol V$ such that  $\{x,U,V\}$ are pairwise minimal.  By the definition of the relation $\trans$ in $\mf S/N$, we have $U\trans V$.  It follows from the consistency axiom in $(G,\mf S)$ that
\[
\min\{\dist_U(x,\rho^V_U), \dist_V(x,\rho^U_V)\}\le E.
\]
The first consistency inequality now follows immediately from \Cref{lem:DistInQuotient}.

For the second statement, suppose that $\ol U\nest \ol V$ and either $\ol V \propnest \ol W$, or $\ol V \trans \ol W$ and $\ol W \not\perp \ol U$. Suppose first that $\ol W\neq \ol S$. We may choose  $U\in \ol U$, $V\in \ol V$, and $W\in \ol W$ with $\{U,V,W\}$ pairwise minimal by \Cref{lem:LiftingTriples}, and the same relations hold between $U, V,$ and $W$ as between $\ol U, \ol V, \ol W$. It follows from the consistency axiom in $(G,\mf S)$ that $\rho^U_W$ is coarsely equal to $\rho^V_W$ in $\mc CW=\mc C\ol W$, and the result again follows from \Cref{lem:DistInQuotient}.  

If instead $\ol W = \ol S$ let $U\in \ol U$ and $V\in \ol V$ be nested representatives. Then $\dist_{X}(\rho^U_S,\rho^V_S)\le E$ by consistency in $\mf S$. As $E\le 2A$, we have $\dist_{\ol X}(\ol{v_U}, \ol{v_V})\le \dist_{\hat X}({v_U}, {v_V})\le 2$.

\par\medskip\textbf{\ref{axiom:hyperbolicity} Hyperbolicity:} If $(G,\mf S)$ is a HHG, then $\mc C\ol U$ is hyperbolic for all $U\in \mf S$ by definition if $U\neq S$ and by \Cref{thm:hyperbolicity_olX_with_points} if $U=S$. If $(G,\mf S)$ is only a relative HHG, then the spaces associated to all domains that are not $\nest$--minimal in $\mf S/N$ are $E$--hyperbolic for the same reason.

\par\medskip\textbf{\ref{axiom:bounded_geodesic_image} Bounded geodesic image:} Let $\ol W\in \mf S/N$,  let $\ol V \propnest \ol W$, and let $\ol x, \ol y\in G/N$ satisfy $\dist_{\ol V}(\ol x, \ol y)> E$. Let $\ol \gamma$ be a geodesic in $\mc C\ol W$ from $\pi_{\ol W}(\ol x)$ to $\pi_{\ol W}(\ol y)$.  We will show that a uniform neighborhood of $\rho^{\ol V}_{\ol W}$ intersects $\ol\gamma$. 

Suppose first that $\ol W \neq \ol S$, so that $\mc C\ol W=\mc CW$. Since $\ol V\propnest \ol W$,  \Cref{non_trans_dom_and_two_points_lift} implies that there exist pairwise minimal representatives $\{x,y,V,W\}$. Since $\{x, V\}$ are minimal and $\mc C\ol V = \mc CV$, we have $\pi_V(x)\subseteq\pi_{\ol V}(\ol x)$, and similarly for $y$. 
 Thus  $\dist_V(x,y)\ge \dist_{\ol V}(\ol x, \ol y)> E$.  By \Cref{rem:closeProj} we have $\rho^V_W\subseteq \rho^{\ol V}_{\ol W}$, so the bounded geodesic image axiom in $(G,\mf S)$ provides a point  $w\in \mc N_E(\rho^V_W)\subseteq \mc N_E(\rho^{\ol V}_{\ol W})$ which lies on any geodesic $\lambda$ connecting $\pi_W(x)$ to $\pi_w(y)$.  Connect the endpoints of $\lambda$ to the endpoints of $\ol\gamma$ with two geodesics, each of whose length is at most $\diam(\pi_{\ol W}(\ol x))\le\beth$ because $\pi_W(x)\subseteq \pi_{\ol W}(\ol x)$ and similarly for $y$. Since geodesic quadrangles in $\mc C W$ are $2E$-slim, we obtain  $\dist_W(w, \ol\gamma)\le 2E+\beth$, so $\ol\gamma\cap N_{3E+\beth}(\rho^{\ol V}_{\ol W})\neq \emptyset$, as desired.

Now suppose $\ol W=\ol S$. In this case, $\ol \gamma$ is a geodesic in $\ol X$ from $\ol x$ to $\ol y$. Complete $\ol \gamma$ to a geodesic triangle $\ol T$ with vertices $\{\ol x, \ol y, \ol{v_V}\}$, and let $T$ be a lift of $\ol T$ to $\hat X$ with vertices $\{x,y,v_V\}$. Let $\gamma$ be the lift of $\ol\gamma$ inside $T$. Since $\{x,y, V\}$ are pairwise minimal, $\dist_V(x,y)\ge \dist_{\ol V}(\ol x, \ol y)>E$, so \Cref{lem:bgi_with_U} yields that $\dist_{\hat X}(v_V,\gamma)\le 2$, concluding the proof of the bounded geodesic image axiom.

\par\medskip\textbf{\ref{axiom:partial_realisation} Partial realization:} Let $\{\ol V_j\}$ be a collection of pairwise orthogonal domains of $\mf S/N$, and fix $p_j\in \mc C\ol V_j$ for each $j$. If $\{\ol V_j\}=\{\ol S\}$, then the unique given point is $\ol p\in \hat X/N$.  Choose any representative $p\in \ol p$, and let $x\in X\subseteq \hat X$ be at distance at most one from $p$. Since vertices of $X$ are all elements of $G$, we can consider $\ol x$ as an element of $G/N$. Note that $\pi_{\ol S}(x)=\ol x$, which is at distance $1$ from $\ol p$ in $\hat X/N$. Hence the first bullet point of the partial realization axiom holds, and the other two vacuously hold.

Thus suppose that $\ol V_j\neq \ol S$ for all $j$. Since the $\ol V_j$ are pairwise orthogonal, \Cref{non_trans_dom_and_two_points_lift} provides a collection $\{V_j\}$ of pairwise orthogonal, pairwise minimal representatives $V_j\in \ol V_j$.  As $\mc C\ol V_j=\mc C V_j$ in this case, we have $p_j\in \mc C V_j$.  Let $x\in G$ be the point provided by the partial realization axiom in $(G,\mf S)$. Since $S\sqsupsetneq V_j$ for each $j$, we have that $\dist_X(x,\rho^S_{V_j})\le E$.  Thus $\{x,V_j\}$ is a minimal pair for each $j$, since $x\in \mc N_A(\rho^S_{V_j})$ and therefore $x$ and $v_{V_j}$ are joined by an edge.  We will show that $\ol x\in G/N$ satisfies the conclusions of the partial realization axiom, \Cref{axiom:partial_realisation}.

The first bullet point of \Cref{axiom:partial_realisation} holds because $\mc C\ol V_j=\mc CV_j$ and $\{x,V_j\}$ is minimal for each $j$, which implies that $\dist_{\ol V_j}(\ol x, p_j) \le \dist_{V_j}(x, p_j) \le E$ by \Cref{lem:DistInQuotient}.

For the second bullet point of the axiom, fix $\ol W\in \mf S/N$ such that $\ol V_j\propnest \ol W$ or $\ol V_j\trans\ol W$ for some $j$, and for simplicity let $\ol V=\ol V_j$. We will uniformly bound $\dist_{\ol W}(\ol x, \rho^{\ol V}_{\ol W})$. If $\ol W=\ol S$ then $\dist_{\ol X}(\ol x,  \rho^{\ol V}_{\ol S})\le \dist_{X}(x,  \rho^V_S)+1\le E+1$ by \Cref{lem:DistInQuotient}. Otherwise, by \Cref{lem:LiftingTriples}, there exist $W\in \ol W$ and $V'\in \ol V$ such that $\{x,V',W\}$ are pairwise minimal. By \Cref{lem:orbit_of_edge_is_unique}, we can also assume that $V'=V$, as $v_{V}$ and $x$ are joined by an edge. Hence $\{x,V,W\}$ are pairwise minimal, which in particular means that the relation between $V$ and $W$ is the same as that between $\ol V$ and $\ol V$. Hence, by minimality and the realization axiom in $\mf S$ we have that $\dist_{\ol W}(\ol x, \rho^{\ol V}_{\ol W})\le \dist_{W}(x, \rho^{V}_{W})\le E$, concluding the proof.

\par\medskip\textbf{\ref{axiom:uniqueness} Uniqueness:}
Let $\ol x, \ol y\in G/N$ and $r\ge 0$, and suppose that $\dist_{\ol U}(\ol x, \ol y)\le r$ for every $\ol U\in \mf S/N$. We will show that $\dist_{G/N}(\ol x, \ol y)$ is uniformly bounded by a constant depending only on $r$.  Let $\{x,y\}$ be minimal representatives with $x\in \ol x$ and $y\in \ol y$.  

Recall that $\hat X$ is the cone-off of $X$ with respect to the family $\mc H=\mc Y\cup \{\mc N_A(\rho^U_S)\}_{U\propnest S}$ of $\max\{K,3E\}$-quasiconvex subsets. By \cite[Proposition~2.27]{Spriano_hyperbolic_I}, there exists $\xi=\xi(E,K)\geq 1$ and a $\xi$--quasigeodesic $\gamma$ from $x$ to $y$ in $\hat X$ such that any de-electrification $\widetilde \gamma$ is itself a $\xi$--quasigeodesic from $x$ to $y$ in $X$. By construction, we have 
\[
\widetilde \gamma=\sigma_1*\eta_1*\cdots * \sigma_k * \eta_k * \sigma_{k+1},
\]
where each $\sigma_i$ is a $\xi$-quasigeodesic in $X$ contained in $\gamma$, while each $\eta_i$ is a geodesic in $X$ connecting points in some subspace $Y_i\in \mc Y$ or in the $A$-neighborhood of some $\rho^{U_i}_S$. 
Let $\alpha$ be (the image in $X$ of) a $\lambda$--hierarchy path in $G$ from $x$ to $y$, whose existence is guaranteed by \Cref{lem:hierarchypath}. By \Cref{lem:stability}, the unparameterized $\lambda$-quasigeodesic $\alpha$ and the $\xi$-quasigeodesic $\widetilde \gamma$ are at Hausdorff distance at most $\Phi$, where $\Phi$ depends only on $\lambda$, $\xi$, and $E$, and so ultimately on $E$ and $K$.

Let $\Psi$ be such that each subgroup $H_Y$ has a $\Psi$--cobounded action on $Y$, which exists by \Cref{hyp:complete}, and fix 
\begin{equation}\label{eq:daleth}
    \daleth>2\Phi + 2\Psi + 3E+2A.
\end{equation} We will consider only the collection of paths $\eta_{i_1},\dots, \eta_{i_m}$ whose endpoints are at distance greater than $\daleth$ in $X$, that is, such that $\dist_X((\eta_{i_j})_-,(\eta_{i_j})_+)>\daleth$. In particular, since $\daleth>2A+E=\diam(\mc N_A(\rho^U_S))$, the endpoints $c_j'=(\eta_{i_j})_-$ and $d_j'=(\eta_{i_j})_+$ must belong to some $Y\in \mc Y$.  Let $c_j,d_j\in Y$ be points in the same $H_Y$-orbit at distance at most $\Psi$ from $c_j'$ and $d_j'$, respectively. Let $a_j,b_j\in \alpha$ be points at distance at most $\Phi$ from $c_j',d_j'$ respectively, so that $\dist_X(a_j,c_j)\le \Phi + \Psi$ and $\dist_X(b_j,d_j)\le \Phi + \Psi$.  See \Cref{fig:DivertedPath1}.

\begin{figure}[htp]
    \centering
    \begin{overpic}[page=3, width=4in, trim={1.25in 6in 2in 1in}, clip]{img/HHGFigs}
    \put(5,26){$x$}
    \put(91,30){$y$}
    \put(15,26){$a_1$}
    \put(35,27){$b_1$}
    \put(63,24){$a_2$}
    \put(78,29){$b_2$}
    \put(19,52){$c_1$}
    \put(14,45){$c_1'$}
    \put(34,60){$d_1$}
    \put(39,60){$d_1'$}
    \put(68,50){$c_2$}
    \put(63,50){$c_2'$}
    \put(77,47){$d_2$}
    \put(81,42){$d_2'$}
    \put(1.5,11){{\color{BrickRed}$\le \Phi + \Psi$}}
    \put(63,15){{\color{BrickRed}$\le \Phi + \Psi$}}
    \put(34,1){\color{gray}$Y_2$}
    \put(55,1){\color{gray}$Y_3$}
    \put(25,62){\color{gray}$Y_1$}
    \put(75,56){\color{gray}$Y_4$}
    \put(10,38){$\sigma_1$}
    \put(26,47){$\eta_1$}
    \put(25,20){$\sigma_2$}
    \put(35,10){$\eta_2$}
    \put(46,16){$\sigma_3$}
    \put(55,8){$\eta_3$}
    \put(54,20){$\sigma_4$}
    \put(70,38){$\eta_4$}
    \put(86,38){$\sigma_5$}
    \end{overpic}
    \caption{The setup for the proof of the Uniqueness Axiom, in the space $X$.  The undulated line passing through $x$, $a_1$, $b_1$, and so on is the $\lambda$-hierarchy path $\alpha$, while the de-electrification $\widetilde{\gamma}$ is the union of the segments $\sigma_1$, $\eta_1$, $\sigma_2$ and so on. In this example, $\dist_{Y_i}(x,y)\geq \daleth$ for $i=1,4$.}
    \label{fig:DivertedPath1}
\end{figure}

Let $\omega_1=\alpha|_{[x,a_1]}\cup [a_1,c_1]$;  for $1<i\le m$, let $\omega_i=[d_{i-1},b_{i-1}]\cup \alpha|_{[b_{i-1},a_i]} \cup [a_i,c_i]$; and  let $\omega_{m+1} = [d_m,b_m]\cup \alpha|_{[b_m,y]}$. See \Cref{fig:DivertedPath2}.

\begin{figure}[htp]
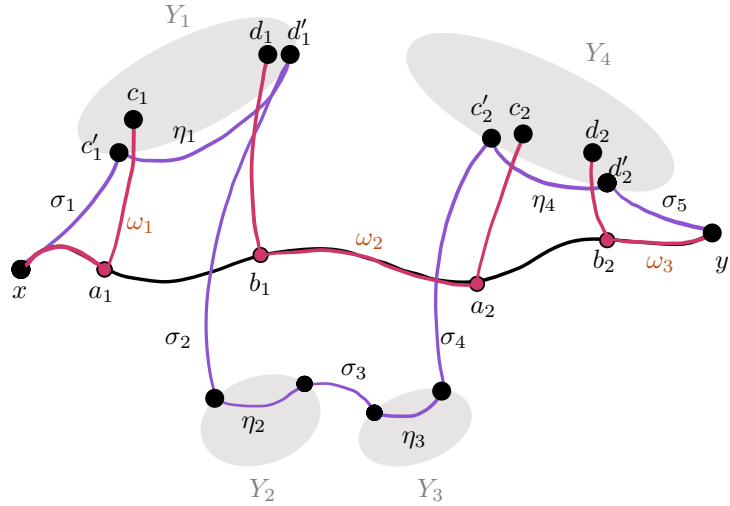

    \centering
    \begin{overpic}[page=2, width=4in, trim={1.5in 3in 2in 4.5in}, clip]{img/HHGFigs}
    \put(5,26){$x$}
    \put(97,30){$y$}
    \put(15,26){$a_1$}
    \put(36,27){$b_1$}
    \put(65,24){$a_2$}
    \put(81,30){$b_2$}
    \put(20,52){$c_1$}
    \put(14,45){$c_1'$}
    \put(36,60){$d_1$}
    \put(41,60){$d_1'$}
    \put(70,50){$c_2$}
    \put(65,50){$c_2'$}
    \put(80,47){$d_2$}
    \put(83,42){$d_2'$}
    \put(20,35){{\color{Bittersweet}$\omega_1$}}
    \put(50,33){{\color{Bittersweet}$\omega_2$}}
    \put(88,30){{\color{Bittersweet}$\omega_3$}}
    \put(36,0){\color{gray}$Y_2$}
    \put(58,0){\color{gray}$Y_3$}
    \put(25,62){\color{gray}$Y_1$}
    \put(80,57){\color{gray}$Y_4$}
    \put(10,38){$\sigma_1$}
    \put(26,47){$\eta_1$}
    \put(25,20){$\sigma_2$}
    \put(35,9){$\eta_2$}
    \put(48,16){$\sigma_3$}
    \put(56,7){$\eta_3$}
    \put(61,20){$\sigma_4$}
    \put(73,38){$\eta_4$}
    \put(90,38){$\sigma_5$}
    \end{overpic}
    \caption{The subpaths $\omega_i$ in the setup for the proof of the Uniqueness Axiom.}
    \label{fig:DivertedPath2}
\end{figure}

Recalling that each vertex of $X$ represents an element of $G$, the path 
\[
\omega_1\cup [c_1,d_1] \cup \omega_2 \cup [c_2,d_2]\cup \cdots \cup [c_m,d_m] \cup \omega_{m+1}
\]
from $x$ to $y$ shows that we can write the element $x^{-1}y\in G$ as 
\[
x^{-1}y = (x^{-1}a_1) \cdot (a_1^{-1}c_1) \cdot (c_1^{-1}d_1) \cdot (d_1^{-1}b_1) \cdot \cdots (b_m^{-1}y).
\]
Since $c_j,d_j$ are in the same coset of some $H_{Y}$, the element $c_j^{-1}d_j$ belongs to $N$.  In particular, we have 
\begin{equation}\label{eqn:productofcosets}
x^{-1}y\in w_1 N \cdot w_2 N \cdot \dots \cdot w_{m+1} N,
\end{equation}
where $w_1=x^{-1}c_1$, and similarly $w_j = d_{j-1}^{-1}c_j$ for $1<j\le m$ and $w_{m+1}= d_{m}^{-1}y$.

Recall that the word metric $\dist_G$ (resp. $\dist_{G/N}$) is induced by the finite generating set $\mc{T}$ (resp. $\ol{\mc{T}}$). In general, if $w$ is a word in the product of cosets $(d_1N)(d_2N)\dots (d_mN)$, then $|\ol w|_{\ol{\mc{T}}}\le \sum_{i=1}^m |d_i|_{\mc{T}}$.  Applying this fact to the representation of $x^{-1}y$ in \Cref{eqn:productofcosets}, we see that it suffices to bound the $\mc{T}$--lengths of the coset representatives $w_j$. Each is constructed as a product of elements of $G$, so we  bound the $\mc{T}$--length of each factor individually.  

We proceed via a sequence of claims. The first two show that the $\mc{T}$--lengths of the coset representatives is approximated by the lengths of the appropriate paths in $X$. We first consider factors in coset representatives that lie on $\alpha$.

\begin{claim}
\label{claim:x_y_smallproj}
    There exist constants $k_1,k_2$ such that if $a,b\in \alpha$, then $\dist_G(a,b) \asymp_{k_1,k_2} \dist_X(a,b)$.
\end{claim}

\begin{claimproof}[Proof of \Cref{claim:x_y_smallproj}] Recall that $\alpha$ is a hierarchy path between $x$ and $y$, which are minimal. Let $[x,y]$ be an $\hat X$-geodesic. Since $x$ and $y$ are minimal, any two points on $[x,y]$ are minimal, as otherwise we could find closer lifts of $\ol x$ and $\ol y$. We first prove that  $\dist_U(x,y)$ is uniformly bounded for all $U\propnest S$. Indeed, suppose there exists $U \propnest S$ with $\dist_U(x,y)>E$. By \Cref{lem:bgi_with_U}, there exists $z\in[x,y]$ such that $\dist_{\hat X}(v_U,z)\le 2$. In turn, since $\{z,x\}$ and $\{z,y\}$ are minimal, \Cref{lem:almost_minimal_proj} provides the existence of $x^*\in \ol x$ and $y^*\in \ol y$ such that $\{x^*,U\}$ and $\{y^*,U\}$ are minimal and $\max\{\dist_U(x,x^*),\dist_U(y,y^*)\}\le 9\aleph+28E$. Hence
$$\dist_U(x,y)\le \dist_U(x,x^*)+\diam\pi_{\ol U}(\ol x)+\dist_{\ol U}(\ol x, \ol y)+\diam\pi_{\ol U}(\ol y)+\dist_U(y,y^*) \le 18\aleph+2\beth+56E+r.$$
Since $\alpha$ is a $\lambda$--hierarchy path in $G$, if $\dist_U(x,y)$ is uniformly bounded, then the diameter of $\pi_U(\alpha)$ is at most some uniform constant $s'$. Therefore $\diam\pi_U(\{a,b\})\le s'$ whenever $a,b\in \alpha$.  Applying the distance formula (\Cref{thm:distance_formula}) with $s=\max\{s'+E+\aleph+1,s_0\}$ yields constants $k_1,k_2$ so that $\dist_G(a,b) \asymp_{k_1,k_2} \dist_X(a,b)$, as required. We note that, for later purposes, we deliberately chose $s$ to be bigger than $\max\{s',s_0\}$, though the latter threshold would have been sufficient to apply the distance formula and conclude the proof of \Cref{claim:x_y_smallproj}.
\end{claimproof}

\noindent We next consider factors in coset representatives labeling paths from $\alpha$ to some  $Y_j$.
\begin{claim}\label{claim:shortsegs}
    For $i=1,\dots, m$, we have
    \[
    d_G(a_i,c_i)\asymp_{k_1,k_2} \dist_X(a_i,c_i) \le \Phi + \Psi,
    \] and similarly
    \[  
    d_G(b_i,d_i)\asymp_{k_1,k_2} \dist_X(b_i,d_i) \le \Phi + \Psi.
    \]
\end{claim}
\begin{claimproof}[Proof of \Cref{claim:shortsegs}]
    We prove the first statement, as the second follows symmetrically. Let $U\in \mf S-\{S\}$ be such that $\dist_U(a_i,c_i)>E$, so that $[a_i,c_i]\cup \mc N_{E}(\rho^U_S)\neq \emptyset$ by the bounded geodesic image axiom. Let $[a_i,c_i] \cup [c_i,d_i] \cup [d_i,b_i] \cup [b_i,a_i]$ be a geodesic rectangle in $X$.  It follows from the definition of $\daleth$ in \Cref{eq:daleth} that the side $[c_i,d_i]$ has length greater than $2\Phi+2\Psi + 3E$. Thus $\mc N_E(\rho^U_S)$ cannot intersect $[b_i,d_i]$ non-trivially, as it already intersects $[a_i,c_i]$. Therefore $\dist_U(b_i,d_i)\le E$ by the bounded geodesic image axiom.  Moreover, $\diam\pi_U(\{c_i,d_i\})\le \aleph$ by \Cref{rem:YtoU_is_bounded}, and by the proof of \Cref{claim:x_y_smallproj}, we see that $\diam\pi_U(\{a_i,b_i\}) \le s'$, since $a_i,b_i$ lie on $\alpha$. Thus
    \[
    \dist_U(a_i,c_i)\le \diam\pi_U(\{a,b\})+\dist_U(b_i,d_i)+ \diam\pi_U(\{c_i,d_i\})\le s'+E+\aleph.\]
     The claim follows by applying the distance formula (\Cref{thm:distance_formula}) with the same threshold $s>s'+E+\aleph$ 
     as in \Cref{claim:x_y_smallproj}, and using that $[a_i,c_i]$ and $[b_i,d_i]$ each have $X$--length at most $\Phi +\Psi$ by construction.
\end{claimproof}

\noindent The final claim relates the sum of the lengths of the subpaths of $\alpha$ to $\dist_{\hat X}(x,y)$.
\begin{claim}\label{claim:subpathsofgeod}
    There is a constant $\Xi\geq 1$ such that     
    \[
    \dist_X(x,a_1) + \sum_{i=1}^{m-1} \dist_X(b_i,a_{i+1}) + \dist_X(b_m,y) \le \Xi.
    \]
\end{claim}
\begin{claimproof}[Proof of \Cref{claim:subpathsofgeod}]
    For convenience, let $b_0=x$ and $a_{m+1}=y$, so that the sum we are interested in bounding is $\sum_{i=0}^m\dist_X(b_i,a_{i+1})$. By the triangle inequality, for each $1\le i \le m-1$ we have that $\dist_X(b_i,a_{i+1}) \le \dist_X(b_i,d_i')+\dist_X(d_i',c'_{i+1})+\dist_X(c'_{i+1},a_{i+1})$. Hence
       \begin{equation}\label{eqn:sum1}
    \dist_X(b_i,a_{i+1}) \le
    2\Phi  + \ell_X(\widetilde \gamma|_{[d_i',c_{i+1}']}) \le 2\Phi + \sum_j \ell_X(\sigma_{i_j}) + \sum_j \ell_X(\eta_{i_j}),
    \end{equation}
    where the sums are taken over indices $i_j$  so that each $\eta_{i_j}$ and $\sigma_{i_j}$ are contained in the subpath of $\widetilde \gamma$ from $d_i'$ to $c_{i+1}'$.  See \Cref{fig:biToai+1}.

\begin{figure}[htp]
    \centering
    \begin{overpic}[page=2, width=2in, trim={2in 7in 3.5in 1in}, clip]{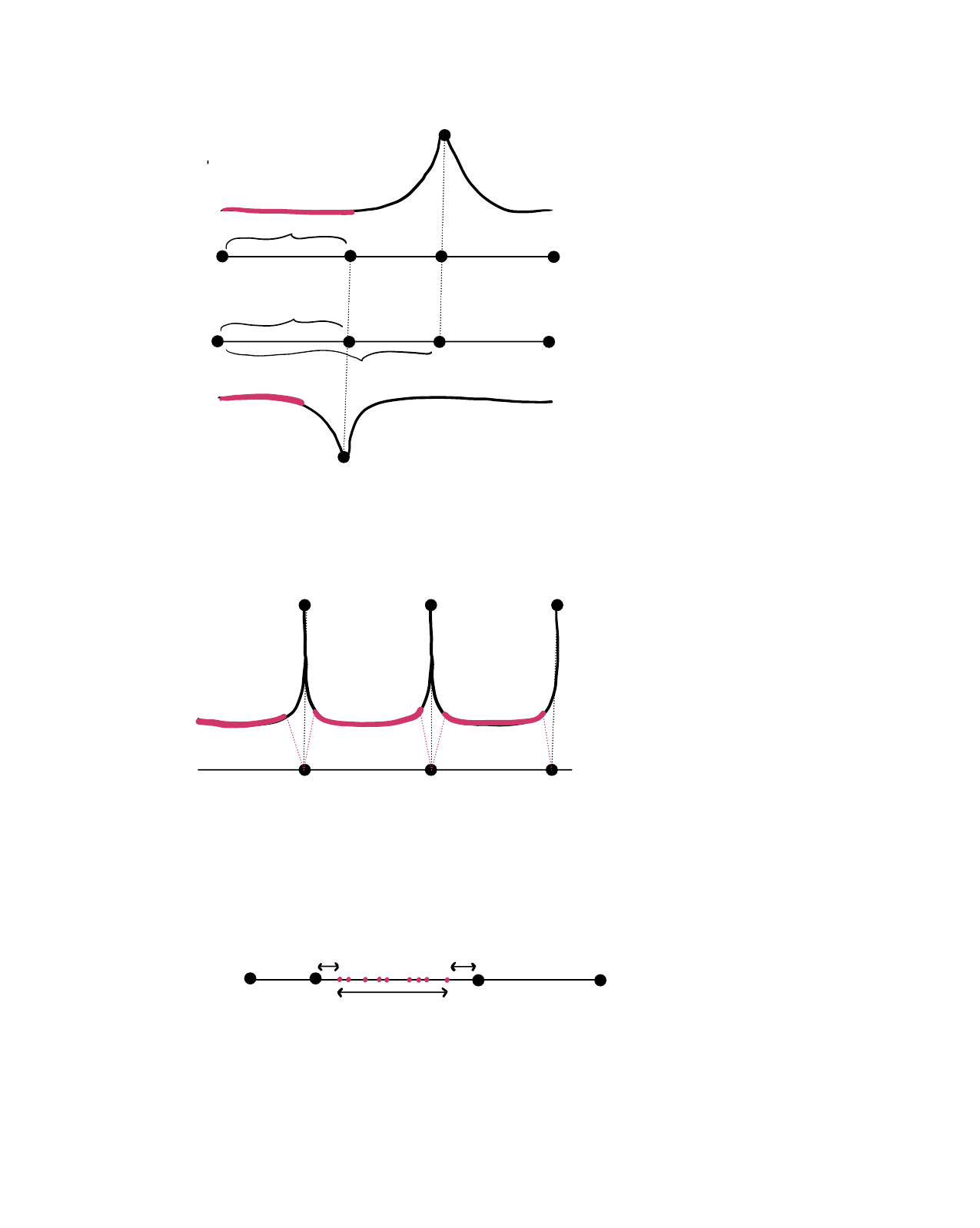}
    \put(8,35){$\sigma_2$}
    \put(25,25){$\eta_2$}
    \put(44,27){$\sigma_3$}
    \put(60,14){$\eta_3$}
    \put(68,65){$\sigma_4$}
    \put(25,3){{\color{gray} $Y_2$}}
    \put(60,3){{\color{gray} $Y_3$}}
    \put(17,84){$d_1$}
    \put(25,84){$d_1'$}
    \put(81,85){$c_2'$}
    \put(94,84){$c_2$}
    \put(29,39){$b_1$}
    \put(77,38){$a_2$}
    \put(28,62){{\color{Cerulean}$\le \Phi$}}
    \put(81,60){{\color{Cerulean}$\le \Phi$}}
    \end{overpic}
    \caption{Using the triangle inequality to obtain the bound \eqref{eqn:sum1} on $\dist_X(b_1,a_2)$ in the proof of \Cref{claim:subpathsofgeod}.  The concatenation $\sigma_2*\eta_2*\sigma_3*\eta_3*\sigma_4$ is the subpath of the de-electrification $\widetilde  \gamma$ from $d_1'$ to $c_2'$.}
        \label{fig:biToai+1}
\end{figure}

\noindent Similarly
    \begin{equation}\label{eqn:sum2}
    \dist_X(b_0,a_1)\le \Phi  + \sum_j \ell_X(\sigma_{0_j}) + \sum_j \ell_X(\eta_{0_j})
    \end{equation}
    and 
    \begin{equation}\label{eqn:sum3}
    \dist_X(b_m,a_{m+1})\le \Phi  + \sum_j \ell_X(\sigma_{m_j}) + \sum_j \ell_X(\eta_{m_j}),
    \end{equation}
    where the sums are taken over indices $i_j$  so that each $\eta_{i_j}$ and $\sigma_{i_j}$ are contained in the subpath of $\widetilde \gamma$ from $x$ to $c_1'$ and from $d_m'$ to $y$, respectively.

    By our choice of $c_i',d_i'$, each subpath $\eta_{i_j}$ in any of the sums in \eqref{eqn:sum1}, \eqref{eqn:sum2}, and \eqref{eqn:sum3} has $X$--length at most $\daleth$.  Moreover, the total number of paths $\sigma_i$ and $\eta_i$ is at most the $\hat X$--length of the $\xi$--quasigeodesic $\gamma$, which is in turn bounded by $\xi \dist_{\hat X}(x,y)+\xi=\xi(r+1)$.  
    In particular, $m\le \xi(r+1)$, and the sums $\sum_{i=1}^m \sum_j \ell_X(\sigma_{i_j})$  and $\sum_{i=1}^m \sum_j \ell_X(\eta_{i_j})$, where the inner sums are as in \eqref{eqn:sum1}--\eqref{eqn:sum3}, each have at most $\xi(r+1)$ terms.  Combining this with the observation that each $\sigma_i$ is a subpath of the $\xi$--quasigeodesic $\gamma$, we obtain
    \begin{align*}
        \sum_{i=0}^m \dist_X(b_i,a_{i+1}) 
        & \le 2l\Phi  +\sum_{i=0}^m \left(\sum_j \ell_X(\sigma_{i_j}) + \sum_j \ell_X(\eta_{i_j}) \right) \\
        & \le2\xi(r+1)\Phi  + \sum_{i=0}^m \left( \sum_j \ell_X(\sigma_{i_j}) + \sum_j \daleth \right) \\
        & \le \xi(r+1)(2\Phi + \daleth) + \sum_{i=0}^m \sum_j\ell_X(\sigma_{i_j}) \\
        & \le \xi(r+1)(2\Phi + \daleth) + \ell_{\hat X}(\gamma) \\
        & \le \xi(r+1)(2\Phi + \daleth+1).
    \end{align*}
    Thus setting $\Xi=\xi(r+1)(2\Phi + \daleth+1)$ completes the proof of the claim.     
\end{claimproof}

\noindent We are now ready to bound the sum of the lengths of the coset representatives.  It follows from \Cref{claim:x_y_smallproj} and \Cref{claim:shortsegs} that for $2\le i\le m$, we have
\begin{align*}
    |w_i|_{\mc T} = \dist_G(d_{i-1},c_{i}) & \le \dist_G(d_{i-1},b_{i-1}) + \dist_G(b_{i-1}, a_i) + \dist_G(a_i, c_i) \\
        & \asymp_{k_1,k_2} \dist_X(d_{i-1},b_{i-1}) + \dist_X(b_{i-1}, a_i) + \dist_X(a_i, c_i) \\
        & \preceq_{k_1,k_2} 2\Phi + 2\Psi + \dist_X(b_{i-1},a_i).
\end{align*}
For $i=1,\ldots,m+1$, a similar argument yields $|w_1|_{\mc T}=d_G(x,c_1)\preceq_{k_1,k_2} \Phi + \Psi + \dist_X(x,a_1)$ and $|w_{m+1}|_{\mc T}=d_G(d_m,y)\preceq_{k_1,k_2} \Phi + \Psi + \dist_X(b_m,y) $. Putting this all together, along with the fact that $m\le \xi(r+1)$ as described in the proof of the previous claim, we have 
\begin{align*}
    \sum_{i=1}^{m+1} |w_i|_{\mc T}& \preceq_{k_1,k_2} \dist_X(x,a_1) + \Phi + \Psi + \sum_{i=2}^{m} \left(2\Phi + 2\Psi + \dist_X(b_{i-1}, a_i)\right)  + \dist_X(b_m,y) + \Phi + \Psi  \\
     & \preceq_{k_1,k_2} \xi(r+1)(2\Phi + 2\Psi) + \dist_X(x,a_1) + \sum_{i=2}^{m} \dist_X(b_{i-1},a_i) + \dist_X(b_m,y)\\
     & \preceq_{k_1,k_2}  \xi(r+1)(2\Phi + 2\Psi) + \Xi ,
\end{align*}
where the final inequality  follows by \Cref{claim:subpathsofgeod}.  This bound is independent of the choice of $x$ and $y$, completing the proof of the uniqueness axiom.

\par\medskip\textbf{\ref{axiom:large_link_lemma} Large links:} By \Cref{lem:passingupisenough}, instead of checking the large link axiom, it suffices to prove that the passing up axiom \ref{axiom:passing_up} (Passing up) holds for $(G/N,\mf S/N)$, for a suitable choice of the hierarchy constant.  To this end we let $t>0$, and we fix the following constants (see \Cref{tableofconstants} for a list of where all involved quantities are defined):
\begin{align*}
&c_0=t+4\aleph+20E+2\beth,\\
&c_1=2A+3E+2C+J(K+D+2E)+\Psi,\\
&c_2=2c_1+2D+2K,\\
&c_3=c_0+c_2+12E+2.
\end{align*}

Set ${\ol{P}}>\max\{P(c_0), (t+1)P(c_3)\}$, where $P\colon \mathbb{R}_{>0}\to \mathbb{N}$ is the function provided by the passing up axiom applied to $(G,\mf S)$ (with hierarchy constant $E$). 

 Let $\ol V\in \mf S/N$ and $\ol x, \ol y\in G/N$, and let $\{\ol U_1, \ol U_2, \dots, \ol U_{\ol{P}}\} \subseteq (\mf S/N)_{\ol V}$ be such that $\dist_{\ol U_i}(\ol x, \ol y)>5E$ for all $i=1\ldots, \ol P$. Our goal is to find a domain $\ol W\nest \ol V$ and an index $i$ such that $\ol{U_i}\propnest \ol W$ and $\dist_{\ol W}(\ol x, \ol y)>t$.

\medskip
\textbf{Case 1: $\ol V\neq \ol S$}. Let $\{x,y,V,U_1,\ldots,U_{\ol{P}}\}$ be pairwise minimal representatives, which exist by \Cref{non_trans_dom_and_two_points_lift}. Since $\dist_{U_i}(x,y)\ge \dist_{\ol U_i}(\ol x, \ol y)$ by \Cref{lem:DistInQuotient}, the passing up axiom for $(G,\mf S)$ provides a domain $W\in \mf S_V$ containing some $U_i$ such that $\dist_W(x,y)> c_0$. 
 Without loss of generality, suppose $U_1\propnest W$. Notice that $\ol U_1\propnest \ol W\propnest \ol V$, as pointed out in \Cref{rem:UniqueReps}.  By \Cref{lem:LiftingTriples} there exist $x'\in \ol x$ and $U_1'\in \ol U_i$ such that $\{x',W,U_1'\}$ are pairwise minimal, and we must have that $U_1=U_1'$ by \Cref{prop:projections_are_welldef}. Then \Cref{prop:LiftsProjCloseInU} applied to $x,x',U,W$ yields that $\dist_W(x,x')\le 2\aleph+9E$. For the same reason, there exists $y'\in \ol y$ with $\{y',W\}$ minimal such that $\dist_W(y,y')\le 2\aleph+9E$. Thus, by \Cref{lem:DistInQuotient} we have that
$$\dist_{\ol W}(\ol x, \ol y)\ge \dist_{W}(x', y')-2\beth\ge \dist_{W}(x, y)-\diam\pi_W(x')-\diam\pi_W(y')-4\aleph-18E-2\beth>t.$$

\textbf{Case 2: $\ol V= \ol S$}. Towards a contradiction, assume that $\dist_{\ol W}(\ol x,\ol y)\le t$ for every $\ol W$ which properly contains some $\ol U_i$, so that, in particular, $\dist_{\ol X}(\ol x, \ol y)\le t$. As in \Cref{lem:lifting_trans_and_points}, let $\{x,y,U_1,\ldots,U_{\ol{P}}\}$ be such that $\{x,y,U_i\}$ is pairwise minimal for each $i$. As ${\ol{P}}> P(c_0)$ by assumption, the passing up axiom for $(G,\mf S)$  provides a domain $W\in \mf S$ properly containing some $U_i$  such that $\dist_W(x,y)> c_0$. If $W\neq S$ then, arguing as in Case 1, we obtain $\dist_{\ol W}(\ol x, \ol y)>t$, contradicting our assumption. 
Hence we must have that $\dist_X(x,y)>c_0$, while $\dist_W(x,y)\le c_0$ for every non-maximal domain $W$ properly containing some $U_i$. Let $[x,y]$ be an $X$--geodesic between $x$ and $y$.  By minimality $\dist_{U_i}(x,y)\ge\dist_{\ol U_i}(\ol x, \ol y)>E$, and so the bounded geodesic image axiom (\Cref{defn:HHS}.\ref{axiom:bounded_geodesic_image}) provides a point $p_i\in [x,y]\cap \mc N_E(\rho^{U_i}_S)$  for every $i=1,\ldots, \ol{P}$. Let $\mf P$ be the collection of such points. The set $\mf P$ must be ``well-distributed'' along $[x,y]$, in the following sense. 

\begin{claim}\label{claim:mfP_well_spread}
    If a subset $\mf P'\subseteq \mf P$ has diameter at most $c_2$, then $|\mf P'|\le P(c_3)$. 
\end{claim}
\begin{claimproof}[Proof of \Cref{claim:mfP_well_spread}]
    If $\min_{z\in \mf{P}'}\dist_X(x,z)\le6E$ let $a=x$.  Else, let $a\in [x,y]$ be the vertex such that $\dist_X(x,a)=\lfloor\min\{\dist_X(x,z)\mid z\in \mf{P'}\}-6E\rfloor$, where $\lfloor \cdot\rfloor$ denotes the integer part. In particular, $\dist(a,\mf{P})\in [6E,6E+1]$. Similarly, if $\min_{z\in \mf{P}'}\dist_X(y,z)\le6E$ let $b=y$; otherwise let $b\in [x,y]$ be the point such that $\dist_X(y,b)=\lfloor\min\{\dist_X(y,z)\mid z\in \mf{P'}\}-6E\rfloor$.  See \Cref{fig:ClaimWellSpread}.

    \begin{figure}[htp]
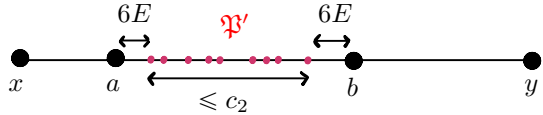

        \centering
        \begin{overpic}[page=1, width=3in, trim={2in 1.8in 3.1in 8in}, clip]{img/HHGFigs}
        \put(3,5){$x$}
        \put(20,5){$a$}
        \put(62,4){$b$}
        \put(93,5){$y$}
        \put(40,15){{\color{red}$\mathfrak P'$}}
        \put(22,17){$6E$}
        \put(57,17){$6E$}
        \put(36,2){$\le c_2$}
        \end{overpic}
        \caption{The construction of the points $a$ and $b$ with respect to the subset $\mathfrak P'$, in the proof of \Cref{claim:mfP_well_spread}.}
        \label{fig:ClaimWellSpread}
    \end{figure}

    Notice that $\dist_{U_i}(x,a)\le E$ for every $i$ such that $p_i\in \mf P'$. This is vacuously true if $a=x$; otherwise $\dist_X([x,a], \rho^{U_i}_S)>E$, or else the segment between $a$ and $\min \mf P'$ would have length at most $3E$. The required inequality now follows from the bounded geodesic image axiom. A symmetric argument holds for $y$ and $b$, so $\dist_{U_i}(a,b)\ge \dist_{U_i}(x,y)-4E>E$. 
    
    If $\mf P'$ contained more than $P(c_3)$ points, then the passing up axiom for $(G,\mf S)$ would imply that $\dist_{W}(a,b)>c_3$ for some $W$ containing some $U_i$; say $U_1\propnest W$. Notice that $W$ cannot be $S$, because \[\dist_X(a,b)\le \diam \mf P'+12E+2\le c_2+12E+2\le c_3.\] Furthermore, if $x\neq a$ then $\dist_X([x,a], \rho^{W}_S)>E$, as otherwise we would have that \[\dist_X(a,\min\mf P')\le 2E+\diam(\rho^W_S\cup \rho^{U_1}_S)\le 5E\] by the consistency axiom \Cref{defn:HHS}.\ref{axiom:consistency}. Thus $\dist_W(x,a)\le E$ by the bounded geodesic image axiom, and the same inequality holds for $b$ and $y$. Hence we again obtain 
    \[\dist_W(x,y)\ge \dist_W(a,b)-4E>c_0,\]
    contradicting our assumption that $x$ and $y$ are $c_0$--close in every domain $W\neq S$.
\end{claimproof}

\noindent Now let $\gamma$ be an $\hat X$--geodesic between $x$ and $y$, let $\widetilde \gamma$ be a de-electrification of $\gamma$, and let $\gamma_0$ be the vertices of $\gamma$ lying in $X$.  Note that $\gamma_0$ consists of at most $t+1$ vertices, since by minimality $\dist_{\hat X}(x,y)=\dist_{\ol X}(\ol x, \ol y)\le t$. Let $\mf Q\subseteq \mf P$ consist of all points in the $(c_2/2)$--neighborhood of $\gamma_0$. which therefore contains at most $(t+1)P(c_3)$ points by \Cref{claim:mfP_well_spread}. Since ${\ol{P}}>(t+1)P(c_3)$, there must be some $U'\in\{U_1,\ldots, U_{\ol{P}}\}$ such that the corresponding vertex $p'\in\mf P$ does not belong to $\mf Q$, that is, $\dist_X(p',\gamma_0)> c_2/2=c_1+D + K$. 

Since $p'\in[x,y]^X$, \Cref{lem:DboundSpriano} produces a point $q'\in \widetilde \gamma$ such that $\dist_X(p',q')\le D$. Furthermore $p'\not \in \mf Q$, which implies that $q'$ lies on a geodesic segment $[c,d]$ of $\widetilde \gamma-\gamma_0$ such that  $\min\{\dist_X(c,q'),\dist_X(d,q')\}>c_1+K$; see \Cref{fig:passingup}. Since $c_1>2A+E$, the vertices $c,d$ do not belong to $\mc N_A(\rho^V_S)$ for any $V\in \mf S$ such that $v_V\in \gamma$. Hence there exists $Y\in \mc Y$ such that $v_Y\in \gamma$ and $c,d\in Y$. Let $r'\in \pi_Y(q')$. Since $Y$ is $K$--quasiconvex, we have $\dist_X(r',q')\le K$, and so  $\min\{\dist_X(c,r'),\dist_X(d,r')\}>c_1$. 

\begin{figure}[htp]
    \centering
    \includegraphics[width=4in]{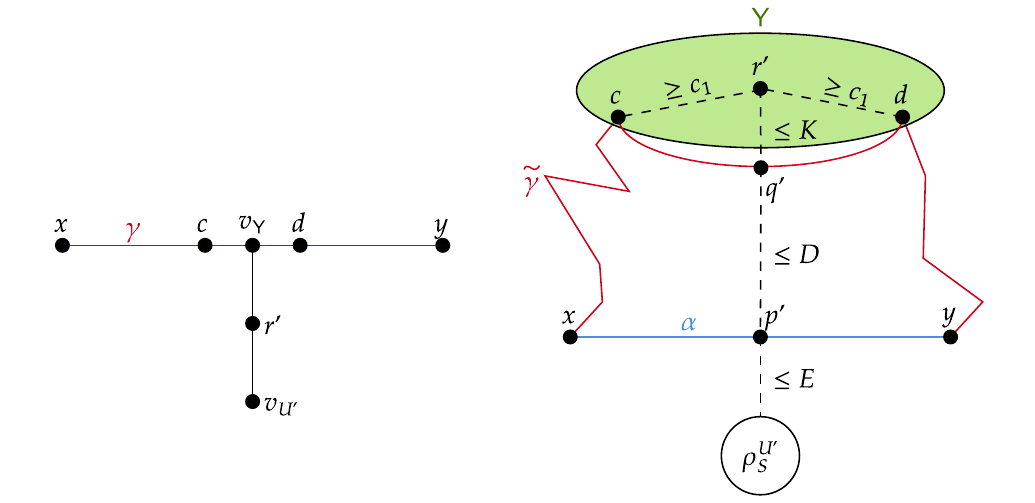}
    \caption{The configuration from the proof of the passing up axiom, as seen from $\hat X$ (on the left) and from $X$ (on the right).}
    \label{fig:passingup}
\end{figure}

Using the triangle inequality and that $\pi_Y$ is $J$--Lipschitz, we have that
\begin{align*}
    \dist^\pi_Y(c,v_{U'})&\ge\dist^\pi_Y(c,r')-\dist^\pi_Y(r',v_{U'})\\
    &\ge \dist_X(c,r')-J\left(\dist_X(r',\rho^{U'}_S)+\diam \rho^{U'}_S\right)\\
    &>c_1-J(K+D+2E)\ge2C. 
\end{align*}
Moreover $\dist^\pi_Y(x,c)\le C$, since the geodesic subsegment of $\gamma$ between $x$ and $c$ does not contain $v_Y$, so by triangle inequality $\dist^\pi_Y(x,v_{U'})>\dist^\pi_Y(c,v_{U'})-\dist^\pi_Y(x,c)>C$.  \Cref{lem:SBGIinHHS} thus yields that every geodesic between $x$ and $v_{U'}$ must pass through $v_Y$. In particular $\dist_{\hat X}(x,v_{U'})=\dist_{\hat X}(x,v_Y)+\dist_{\hat X}(v_Y,v_{U'})$, so the path $\eta_x=\gamma|_{[x,v_Y]}*[v_Y,r']*[r',v_{U'}]$ is a geodesic as it realizes the distance between $x$ and $v_U$. Similarly, the path $\eta_y=\gamma|_{[y,v_Y]}*[v_Y,r']*[r',v_{U'}]$ is also a geodesic. 

Since $H_Y$ acts $\Psi$--coboundedly on $Y$, there exists $h_Y\in H_Y$ such that $d'\coloneq h_Yd$ is at distance at most $\Psi$ from $c$.  Bend the geodesic tripod with sides $\gamma\cup \eta_x\cup \eta_y$ at $v_Y$ by $h_Y$. 
Let $\eta_y'$ be the image of $\eta_y$ after  bending, and let $y'=h_Yy$. Notice that, since we bent a geodesic triangle, $\{x, y',U'\}$ are again minimal, as so were $\{x, y,U'\}$. The situation is as in \Cref{fig:passingup_bended}.

\begin{figure}[htp]
    \centering
    \includegraphics[width=4in]{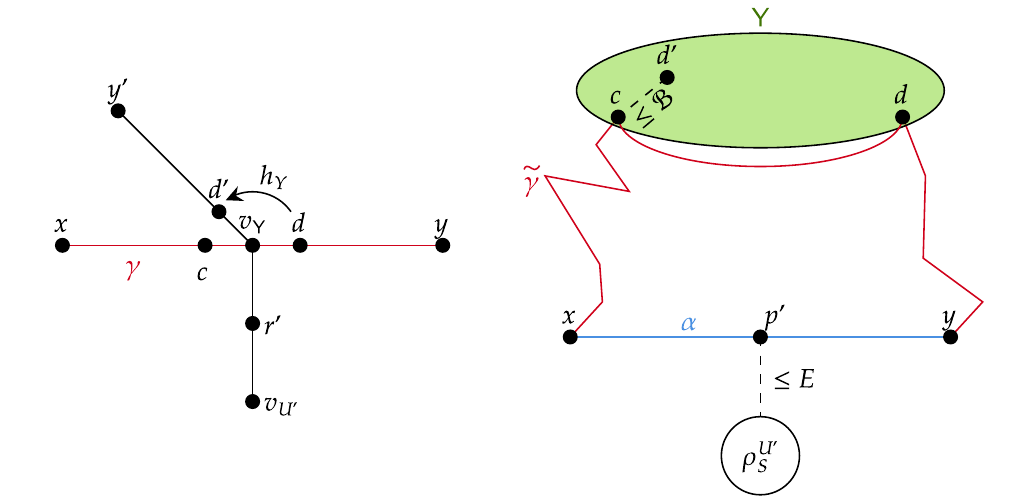}
    \caption{The configuration from the proof of the Passing up axiom, after bending the geodesic connecting $y$ to $v_{U'}$. The points $\{x,y'\}$ are still minimal with $U'$.}
    \label{fig:passingup_bended}
\end{figure}

As $c$ is the last point on the geodesic $\eta_x$ before $v_Y$, \Cref{lem:Unotnearleg_NEW} implies that $\dist_{U'}(x,c)\le E$. The same argument applied to $\eta_y'$ yields that $\dist_{U'}(d',y')\le E$. Finally, we must have that $\dist_{U'}(c,d')\le E$. Indeed, if this was not the case then the bounded geodesic image axiom for $(G, \mf S)$ would give that $\dist_X(\rho^{U'}_S, [c,d'])\le E$ for any $X$--geodesic $[c,d']$. But this would contradict the fact that $p'\not \in \mf Q$, as we would have that
\[\dist_X(p',c)\le \diam \mc N_E(\rho^{U'}_S)+\dist_X(c,d')\le 3E+\Psi<c_1.\]
Combining the above inequalities, we obtain
$$\dist_{U'}(x,y')\le\dist_{U'}(x,c)+\diam\pi_{U'}(c')+\dist_{U'}(c,d')+\diam\pi_{U'}(d')+\dist_{U'}(d',y')\le 5E.$$
However this is a contradiction, because $\{x, y', U'\}$ are pairwise minimal and therefore $\dist_{U'}(x,y')\ge \dist_{\ol U'}(\ol x, \ol y)>5E$. This concludes the proof of the passing up axiom.

\par\medskip\textbf{Relative HHG structure:} In order to complete the proof of \Cref{thm:quotientishhg}, it remains to check that $\mf S/N$ is a relative hierarchically hyperbolic \textit{group} structure on $G/N$. The cofinite $G$--action on $\mf S$ induces a cofinite action of $G/N$ on $\mf S/N$. The action preserves the relations $\nest$ and $\orth$, since  we noticed in \Cref{rem:UniqueReps} that two domains in $\mf S/N$ are nested (resp. orthogonal) if and only if they admit nested (resp. orthogonal) representatives. Furthermore, the isometries $g\colon \mc CU \to \mc C(gU)$ in $(G,\mf S)$ descend to isometries $\ol g \colon \mc C\ol U \to \mc C(\ol g \ol U)$.  That these isometries satisfy the conditions from \Cref{defn:rel_HHG} follows immediately from the fact that they satisfy those conditions in $(G,\mf S)$, along with the fact that all projections and relative projections between domains are defined in terms of minimal representatives, which are permuted by the $G$--action on $\hat X$. 
\end{proof}

\section{Quotients by random walks}\label{sec:RW}

This section introduces background on random walks on acylindrically hyperbolic groups and connections with spinning families. These tools will then be used to provide proofs for \Cref{thm:main_intro}, \Cref{cor:hyp}, and \Cref{cor:relhyp}. 

\subsection{Background on random walks}\label{sec:RWBackground}

Let $\mu$ be a probability distribution on a group $G$. We denote by $\Supp(\mu)$ the \emph{support} of $\mu$, that is, the set of elements $g\in G$ such that $\mu(g)>0$. Let $\Gamma_\mu$ be the semi-group generated by the support of $\mu$. If $\Gamma_\mu$ is, in fact, a subgroup of $G$, then $\mu$ is called \emph{reversible}. We say $\mu$ is \emph{countable} if $\Supp(\mu)$ is countable,  is \emph{finitely supported} if $\Supp(\mu)$ is finite, and  has \emph{full support} if $\Gamma_\mu=G$. Given a fixed acylindrical action of $G$ on a hyperbolic metric space $X$, the probability distribution $\mu$ is \emph{bounded} if some (equivalently, every) orbit of $\Supp(\mu)$ is a bounded subset of $X$ and \emph{non-elementary} if the action of $\Gamma_\mu$ on $X$ is non-elementary.

Given a reversible, non-elementary probability distribution $\mu$ on an acylindrically hyperbolic group $G$, there exists a unique maximal finite subgroup of $G$ normalized by $\Gamma_\mu$ \cite[Lemma 5.5]{H16}. We denote this subgroup by $\mc E_G(\mu)$, or just $\mc E(\mu)$ when $G$ is understood. We note that $\mc E(\mu)$ will always contain the maximal finite normal subgroup of $G$, which we denote by $\mc E(G)$.

\begin{defn}\label{def:permissible}
The measure $\mu$ is \emph{permissible} (with respect to $X$) if it is bounded, countable, reversible, non-elementary, and $\mc E(\mu)=\mc E(G)$. 
\end{defn}

\begin{remark}
    A canonical example of a permissible probability measure is when $G$ is finitely generated and the support of $\mu$ is a finite symmetric generating set of $G$.  In this case $\mu$ will be finitely supported, hence countable and bounded for any action of $G$. In addition, such $\mu$ will have full support and hence be reversible and non-elementary for any non-elementary, acylindrical action of $G$. The fact that $\Gamma_\mu=G$ also implies that $\mc E(\mu)=\mc E(G)$. 
\end{remark} 

\begin{defn}
    $\Gamma_\mu$ acts on $\mc E(G)$ by conjugation, so there exists a map $c_\mu\colon \Gamma_\mu\to \Aut(\mc E(G))$. The \emph{characteristic index} $\sigma_\mu$ is the cardinality of $c_\mu(\Gamma_\mu)$. 
\end{defn}

\begin{remark}
\label{rem:char_index}
    Notice that the characteristic index is $1$ if and only if $\Gamma_\mu$ centralizes $\mc E(G)$, which in particular happens if $\mc E(G)$ coincides with the center of $G$.
\end{remark}

 \begin{hyp}
Throughout this section, we fix a group $G$, a cobounded, acylindrical action of $G$ on a $\delta$--hyperbolic space $(X,\dist)$, and a basepoint $x_0\in X$. Let $\mu_1,\dots, \mu_k$ be $k$ permissible probability measures, and let $w_{1,n},\ldots w_{k,n}$ be independent random walks of length $n$, starting at $x_0$, where the step of each $w_{i,n}$ is chosen according to $\mu_i$.  When $n$ is fixed, we often suppress the $n$ in the notation for a random walk for simplicity, writing $w_i=w_{i,n}$.  If, moreover, a statement applies to all $i$, we may simply write $w$ for $w_{i,n}$. We say that a property holds \emph{asymptotically almost surely}, or \emph{a.a.s}, if it holds with probability approaching 1 as $n\to \infty$.
 \end{hyp}

\begin{prop}[{\cite[Theorem 1.2]{MaherTiozzo18}}]\label{prop:drift}
    For all $i$, there exists $\Delta_i=\Delta_i(G, \mu_i)>0$, called the \emph{drift} of the random walk, such that $\lim_{n\to\infty} \frac{1}{n} \dist(x_0, w_{i,n}x_0)=\Delta_i$ almost surely. 
\end{prop}

 For an isometry $g$ of a hyperbolic space $X$, the \emph{asymptotic translation length} of $g$ is 
    \[\tau(g)= \lim_{k\to\infty} \frac1k \dist(x_0,g^kx_0).
    \]
An element with non-zero asymptotic translation length is a loxodromic isometry, and thus the following proposition implies that $w_{i,n}$ is asymptotically almost surely loxodromic.

\begin{prop}[{\cite[Theorem~1.4]{MaherTiozzo18}}] \label{prop:TransLength}
    For each $i$, $\tau(w_{i,n})>\Delta_i n$ a.a.s.
\end{prop}

Since the $G$-action on $X$ is acylindrical, \cite[Lemma 6.5]{DGO} implies the existence of a maximal virtually cyclic subgroup $\mc E(w)$ of $G$ containing $w$, which consists of all the elements that stabilize any quasi-axis for $w$ up to finite Hausdorff distance. The following proposition states that, for a random element, $\mc E(w)$ is as small as possible:
\begin{prop}[{\cite[Proposition 5.1]{MaherSisto}}]\label{prop:E(w)}
    For every $i$, $\mc E(w_i)=\mc E(G)\rtimes\langle w_i\rangle$ a.a.s.
\end{prop}

\begin{construction}\label{constr_alpha}
    We now produce a $(2,\delta)$--quasi-axis for $w_i$, which we call $\alpha_i$. Let $y\in X$ be such that $\dist(y,w_{i}y)\le \inf_{x\in X}\dist(x,w_{i}x)+\delta$. Given any geodesic $[y,w_{i}y]$, define $$\alpha_i\coloneq \bigcup_{r\in \mathbb{Z}} w_{i}^r[y,w_{i}y].$$
Notice that $\alpha_i$ is $w_i$-invariant by design. Moreover, if the translation length of $w_{i}$ is sufficiently large with respect to $\delta$ (which happens a.a.s.~by \Cref{prop:TransLength}), then $\alpha_i$ is a $7\delta$-quasiconvex $(2,\delta)$-quasigeodesic, by  \cite[Corollary~2.7 and Lemma~3.2]{Coulon}.

It is also useful to fix a geodesic $\gamma_i$ in $X$ from $x_0$ to $w_{i}x_0$. If we are considering a single random walk $w$, we denote the quasi-axis $\alpha_i$ by $\alpha$, and the geodesic $\gamma_i$ by $\gamma$. 
\end{construction} 

 An important tool in the study of random walks on hyperbolic spaces with a $G$--action is matching estimates. We adapt to quasigeodesics the definition of matching from Maher--Sisto \cite{MaherSisto}. 
\begin{defn}
Let $A,B\ge 0$ and $g\in G$. Two quasigeodesics $p$ and $q$ in $X$ have an $(A,B,g)$--\emph{match} if there are subpaths $p'\subseteq p$ and $q'\subseteq q$ of diameter at least $A$ such that $\dist_{Haus}(gp', q')\le B$. If in addition $p=q$, then $p$ has a $(A,B,g)$-\emph{self-match}; in this case, we say that $p$ has a \emph{disjoint} $(A,B,g)$-\emph{self-match} if $p'$ and $q'$ are disjoint. We often drop the element $g$ and/or the constants $(A,B)$ when they are not relevant, and simply speak of a match between $p$ and $q$.
\end{defn}

 For the rest of the section, let $\Delta=\min_i\Delta_i$ be the minimum drift among all random walks.
\begin{prop}[{\cite[Corollary~9.13]{MaherTiozzoCremona}}]\label{prop:selfmatch_gamma}
    Let $w_1, w_2$ be independent random walks of length $n$ with respect to  permissible probability measures.  For any $0<\varepsilon <1$ and any $Q\ge 0$, $\gamma_1$ and $\gamma_2$ do not have a $(\varepsilon \Delta n, Q)$--match a.a.s.
\end{prop}

 We note that \cite[Corollary 9.13]{MaherTiozzoCremona} is stated for disjoint subpaths of a single random walk, but the same argument, with only the obvious changes, proves our proposition as stated. We next control the matches of overlapping segments of $\gamma$.

\begin{prop}\label{prop:GammaGammaMatch_new}
    Let $w$ be a random walk of length $n$ with respect to a permissible probability measure. For any $0<\varepsilon <1$ and any $Q\ge 0$, if $\gamma$ has a $(\varepsilon \Delta n, Q,g)$-self-match then $g\in \mc E(G)$ a.a.s.
\end{prop}
 \noindent Notice that $\mc E(G)$ might act trivially on $X$, so we cannot forbid self-matches tout-court.
\begin{proof}
    This is \cite[Lemma 2.13]{AbbottHull}, which is stated in the case when $\mc E(G)=\{1\}$ but whose proof runs verbatim in the general case.
\end{proof}

 The following proposition can most likely be extracted from \cite[Section 11]{MaherTiozzoCremona}, but we provide a proof for clarity and self-containment.

\begin{prop}\label{prop:Matching_new}
    Let $w_1, w_2$ be independent random walks of length $n$, with respect to permissible probability measures. For every $0<\varepsilon <1$ and every $Q\ge 0$, the following hold a.a.s.:
    \begin{itemize}
        \item the axes $\alpha_1$ and $\alpha_2$ do not have a $(\varepsilon \Delta n, Q)$--match; and
        \item if $\alpha_1$ has a $(\varepsilon \Delta n, Q, g)$-self-match, then $g\in \mc E(w_1)$.
    \end{itemize}
\end{prop}

\begin{proof}
Let $\Omega=\Omega(\delta,7\delta)$ be the constant from \Cref{lem:NppGivesQgeo}, and let $\Phi= \Phi(2,\max\{\Omega,  20\delta\},\delta)$ be the Morse constant provided by \Cref{lem:stability}. Note that $\Omega$ and $\Phi$ depend only on $\delta$.

Now, assume the following hold for $i=1,2$.
        \begin{enumerate}[label=(\alph*{})]
            \item $\tau(w_{i,n})>\Delta n$.
            \item\label{item:bound on varepsilons} $\gamma_1$ and $\gamma_2$ do not have a $(\varepsilon'\Delta n, Q+4\Phi)$--match, and $\gamma_1$ does not have a $(\varepsilon'\Delta n, Q+4\Phi,g)$--self-match unless $g\in \mc E(w_1)$.
        \end{enumerate}
These properties hold a.a.s.~by \Cref{prop:TransLength} and, respectively,   \Cref{prop:selfmatch_gamma} and \Cref{prop:GammaGammaMatch_new}. Now choose $\varepsilon'\in (0,\varepsilon/4)$, and fix $n$ sufficiently large so that the following hold.
    \begin{enumerate}[label=(\arabic*{})]
        \item \label{item:tau} $\Delta n > \Omega$.
        \item \label{item:NBd} $\varepsilon\Delta n/4 - 3Q-4\Phi>\varepsilon' \Delta n$.
    \end{enumerate}
The first is possible because $\Omega$ does not depend on $n$, while the second is possible by our choice of $\varepsilon'<\varepsilon/4$ and the fact that $\Phi$ does not depend on $n$.

We shall show that, if $\alpha_1$ and $\alpha_2$ have a $(\varepsilon\Delta n, Q)$--match, then $\gamma_1$ and $\gamma_2$ have a $(\varepsilon' \Delta n, Q+4\Phi)$--match, contradicting \ref{item:bound on varepsilons}. With minimal differences in the proof one can also show that, if $\alpha_1$ has a $(\varepsilon\Delta n, Q,g)$--self-match with $g\not\in \mc E(w_1)$, then $\gamma_1$ has an $(\varepsilon' \Delta n, Q+4\Phi,g)$--self-match. Hence we shall only highlight the parts where the arguments diverge.

First, notice that $\alpha_i$ is contained in $\mc N_{2\Phi}\left(\bigcup_{j\in \mathbb Z} w_i^j\gamma_i\right)$. To see this, let $y_i\in\pi_{\alpha_i}(x_0)$ belong to the closest point projection of $x_0$ onto $\alpha_i$.  Since $\alpha_i$ is invariant under $w_i$, it follows that $w_iy_i\in \pi_{\alpha_i}(w_ix_0)$; see \Cref{fig:gammawInGammaUnion}. Furthermore, since $\alpha_i$ is $7\delta$-quasiconvex and $\dist(y, w_iy)\ge \tau(w_i)>\Delta n>\Omega$,  \Cref{lem:NppGivesQgeo} yields that any nearest point path $[x_0,y_i]\cup [y_i,w_iy_i] \cup [w_iy_i,w_ix_0]$ is a $(1,\Omega)$--quasigeodesic, and therefore lies in the $\Phi$-neighborhood of $\gamma_i$. In turn, the $(2,\delta)$-quasigeodesic $\alpha_i|_{[y_i,w_iy_i]}$ is contained in the $\Phi$--neighborhood of $[y_i,w_iy_i]$, and therefore in the $2\Phi$-neighborhood of $\gamma_i$.  By applying $w_i^j$ for any $j$, we see that $\alpha_i|_{[w_i^jy_i,w_i^{j+1}y_i]}$ is contained in the $2\Phi$--neighborhood of $w_i^j\gamma_i$, as desired.

\begin{figure}[htp]
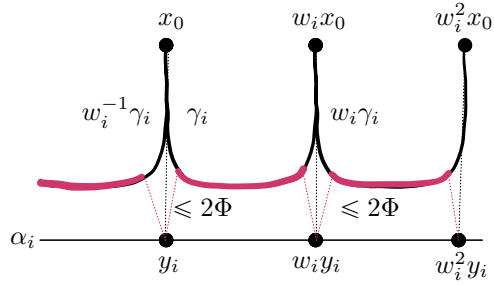

    \centering
    \begin{overpic}[page=1, width=2.5in, trim={1.5in 3.8in 3.5in 5in}, clip] {img/HHGFigs.pdf}
    \put(31,50){$x_0$}
    \put(59,50){$w_ix_0$}
    \put(15,30){$w_i^{-1}\gamma_i$}
    \put(37,30){$\gamma_i$}
    \put(67,30){$w_i\gamma_i$}
    \put(89,50){$w_i^2x_0$}
    \put(31,-1){$y_i$}
    \put(0,4){$\alpha_i$}
    \put(59,-1){$w_iy_i$}
    \put(89,-2){$w_i^2y_i$}
    \put(34,10){$\le 2\Phi$}
    \put(69,10){$\le 2\Phi$}
    \end{overpic}
    \caption{The axis $\alpha_i$ (here, the horizontal line) and the thickened subpaths of $\bigcup_{j\in \mathbb Z} w_i^j\gamma_i$ (in red in the online version) are at Hausdorff distance at most $2\Phi$.}
    \label{fig:gammawInGammaUnion}
\end{figure}

Since $\alpha_1$ and $\alpha_2$ have an $(\varepsilon \Delta n , Q)$--match, there are subpaths $q\subseteq \alpha_{1}$ and $p\subseteq \alpha_{2}$ of diameter $\varepsilon \Delta n$ and $g\in G$ such that $\dist_{Haus}(q,gp)\le Q $. In the case of a self-match, $p,q\subseteq\alpha_1$ and $g\not\in \mc E(w_1)$.  See \Cref{fig:GeomSep}.

We will consider the orbits $\langle w_1\rangle \cdot y_1$ and $\langle w_2^g\rangle \cdot gy_2$.  There are several cases to consider, depending on how these orbits intersect $q$ and $gp$. Note that $w_1$ translates points on $\alpha_1$ by at least $\tau(w_1)\geq \Delta n$. Since $q$ has diameter $\varepsilon \Delta n < \Delta n$, at most one orbit point $w^j_1 y_1$ can lie on $q$. Similarly, the intersection between the orbit $\langle w_2^g\rangle \cdot gy_2$ and $gp$ consists of at most one point.

If no orbit point lies on either $q$ or $gp$, then $q$ and $gp$ are each contained in the $2\Phi$--neighborhood of some $w_1^{j_1}\gamma_1$ and $gw_2^{j_2}\gamma_2$, respectively.  Up to replacing $g$ by $g'=w_1^{-j_1}gw_2^{j_2}$, we can assume that $q\subseteq \mc N_{2\Phi}(\gamma_1)$ and $p\subseteq \mc N_{2\Phi}(\gamma_2)$. In the case of a self-match, use $g'=w_1^{-j_1}gw_1^{j_2}$, which again does not belong to $\mc E(w_1)$. Thus $\gamma_1$ and $\gamma_2$ have an $(\varepsilon\Delta n - 4\Phi, Q+4\Phi,g')$--match.  By \ref{item:NBd}, they therefore have an $(\varepsilon'\Delta n, Q+4\Phi,g')$--match, contradicting \ref{item:bound on varepsilons}.

Now suppose without loss of generality that an orbit point $w_1^jy_1$ lies on $q$. Then $w_1^jy_i$ divides $q$ into two subpaths, $q_1$ and $q_2$, and one of these has diameter at least $\varepsilon\Delta n/2$. Suppose without  loss of generality that it is $q_1$, and let $z\in gp$ be a point at distance at most $Q$ from $w_1^jy_1$.  Consider the subpath $gp_1$ of $gp$ from its initial point to $z$. This subpath is at Hausdorff distance at most $Q$ from $q_1$ and has diameter at least $\varepsilon\Delta n/2 -2Q$.  If no orbit point lies on $gp_1$, then $\gamma_1$ and $\gamma_2$ have a $(\varepsilon\Delta n/2 -2Q-4\Phi, Q + 4\Phi)$--match; in the case of a self-match, this is realized by some $g'\not\in \mc E(w_1)$.  Again by \ref{item:NBd}, this contradicts \ref{item:bound on varepsilons}.  

Finally, suppose that an orbit point $gw_2^{j'}y_2$ lies on $gp_1$. Then this point divides $gp_1$ into two subpaths, $gp_{1,1}$ and $gp_{1,2}$, one of which has diameter at least $\varepsilon\Delta n/4 - Q$.  Without loss of generality, assume it is $gp_{1,1}$; see \Cref{fig:GeomSep}. Let $z'$ be a point on $q_1$ at distance at most $Q$ from $gw_2^{j'}y_2$.  The subpath of $q_1$ from its initial point to $z'$ is at Hausdorff distance at most $Q$ from $gp_{1,1}$, so it has diameter at least $\varepsilon\Delta n/4 - 3Q$.  Again by \ref{item:NBd}, this contradicts \ref{item:bound on varepsilons}.  

In all possible configurations of orbit points, we have reached a contradiction with \ref{item:bound on varepsilons}, so the proof is complete.
\end{proof}

\begin{figure}[htp]
    \centering
    \begin{overpic}[page=1, width=3in, trim={1.5in 6.5in 3.5in 1in}, clip]{img/HHGFigs.pdf}
    \put(100,55){$q\subseteq \alpha_1$}
    \put(100,34){$gp \subseteq \alpha_2$}
    \put(40,72){$w_1^{j}\gamma_1$}
    \put(80,72){$w_1^{j+1}\gamma_1$}
    \put(18,13){$gw_2^{j'}\gamma_2$}
    \put(70,13){$gw_2^{j'+1}\gamma_2$}
    \put(45,26){$gp_1$}
    \put(28,42){$gp_{1,1}$}
    \put(27,63){$q_{1,1}$}
    \put(47,37){$gw_2^{j'}y_2$}
    \put(38,50){$z'$}
    \put(45,45){$\le Q$}
    \put(55,50){$w_1^jy_1$}
    \put(80,59){$q_2$}
    \put(70,45){$\le Q$}
    \put(69,30){$z$}
    \put(50,59){$q_1$}
    \end{overpic}
    \caption{The proof of \Cref{prop:Matching_new} in the case that a cone point lies on both $q$ and $gp$.  The subpaths of translates of $\gamma_1$ and $\gamma_2$ (in red in the online version) form an $(\varepsilon'\Delta n, Q + 4\Phi)$--match.}
    \label{fig:GeomSep}
\end{figure}

\subsection{Random walks and spinning families}\label{sec:RWandSpinning}
In this Section we show that random walks on acylindrically hyperbolic groups and relative HHG fit the framework of spinning families, up to passing to a uniform power. We first set up some notation.
\begin{notation}\label{notation:hyp_main_thm}
    Let $G$ be an acylindrically hyperbolic group. Let $X$ be a hyperbolic space with a cobounded, non-elementary acylindrical $G$-action; up to replacing $X$ via an equivariant quasi-isometry, we can assume that $X$ is a Cayley graph for $G$  with respect to a (possibly infinite) generating set, so that the action is transitive. Let $\mu_1,\ldots, \mu_k$ be permissible probability measures for this action, and let $\sigma=\text{lcm}(\sigma_{\mu_1},\ldots, \sigma_{\mu_n})$ be the least common multiple of all characteristic indices. Consider independent random walks $w_{1,n},\ldots, w_{k,n}$ of length $n$ with respect to these probability measures. For every $i=1,\ldots, k$, let $Y_i=\mc E(G)\cdot \alpha_i$; notice that $\mc E(w_i)$ acts on $Y_i$ a.a.s., since $\mc E(w_i)=\mc E(G)\rtimes \langle w_i\rangle$ a.a.s  by \Cref{prop:E(w)}. Let $\mc Y$ be the union of the $G$-orbits of the $Y_i$, and for every translate $gY_i$ let $H_{gY_i}=g \langle w_i^\sigma\rangle g^{-1}$. Finally, let $N=\llangle H_{Y_i}\rrangle_{i=1,\ldots, k}$. We often drop the indices when they are not relevant.  The exponent $\sigma$ in the definition of $H_{gY_i}$ is crucial for our applications; see \Cref{rem:E_is_trouble} for further discussion.
\end{notation}

\begin{thm}\label{thm:RandomSubgroupIsSpinning}
   In the setting of \Cref{notation:hyp_main_thm}, there are constants $E,K,M_0,R,L$, where $M_0$ and $L$ depend on $n$, such that the collection $(X, \mc Y, G, \{H_Y\}_{Y\in \mc Y})$ a.a.s.~satisfies \Cref{hyp:spinning} and the assumption of \Cref{cor:preserveAH} with respect to $(E,K,M_0,R,L)$. In particular, $G/N$ is a.a.s.~acylindrically hyperbolic.
\end{thm}

\begin{proof}
    \Cref{hyp:complete} extends \Cref{hyp:metric}, so we first verify the latter. First, by assumption $X$ is $E$-hyperbolic for some $E\ge 0$. Moreover, there exists a constant $K=K(E)$ such that every $Y$ is $K$--quasiconvex. To see this, first notice that, as a consequence of, e.g., \cite[Lemma 6.5 \& Theorem 6.14]{DGO}, for every $g\in \mc E(G)$ the translate $g\alpha$ is a $(2,E)$--quasigeodesic with the same ideal endpoints as $\alpha$. Hence any two $\mc E(G)$ translates of $\alpha$ lie at Hausdorff distance at most $\Phi=\Phi(2,E,E)$ by \Cref{lem:stability}. By $E$--slimness of triangles in $X$, this implies that every geodesic $[x,y]$ with $x,y\in Y$ is in the $(E+\Phi)$--neighborhood of a geodesic between points in the same translate of $\alpha$. Since $\alpha$ is $7E$--quasiconvex, it follows that $Y$ is $K$--quasiconvex, where $K=8E+\Phi$.

    Now fix a constant $0<\varepsilon<1$ to be determined later, and let 
    \begin{equation}\label{eq_m_0(e)}
        M_0=\varepsilon \Delta n+4K+4E+2\Phi.
    \end{equation}
    We claim that the family $\mc Y$ is a.a.s.~$M_0$-geometrically separated, as defined in \Cref{lem:geomsep_for_qc}. In other words, we have to show that the diameter of $\mc N_{2K+2E}(Y)\cap Y'$ is at most $M_0$ for every $Y\neq Y'\in \mc Y$. Up to translation and relabeling, we may assume $Y=\mc E(G)\cdot \alpha_1$ and $Y'=g\mc E(G)\cdot\alpha_i$, where if $i=1$ then $g\not\in \mc E(w_1)$. Notice that 
    \[
    \mc N_{2K+2E}(Y)\cap Y'\subseteq \mc N_{2K+2E+\Phi}(\alpha_1) \cap \mc N_{\Phi}(g\alpha_i).
    \]
    If $\diam (\mc N_{2K+2E}(Y)\cap Y')\ge M_0$, then $\alpha_1$ and $g\alpha_i$ would have a $(\varepsilon\Delta n, 2K+2E+2\Phi, g)$-match, and by \Cref{prop:Matching_new} this  a.a.s.~does not occur.

    We now prove \Cref{hyp:spinning}. By assumption the action $G\curvearrowright X$ is transitive, i.e., $R$-cobounded for $R=0$. Furthermore, $G$ acts on $\mc Y$ by construction, and notice that $\Stab_G(Y_i)=\mc E(w_i)$ as every element fixing $Y_i$ preserves the endpoints of $\alpha_i$, up to inversion. In particular, since $\mc E(w_i)=\mc E(G)\rtimes \langle w_i\rangle$ a.a.s., and since $\mc E(G)$ commutes with $w_i^\sigma$ by the definition of the characteristic index, we have that $\Stab_G(Y_i)$ centralizes $H_{Y_i}=\langle w_i^\sigma\rangle$, so that the collection $\{H_Y\}_{Y\in \mc Y}$ is $G$-equivariant. 
    
    The next claim shows that we can choose the spinning constant $L$ greater than any given constant which is bounded linearly in $M_0$. Applying the claim with $\mc L=\ol L$ from \Cref{eq:bound_on_l_spinning} will then prove \Cref{hyp:spinning}.\eqref{I:spinning_bound}.

    \begin{claim}\label{claim:spinning_for_random} Let $\mc L(E,K,M_0)$ be a constant which is bounded linearly in terms of $M_0$. There exist $0<\varepsilon<1$ and $L\ge 0$ such that the following hold a.a.s.:
        \begin{itemize}
            \item $L>\max\{\mc L(E,K,M_0(\varepsilon)),\,\varepsilon \Delta n\}$, where $M_0$ depends on $\varepsilon$ as in \Cref{eq_m_0(e)}; and
            \item for any $Y\in \mc Y$, any $x\neq v_Y\in \hat X$, and any $h\in H_Y-\{1\}$, we have $\dist^\pi_Y(x,hx)>L$.
        \end{itemize}
    \end{claim}
    
    \begin{claimproof}[Proof of \Cref{claim:spinning_for_random}]
        Up to the action of $G$, assume that $Y=Y_i$ for some $i\in \{1,\ldots, k\}$, so that $h=w_{i,n}^{r\sigma}$ for some $r\in \Z-\{0\}$. By \Cref{prop:TransLength}, if $y\in Y_i$ then $\dist(y,w_{i,n}^{r\sigma}y)> \Delta n$. As a consequence, for every $x\in \hat X-\{v_{Y_i}\}$ we have that $\dist_{Y_i}(x,w_{i,n}^{r\sigma}x)>\Delta n-2B$, where $B=B(E,K,M_0)$ is the constant from \Cref{lem:boundedproj_for_qc} that bounds the diameter of $\pi_{Y_i}(x)$. Since $K=K(E)$, and both $B$ and $\mc{L}$ are bounded linearly in $M_0$, we can find constants $a(E)>0$, $b(E)$, and $b'(E)$ such that \[\mc{L}+2B\le a(E)M_0+b(E)=a(E)\varepsilon\Delta n+b'(E).\] Now choose $\varepsilon$ in such a way that $1-a(E)\varepsilon>\varepsilon$, and set $L=\Delta n-2B$. By construction $L\ge \max\{\mc L, \varepsilon\Delta n\}$ for all sufficiently large values of $n$, as required. 
        \end{claimproof}

We now check the assumption of \Cref{cor:preserveAH}. Let $w_{k+1}$ and $w_{k+2}$ be random walks of length $n$ with respect to $\mu_1$ such that $\{w_1,\ldots, w_{k+2}\}$ are pairwise independent. Let $\alpha_{k+1}$ and $\alpha_{k+2}$ be the quasi-axes as in \Cref{sec:RWBackground}, and  let $h,h'\in G$ be such that $x_0\in h\alpha_{k+1}\cap h'\alpha_{k+2}$. Such elements exist because $G$ acts transitively on $X$.  Finally, let $f=hw_{k+1}h^{-1}$ and $g=h'w_{k+2}(h')^{-1}$, which are a.a.s.~loxodromic by \Cref{prop:TransLength} and therefore WPD as the action $G\curvearrowright X$ is acylindrical. 
    
    \begin{claim}\label{claim:fg_indep}
        $f$ and $g$ are a.a.s.~independent.
    \end{claim}  
    \begin{claimproof}
        If $f$ and $g$ share an ideal endpoint $\xi\in \partial X$, then by \Cref{lem:stability} the sub-rays $\eta\subseteq h\alpha_{k+1}$ and $\eta'\subseteq h'\alpha_{k+2}$ connecting $x_0$ to $\xi$ would satisfy $\dist_{Haus}(\eta, \eta')\le \Phi$. In particular, $\alpha_{k+1}$ and $\alpha_{k+2}$ would have a $(\Delta n,\Phi)$--match, contradicting \Cref{prop:Matching_new}. 
    \end{claimproof}

    We now show that the axes of $f$ and $g$ are ``transverse'' to those of the other random walks, in the following sense:

    \begin{claim}\label{claim:shortproj_for_random} 
        $\sup_{Y\in \mc Y}\sup_{m\in \mathbb{Z}}\dist^\pi_Y(x_0,f^mx_0)< L/80$ a.a.s., and similarly for $g$.
    \end{claim} 
    \begin{claimproof}
        Suppose toward a contradiction that there exist $n\in \mathbb{Z}$, $i\in\{1,\ldots, k\}$, and $t\in G$ such that $\dist^\pi_{tY_i}(x_0,f^mx_0)=\diam\pi_{tY_i}(\{x_0,f^mx_0\} )\ge L/80$. Notice that $\diam\pi_{tY_i}(x_0)$ does not depend on $n$, but only on the (uniform) constants $E$ and $K$, while $L$ grows linearly in $n$. In particular, if $n$ is sufficiently large it must be the case that $m\neq 0$. Let $y\in \pi_{tY_i}(x_0)$ and $y'\in \pi_{tY_i}(f^mx_0)$ be such that $\dist(y,y')\ge L/80-1$. By \Cref{lem:NppGivesQgeo}, if $n$ is sufficiently large then the nearest point path $[x_0,y]\cup [y,y']\cup [y', f^mx_0]$ is a $(1,\Omega)$-quasigeodesic, where $\Omega$ only depends on $E$. In turn, since $x_0$ and $ f^mx_0$ belong to $ h\alpha_{k+1}$, \Cref{lem:stability} yields that $y,y'\in \mc N_{\Phi'}(h\alpha_{k+1})$, where $\Phi'=\Phi(2,\Omega,E)$. Let $y''\in Y$ belong to the same translate of $\alpha_i$ as $y$, chosen in such a way that $\dist(y',y'')\le \Phi$.  Then $\dist(y,y'')\ge L/80-1-\Phi$.  Setting $\Phi''=\max\{\Phi, \Phi'\}$, we have $y,y''\in \mc N_{2\Phi''}(h\alpha_{k+1})$. Therefore $\alpha_{k+1}$ and $\alpha_i$ have a $(L/80-5\Phi''-1, 2\Phi'')$--match. Since $L>\varepsilon\Delta n$,  there exists $0<\chi<1$ such that $L/80-5\Phi''-1>\chi \Delta n$ for large enough values of $n$. In particular, $\alpha_{k+1}$ and $\alpha_i$ have a $(\chi \Delta n, 2\Phi'')$--match, contradicting \Cref{prop:Matching_new}.
    \end{claimproof}  
     \noindent As a consequence of \Cref{claim:shortproj_for_random}, $f$ and $g$ are a.a.s.~independent loxodromic WPD elements satisfying 
     \[\sup_{Y\in \mc Y,\,l,m\in \mathbb{Z}}\dist^\pi_Y(f^m x_0,g^lx_0)\le \sup_{Y\in \mc Y,\,m\in \mathbb{Z}}(\dist^\pi_Y(x_0,f^m x_0)+\dist^\pi_Y(x_0,g^m x_0))<L/40.\]
     This shows that the hypotheses of \Cref{cor:preserveAH} are a.a.s.~satisfied, as required.
\end{proof}

\noindent We can specialize \Cref{thm:RandomSubgroupIsSpinning} to the case of relative HHG:
\begin{prop}\label{prop:randomsgr is spinning HHG}
    In the setting of \Cref{thm:RandomSubgroupIsSpinning},
    assume further that $G$ is an acylindrically hyperbolic relative HHG and that $X$ is its top-level coordinate space. Then the collection $(X, \mc Y, G, \{H_Y\}_{Y\in \mc Y})$ a.a.s.~satisfies \Cref{hyp:complete} with respect to some constants $E,K,M_0,R,L$, where $M_0$ and $L$ depend on $n$.
\end{prop}
\begin{proof}
   Since \Cref{hyp:metric} and \Cref{hyp:spinning} hold by \Cref{thm:RandomSubgroupIsSpinning}, we are left to check the remaining assumptions of \Cref{hyp:complete}. It is clear that $G$ acts cofinitely on $\mc Y$.  Furthermore, each $\langle w_i^\sigma\rangle$ acts geometrically on its axis $\alpha_i$, and therefore on $Y_i$, since all $\mc E(G)$-translates of $\alpha_i$ are within finite Hausdorff distance. Finally, if in \Cref{claim:spinning_for_random} we choose $\mc L=\widetilde L$ from \Cref{eq:tildeL}, we can also ensure that $L>\widetilde L$, completing the proof.
\end{proof}

\begin{remark}\label{rem:translength_random} By \Cref{cor:tau}, for every $x\in X$ and every $n\in N-\{1\}$ we have that $\dist_X(x,nx)>\tau=(L/10-2(B+JR))/J=(L/10-2B)/J$; here we used that $R=0$ as $X$ is a Cayley graph for $G$. Since $B$ is bounded linearly in terms of $M_0$, we can choose the constant $\mc L$ from  \Cref{claim:spinning_for_random} to be larger than $20B$ and ensure that $L/10$ grows faster than $2B$ as $n\to \infty$.  In this case, the translation length of the subgroup $N$ will a.a.s.~grow at least linearly in $n$. If there is a single random walk, this answers \cite[Question 1.11.(3)]{MaherTiozzoCremona} in the case of a permissible probability measure.
\end{remark}

\begin{remark}\label{rem:E_is_trouble}
    Taking the $\sigma$-th power of the random walks is crucial to build a spinning family. Indeed, if $w$ is a random walk with respect to a permissible probability measure $\mu$, the probability that $\llangle w\rrangle\cap \mc E(G)\neq \{1\}$ converges to $1-1/\sigma_\mu$ as $n\to \infty$ \cite[Theorem 11.12]{MaherTiozzoCremona}; if this happens, the minimum translation length of $\llangle w\rrangle$ is always zero, contradicting \Cref{rem:translength_random}.  Recall, however, from \Cref{rem:char_index}, that $\sigma=1$ if $\mathcal E(G)$ is central in $G$.
\end{remark}

\begin{remark}\label{rem:varyinglengths}
    In this paper, we have assumed that all of our random walks have the same length.  Some assumptions on the relative lengths of the random walk is necessary for our methods to hold, as we now explain.  The matching estimates introduced in \Cref{sec:RW} are key to the arguments in this section.  If, for example, both $w_1$ and $w_2$ are driven by a uniform measure on the same finite generating set, and the length of $w_{1,n_1}$ is logarithmic in the length of $w_{2,n_2}$, then $w_{1,n_1}$ will a.a.s.~appear as a subword of $w_{2,n_2}$ \cite[Section~4]{SistoTaylor}.  In particular, it will no longer hold that the axes of these two random walks do not have an $(\varepsilon \Delta n_1,Q)$--match a.a.s., as in \Cref{prop:Matching_new}, and so the collection of random walks will not a.a.s.~satisfy \Cref{hyp:complete}.  On the other hand, a straightforward generalization of the techniques in this section should show that \Cref{thm:RandomSubgroupIsSpinning} holds if the lengths of the random walks differ by linear functions.  With more work, it may be possible to extend our methods in the case that $n_2$ is only \emph{polynomial} in $n_1$.  For simplicity and in the interest of space, we chose not to pursue these directions here.
\end{remark}

\subsection{Random quotients of acylindrically hyperbolic groups}\label{sec:RandomQuotients_of_acylhyp}
    In the setting of \Cref{notation:hyp_main_thm}, let $\mc K=\llangle w_{i}\rrangle_{i=1,\ldots, k}$. Notice that if $\sigma=1$ then $N=\mc K$; this happens in particular if $\mc E(G)$ is central in $G$. The following is a more precise version of \Cref{thm:AHQuotient}. 
\begin{thm}\label{thm:rqaH_full}
   In the setting of \Cref{notation:hyp_main_thm}, $G/\mc K$ is acylindrically hyperbolic.
\end{thm}
\begin{proof}
    If $\mc E(G)$ coincides with the center of $G$, then $N=\mc K$ and this is \Cref{thm:RandomSubgroupIsSpinning}, so suppose otherwise. Let $G'=G/\mc E(G)$ and $X'=G/\mc E(G)$. For every $i=1,\ldots, k$, let $\mu_i'$ be the \emph{push-forward} of $\mu_i$ on $ G'$, which, concretely, is defined by mapping every $g'\coloneq g\mc E(G)\in G'$ to $\mu_i'(g')=\sum_{g\in g'}\mu_i(g)$. It is straightforward to check that $G'$ is acylindrically hyperbolic with $\mc E(G')=\{1\}$, and that each $\mu_i'$ is permissible with respect to the action on $X'$.

Let $w_{1},\ldots, w_{k}$ be as in \Cref{notation:hyp_main_thm}, and notice that their images $w_1',\ldots, w_k'$ are independent random walks on $G'$ with respect to $\mu_1', \ldots, \mu_k'$, respectively. By~\Cref{thm:RandomSubgroupIsSpinning}, if we set $\mc K'=\llangle w_1',\ldots, w_k' \rrangle_{G'}$, then $G'/\mc K'$ is acylindrically hyperbolic. Notice that $\mc K$ projects to $\mc K'$ inside $G'$.  Therefore, $G/\mc K$ is a finite extension of the acylindrically hyperbolic group $G'/\mc K'$.  The non-elementary acylindrical action of $G'/\mc K'$ on a hyperbolic space induces a non-elementary acylindrical action of $G/\mc K$ on the same hyperbolic space, and so $G/\mc K$ is acylindrically hyperbolic, as well.
\end{proof}

\subsection{Random quotients of relative HHG}\label{sec:RandomQuotients}
We now apply \Cref{prop:randomsgr is spinning HHG} to prove \Cref{thm:main_intro} and its corollaries. As explained in Remark~\ref{rem:E_is_trouble}, if $G$ is a relative HHG and $\mc E(G)$ is not central in $G$, then we cannot apply our \Cref{prop:randomsgr is spinning HHG} directly, as the quotient appearing there is by the $\sigma$-th power of the random elements. Hence, we first prove every result under the hypothesis that $\mc E(G)$ is central in $G$; then, in order to reduce to the previous case, we invoke \Cref{thm:HHG/E} from the Appendix, which states that $G/\mc E(G)$ is itself a relative HHG.

\begin{proof}[Proof of \Cref{thm:main_intro}]
    Let $(G,\mf S)$ be an acylindrically hyperbolic (relative) HHG, and let $\mu_1,\dots, \mu_k$ be permissible probability measures on $G$.  Let $w_{1,n},\dots, w_{k,n}$ be independent random walks of length $n$ with respect to $\mu_1,\dots, \mu_k$, and let $\mc K=\llangle w_{i}\rrangle_{i=1,\ldots, k}$.

    Suppose first that $\mc E(G)$ is central in $G$. Then \Cref{thm:RandomSubgroupIsSpinning} shows that the $G/\mc K$-action on $\mc C S/\mc K$ is non-elementary, and \Cref{prop:randomsgr is spinning HHG} proves that the assumptions of \Cref{thm:quotientishhg} are satisfied. Therefore $(G/\mc K,\mf S/\mc K)$ is an acylindrically hyperbolic (relative) HHG.

    In the general case, let $G'=G/\mc E(G)$ and $\mc K'=\mc K/(\mc K\cap \mc E(G))$, as in \Cref{thm:rqaH_full}. \Cref{thm:HHG/E} shows that $G'$ is an acylindrically hyperbolic (relative) HHG, and the above arguments yield that $G'/\mc K'$ is an acylindrically hyperbolic (relative) HHG. Finally, $G/\mc K$ is a finite extension of $G'/\mc K'$, so it is an acylindrically hyperbolic (relative) HHG via the action on the same structure.
\end{proof}

\begin{proof}[Proof of \Cref{cor:hyp}]
    Let $G$ be a non-elementary hyperbolic group, and let $w_{1,n},\dots, w_{k,n}$ be independent random walks of length $n$ with respect to permissible probability measures on $G$. Since the class of non-elementary hyperbolic groups is closed under taking finite extensions and quotients by finite subgroups, we can assume that $\mc E(G)=\{1\}$.
    
    Notice that $(G,\mf S)$ is a HHG, where $\mf S=\{S\}$, and $\mc CS$ is the Cayley graph of $G$ with respect to a finite generating set.  By \Cref{thm:main_intro}, if $\mc K=\llangle w_{1,n},\dots, w_{k,n}\rrangle$, then $G/\mc K$ is an acylindrically hyperbolic HHG.  Moreover, the hierarchy structure on $G/\mc K$ is $\mf S/\mc K=\{\ol S\}$, as can be seen from \Cref{constr:HHGStructureQuotient} in the proof of \Cref{thm:quotientishhg}. In particular, $(G/\mc K,\mf S/\mc K)$ has no orthogonality, and so is a rank 1 HHG.  By \cite[Corollary~2.16]{BHS_HHS_Quasiflats}, $G/\mc K$ is a hyperbolic group, and it is non-elementary as it is acylindrically hyperbolic.
\end{proof}

\noindent The proof of \Cref{cor:relhyp} is similar, but uses relative HHG structures instead of HHG structures.
\begin{proof}[Proof of \Cref{cor:relhyp}]
    Let $G$ be non-elementary hyperbolic relative to a finite collection $\mc H$ of infinite, finitely generated subgroups. Let $w_{1,n},\dots, w_{k,n}$ be independent random walks of length $n$ with respect to  permissible probability measures on $G$, and set $\mc K=\llangle w_{1,n},\dots, w_{k,n}\rrangle$.
    
    \textbf{If $\mc E(G)$ is central}: As for \Cref{thm:main_intro}, we first suppose that $\mc E(G)$ is central in $G$. Let $\mc T$ be a finite generating set for $G$, such that $\mc T\cap H$ generates $H$ for every $H\in \mc H$. By \cite[Theorem~9.3]{BHS_HHSII}, there is a relative HHG structure $(G,\mf S)$, where:
    \begin{itemize}
        \item $\mf S=\{S\}\cup G\mc H$;
        \item $\mc C S=\Cay\left(G,\mc T \cup \bigcup_{H\in \mc H} H\right)$, while $\mc C(gH)=g\Cay\left(H,\mc T \cap H\right)$ for every $H\in \mc H$;
        \item every domain $gH$ is nested in $S$ and transverse to every other domain.
    \end{itemize}
    Let $\ol {\mc T}$ be the image of $\mc T$ in $G/\mc K$. For every $H\in H$, let $\ol H$ be the image of $H$ in $G/\mc K$, and let $\ol{\mc H}=\{\ol H\}_{H\in \mc H}$. Notice that every $H\in \mc H$ fixes a set of diameter $1$ in $\mc CS$, while every non-trivial $n\in \mc K$ a.a.s.~has translation length at least $2$ by \Cref{rem:translength_random}. Hence $H\cap \mc K=\{1\}$, and so $H\cong \ol H$, proving the ``moreover'' part of \Cref{cor:relhyp}.\eqref{item:peripherals}.

    By \Cref{thm:main_intro}, the quotient $G/\mc K$ is an acylindrically hyperbolic relative HHG with hierarchy structure $\mf S/\mc K$. The relation between any two domains of $\mf S/\mc K$ is the relation between any minimal representatives, as can be seen from \Cref{constr:HHGStructureQuotient} in the proof of \Cref{thm:quotientishhg}. In particular, $(G/\mc K,\mf S/\mc K)$ has no orthogonality, and so is a rank 1 relative HHG. By combining \cite[Theorem 3.2]{Russell_Rel_Hyp} and \cite[Proposition~5.1]{drutu_relhyp}, we conclude that $G/\mc K$ is non-elementarily hyperbolic relative to a collection $\mc{Q}$ of subgroups with the property that each $Q\in \mc Q$ is at finite Hausdorff distance in $\Cay(G/\mc K, \ol{\mc T})$ from a unique coset of some $\ol{H}\in \ol{\mc H}$. More precisely, the proof of \cite[Proposition 5.1, Step (3)]{drutu_relhyp} shows that one such $Q$ exists for every infinite $\ol H$, and in particular for every $\ol H\in \ol{\mc H}$ by our assumptions.
    
    Let $Q$ and $\ol H\in\ol{\mc H}$ be as above. Up to replacing $Q$ by some conjugate, we can assume that $\dist_{Haus}(Q,\ol H)$ is finite, say, bounded by some $r\ge 0$. We will now show that $\ol H=Q$. Let $K=Q\cap \ol H$. By \cite[Lemma 4.5]{Hruska_Wise} there exists a constant $r'>0$ such that
\[Q\subseteq Q\cap \mc N_r(\ol H)\subseteq \mc N_{r'}(Q\cap \ol H)=\mc N_{r'}(K).\]
Thus, $K$ has finite index in $Q$. The same argument shows that $K$ has finite index in $\ol H$, as well.

Given $\ol g\in \ol H$, the subgroup $K\cap \ol gK\ol g^{-1}$ has finite index in $K$, as it is the intersection of two finite-index subgroups of $\ol H$. Hence $K\cap \ol gK\ol g^{-1}$ has finite index in $Q$ as well, since $K$ has finite index in $Q$. Thus $Q\cap \ol gQ\ol g^{-1}$, which contains $K\cap \ol gK\ol g^{-1}$, also has finite index in $Q$. Now $Q$ is infinite, as it is commensurable to $\ol H\cong H$, and almost malnormal, as it is a peripheral subgroup in a relative hyperbolic structure.  Hence, we must have that $\ol g\in Q$. As $\ol g$ was an arbitrary element of $\ol H$, this proves that $\ol H\le Q$. 
    
For the reverse inclusion, suppose toward a contradiction that there is some $\ol g\in Q-\ol H$. Then $\ol H$ and $\ol g\ol H$  correspond to transverse $\nest$--minimal domains in the relative HHG structure of the quotient. Let $\mathbf P_{\ol H}$ be the product region for $\ol H$ in the quotient structure, as in \Cref{defn:prodreg}, and similarly define $\mathbf P_{\ol g\ol H}$. Since $\ol H$ and $\ol g \ol H$ are $\nest$--minimal, the relative projection $\rho^{\ol H}_{\ol g \ol H}$ is coarsely the nearest point projection of $\mathbf P_{\ol H}$ to $\mathbf P_{\ol g \ol H}$, and similarly for the other relative projection; see \cite[Remark~1.16]{BHS_HHS_AsDim} and \cite[Lemma~3.1]{Russell_Rel_Hyp}.  Since $H$ is infinite,  $\mathbf P_{\ol H}$ and $\mathbf P_{\ol g \ol H}$ cannot be within finite Hausdorff distance, or else we would contradict that the relative projections are bounded diameter sets.   On the other hand, $\mathbf P_{\ol H}$ and $\mathbf P_{\ol g \ol H}$ coarsely coincide with $\ol H$ and $\ol g\ol H$, respectively, and the latter are within finite Hausdorff distance since $\ol g\in Q$.  This is a contradiction, and so we conclude that $\ol H = Q$, as required.

\textbf{General case}: Suppose now that $\mc E(G)$ is not central in $G$, let $G'=G/\mc E(G)$, and let $\mc K'=\mc K/(\mc K\cap \mc E(G))$. By \cite[Lemma 4.4]{AMO_SQ_univ}, every $H\in \mc H$ contains $\mc E(G)$, and $G'$ is hyperbolic relative to $\{H'\coloneq H/\mc E(G)\}_{H\in \mc H}$. In turn, by the above arguments, $G'/\mc K'$ is hyperbolic relative to $\{\widetilde H\}_{H\in \mc H}$, where each $\widetilde H$ is the image of $H'$ in $G'/\mc K'$ and is isomorphic to $H'$ via the quotient map.

Now, $G/\mc K$ is a finite extension of $G'/\mc K'$, and for every $H\in \mc H$ the subgroup $\ol H=H/(H\cap \mc K)$ is the preimage of $\widetilde H$ in $G/\mc K$. Thus $G/\mc K$ is hyperbolic relative to $\{\ol H\}_{H\in \mc H}$: this follows from the general \Cref{lem:relhyp_finite_ext} below, which is surely known to experts and likely follows from \cite{Drutu_Sapir_Rel_hyp}, but we include a proof for completeness. Furthermore, every $\ol H$ is commensurable to $\widetilde H\cong H'$, hence to $H$, and therefore $|H\cap \mc K|<\infty$, as required.
\end{proof}

\begin{lem}\label{lem:relhyp_finite_ext}
    Consider a group extension $1\to E\to G\xrightarrow[]{\pi}G'\to 1$, where $E$ is finite and $G$ is countable. If $G'$ is hyperbolic relative to a finite collection $\{P_i\}_{i=1,\ldots, n}$ of infinite subgroups, then $G$ is hyperbolic relative to $\{\pi^{-1}(P_i)\}_{i=1,\ldots,n}$.
\end{lem}

\begin{proof}
    By Bowditch's equivalent definition of relative hyperbolicity, a countable group is hyperbolic relative to a finite collection $\mc P$ of infinite subgroups if and only if it acts on a fine hyperbolic graph with finite edge stabilizers and finitely many orbits of edges, and $\mc P$ is a collection of representatives of the conjugacy classes of infinite vertex stabilizers (see \cite{Bowditch_relhyp,Hruska_rel_hyp} for further details). Thus let $Y$ be a fine hyperbolic graph on which $G'$ and $\{P_i\}_{i=1,\ldots, n}$ act as above. The quotient map $G\to G'$ induces a $G$-action on $Y$, which again has finite edge stabilizers and finitely many orbits of edges; furthermore, by construction $\{\pi^{-1}(P_i)\}_{i=1,\ldots,n}$ is a collection of representatives of the conjugacy classes of infinite vertex stabilizers. Thus $G$ is hyperbolic relative to $\{\pi^{-1}(P_i)\}_{i=1,\ldots,n}$, as required.
\end{proof}

\section{Common quotients of HHGs}\label{sec:commonquot}

\noindent This section is devoted to the proof of \Cref{thm:commonquot_intro}, which we restate here for the convenience of the reader.
\begin{thm}\label{thm:commonquot}
    If $G_1$ and $G_2$ are acylindrically hyperbolic (relative) HHGs, then there exists $H$ an acylindrically hyperbolic (relative) HHG and surjections $G_1\twoheadrightarrow H$ and $G_2\twoheadrightarrow H$.
\end{thm}
\noindent We split the proof into a series of lemmas. Let $(G_1,\mf S_1)$ and $(G_2, \mf S_2)$ be acylindrically hyperbolic (relative) HHG structures. It is no loss of generality to assume $\mc E(G_i)=\{1\}$ for $i=1,2$, for $G_i\twoheadrightarrow G_i/\mc E(G_i)$ and the latter are acylindrically hyperbolic (relative) HHG by \Cref{thm:HHG/E}.  We also fix finite, symmetric generating sets $\mc{A}_i$ of size $k_i$ for $G_i$, and set $\Gamma_i = \Cay(G_i, \mc{A}_i)$. By \cite[Corollary 2.9]{BHS_HHS_AsDim}, the top level coordinate space $\mc C S_i$ of the (relative) HHG structure for $G_i$ can be taken to be the cone-off $\hat \Gamma_i$ of $\Gamma_i$ with respect to the collection of all proper product regions.

    \begin{lem}\label{claim:freeprodHHG}
        There is an acylindrically hyperbolic (relative) HHG structure $\mf S$ on $G_1*G_2$ whose top-level coordinate space $X$ has the structure of a tree of spaces in which vertex spaces are isometric copies of $\Gamma_1$ and $\Gamma_2$ and edge spaces are points.
    \end{lem}

    \begin{proof}
        The free product $G_1*G_2$ admits a (relative) HHG structure with the following properties (see for example \cite[Theorem~8.24]{BHS_HHSII}, which is stated for HHG but never uses hyperbolicity of coordinate spaces):
        \begin{itemize}
            \item The index set is $\tilde{\mf S}=S\cup  \bigcup_{x\in G_1*G_2/G_1}\{x\mf S_1\}\cup \bigcup_{x\in G_1*G_2/G_2}\{x\mf S_2\}$, where $x\mf S_i$ is a copy of $\mf S_i$, with the same relations;
            \item Every domain is nested in $S$, and any two domains from different copies of $\mf S_i$ are transverse;
            \item $\mc C_{\tilde{\mf S}} S$ is the \emph{Bass-Serre} tree of the splitting, which is the simplicial tree $T$ whose edges are labeled by elements $h \in G_1*G_2$ and vertices are labeled by cosets $hG_i$ for $i = 1,2$;  
            \item For every coset $x\in G_1*G_2/G_i$ and $V\in \mf S_i$, the associated hyperbolic space is $\mc C_{\tilde{\mf S}} (xV)\cong \mc C_{\mf S_i}V$.
        \end{itemize}
        Notice that, for each $i=1,2$, every $U\in \mf S_i$, and every $x\in G_1*G_2$, the product region for $xU$ in $(xG_i,x\mf S_i)$ coarsely coincides with the product region for $xU$ in $(G_1*G_2,\tilde{\mf S})$, as every non-maximal domain outside $x\mf S_i$ is transverse to $xU$ and  all coordinate spaces for $\mf S_i$ are the same as in the corresponding factor.
        
        Arguing as in \cite[Theorem 3.14]{perlmutter}, which is a relative version of \cite[Theorem 3.7]{ABD}, we construct a second (relative) HHG structure on $G_1*G_2$ with the following properties:
        \begin{itemize}
            \item The index set $\mf S$ is obtained from $\tilde{\mf S}$ by removing all translates of $S_1$ and $S_2$;
            \item $\mc C_{\mf S} S$ is the cone-off of $\Cay(G_1*G_2, \mc{A}_1\cup \mc A_2)$ with respect to the product regions of all domains in $\mf S$, excluding $S$ and translates of $S_1$ and $S_2$; and
            \item All other coordinate spaces are unchanged.
        \end{itemize}
        Since $\mc C S_i=\hat \Gamma_i$, and since the product region for $U\in \mf S_i-\{S_i\}$ is the same in $\mf S_i$ and in $\tilde{\mf S}$, up to bounded Hausdorff distance, it follows that $\mc C_{\mf S} S$ is quasi-isometric to the \emph{tree of spaces} $X$ specified by the following data:
        \begin{itemize}
            \item the underlying simplicial tree is $T$;
            \item for every $x\in G_1*G_2$ and $i=1,2$ the vertex space associated to $xG_i$ is a copy of $\hat\Gamma_i$, while the edge space associated to $x$ is $\{x\}$;
            \item for every $x\in G_1*G_2$ and $i=1,2$, the attaching map $\{x\}\to x\hat\Gamma_i$ is the natural inclusion.
        \end{itemize}
        Since $T$ is unbounded, $X$ is unbounded as well. Hence $(G_1*G_2,\mf S)$ is an acylindrically hyperbolic (relative) HHG, as it is not virtually cyclic.
    \end{proof}
    
     For later purposes, we digress to describe geodesics in $X$. Let $p\colon X\to T$ map each vertex space $x\hat\Gamma_i$ to the vertex $xG_i$. For every $x\in G_1*G_2$, denote by $\{x\}$ the edge of $X$ connecting $x\in x\hat\Gamma_1$ and $x\in x\hat\Gamma_2$. Notice that every path in $X$ from $x\hat\Gamma_1$ to $x\hat\Gamma_2$ must pass through $\{x\}$, by the construction of $X$. As a consequence we immediately obtain:
    \begin{lem}\label{claim:X-geo_vs_T-geo}
         A path $\lambda$ in $X$ is an $X$-geodesic if and only if $p(\lambda)$ is a $T$-geodesic and for every $xG_i\in p(\lambda)$, its preimage $\lambda\cap  x\hat\Gamma_i$ is a geodesic in $x\hat\Gamma_i$.
    \end{lem}

    For $i=1,2$ let $\mu_i$ be permissible probability measures on $G_i$ with respect to the action on $\Gamma_i$. Fix $n\in \mathbb N$, and let $w_1,\dots, w_{k_2}$ denote the $n$th steps of $k_2$ independent random walks on $G_1$ chosen according to $\mu_1$, and let $v_1,\dots, v_{k_1}$ denote the $n$th steps of $k_1$ independent  random walks on $G_2$ chosen according to $\mu_2$.  Consider the group
    \[
    H=G_1*G_2/\llangle s_i^{-1}v_i, t_j^{-1}w_j\mid s_i\in \mc{A}_1,t_j\in \mc{A}_2, 1\leq i\leq k_1,1\leq j\leq k_2\rrangle.
    \]
    For convenience, denote the relations by $r_i = s_i^{-1}v_i$ and $R_j = t_j^{-1}w_j$ for $1 \leq i \leq k_1$ and $1 \leq j \leq k_2$; whenever dependencies on $i$ and $j$ are not important, we drop the indices and write $r=s^{-1}v$ and $R=t^{-1}w$.
    
    \begin{lem}
        The groups $G_1$ and $G_2$ both surject onto $H$.
    \end{lem}
    \begin{proof} 
    Given presentations $G_1 = \langle \mc{A}_1 \mid \mf{R}_1 \rangle$ and  $G_2 = \langle \mc{A}_2 \mid \mf{R}_2 \rangle $, straightforward Tietze transformations give the following alternate presentations 
\[
G_1 = \langle \mc{A}_1 \cup \mc A_2 \mid \mf{R}_1 \cup \{ R_j \mid 1 \leq j \leq k_2 \} \rangle \quad\text{and}\quad G_2 = \langle \mc{A}_1 \cup \mc A_2 \mid \mf{R}_2 \cup \{ r_i \mid 1 \leq i \leq k_1 \} \rangle .
\]
    We have the following commutative diagram, where every arrow is a surjection and its edge label denotes the corresponding kernel.
    \begin{center}
    \begin{tikzcd}
             & F_{k_1}*F_{k_2} \ar[dl, "\llangle\mf{R}_1\rrangle"']\ar[dr, "\llangle\mf{R}_2\rrangle"]\ar[dd]\\
        G_1*F(\mc{A}_2) \ar[d, "\llangle \{ R_j \} \rrangle"'] & & F(\mc{A}_1)*G_2 \ar[d, "\llangle \{r_i\} \rrangle"] \\
        G_1 \ar[dr, "\llangle \{r_i\} \rrangle"'] \ar[r, lightgray, hook] & G_1*G_2 \ar[d] & G_2 \ar[dl, "\llangle \{ R_j \} \rrangle"] \ar[l, lightgray, hook] \\
            &  H
    \end{tikzcd}
    \end{center}
It is then immediate that both $G_1$ and $G_2$ surject onto $H$.
    \end{proof}
    
     It remains to prove that $H$ is an acylindrically hyperbolic (relative) HHG. More precisely, we will show that there is a collection of subspaces $\mc Y$ and subgroups $H_Y$ for each $Y\in \mc Y$ such that $(X,\mc Y, G_1*G_2,\{H_Y\}_{Y\in \mc Y})$ a.a.s.~satisfies \Cref{hyp:complete} and the assumption of \Cref{cor:preserveAH}.  We will then conclude by invoking \Cref{thm:quotientishhg}.

   \begin{construction}
   The cyclically reduced word $r=s^{-1}v$ is a loxodromic isometry of $X$, and we now describe an $r$-invariant geodesic of $X$. By \Cref{prop:drift}, the element $v$ is a.a.s.~loxodromic on $\hat\Gamma_2$. Let $\gamma_{v}$ be a geodesic in $\hat\Gamma_2$ connecting $1$ and $v$, and let $e_{s^{-1}}$ be the edge of $\hat\Gamma_1$ connecting $1$ to $s^{-1}$. Hence the concatenation 
    $$\theta=\{1\}*e_{s^{-1}}*\{s^{-1}\}*s^{-1}\gamma_{v}$$
    connects $1\in \hat\Gamma_2$ to $r\in r\hat\Gamma_2$. Now let 
    $$\lambda=\bigcup_{n\in \Z} r^n\theta.$$ 
    Notice that $p(\lambda)$ is the geodesic axis for the action of $r$ on $T$, so $\lambda$ is an $X$--geodesic by \Cref{claim:X-geo_vs_T-geo}.  It is also clear from the construction of $\lambda$ that its intersection with the vertex spaces it passes through alternates between an edge inside a translate of $\Gamma_1$ and a geodesic in a translate of $\Gamma_2$ whose length $\ell(\gamma_{v})$ a.a.s.~goes to infinity as $n\to\infty$.
    
    Similarly, for $R=t^{-1}w$ we can define $\Theta=\{1\}*e_{t^{-1}}*\{t^{-1}\}*t^{-1}\gamma_{w}$, which connects $1\in \hat \Gamma_1$ to $R\in R\hat \Gamma_1$, and then $\Lambda=\bigcup_{n\in \Z} R^n\Theta$  is an $R$-invariant $X$-geodesic.         
   \end{construction}
 
  Let $\mc Y$ be the collection of $G$--orbits of the axes $\lambda_i$ and $\Lambda_j$ for $1\leq i\leq k_1$ and  $1\leq j\leq k_2$.   For each $Y\in \mc Y$, we have $Y=g\lambda_i$ or $Y=g\Lambda_j$ for some $g\in G_1*G_2$ and some $1\leq i\leq k_1$ or  some  $1\leq j\leq k_2$.  If $Y=g\lambda_i$, then let $H_Y:=\langle r_i\rangle^g$, and if $Y=g\Lambda_j$, let $H_Y:=\langle R_j\rangle ^g$.  Finally, let $N=\llangle H_Y\mid Y=\lambda_i \text{ for } 1\leq i\leq k_1 \text{ or }Y=\Lambda_j \text{ for } 1\leq j\leq k_2\rrangle$. 

   \begin{prop}\label{prop:commonquot_checkhyp}
       $(X,\mc Y, G_1*G_2,\{H_Y\}_{Y\in \mc Y})$ a.a.s.~satisfies \Cref{hyp:complete} and the requirements of \Cref{cor:preserveAH}.
   \end{prop}

   \begin{proof}
       The proof is analogous to those of \Cref{thm:RandomSubgroupIsSpinning} and \Cref{prop:randomsgr is spinning HHG}. 

        We first verify \Cref{hyp:metric}.  By assumption, $X$ is $E$--hyperbolic, where $E$ is any HHG constant for $(G_1*G_2,\mf S)$. Each $Y\in \mc Y$ is an $X$-geodesic and hence $K$-quasiconvex for some $K(E)$ given by \Cref{lem:stability}. Let $Y,Y'\in \mc Y$, and suppose without loss of generality that $Y=\lambda_i$ for some $i$. 
        
        \begin{claim}
            If $Y\neq Y'$ then $p(Y)\cap p(Y')$ has diameter at most 1 in $T$.
        \end{claim} 
        \begin{claimproof}
        Recall that $p(Y)$ is the axis in $T$ for the action of $r_i=s_i^{-1}v_i$, which is a product of two elliptic elements on adjacent vertices and therefore has translation length $2$ on $p(Y)$. Suppose that $\diam_T(p(Y)\cap p(Y'))\ge 2$. Up to translation by a power of $r_i$ we can assume that the edge labelled by $1$ belongs to the intersection, so that $Y'$ must be of the form $\lambda_{i'}$ for some $1\le i'\le k_1$, or $\Lambda_j$ for some $1\le j\le k_2$. Since $\diam_T(p(Y)\cap p(Y'))\ge 2$ and the intersection of two lines in a tree is connected, it follows that one of the vertices $r_iG_2$ or $v_i^{-1}G_1$ belongs to the intersection, since these are the vertices of $p(Y)$ which are adjacent to the edge labelled by $1$. There are thus four cases to consider: 
        \begin{itemize}
        \item If $Y'=\lambda_{i'}$ and the vertex $r_iG_2$ belongs to the intersection, then $s^{-1}_i G_2=s^{-1}_{i'} G_2$.  Since $s_i$ and $s_{i'}$ belong to $G_1$, we have $i=i'$. This violates the assumption that $Y\neq Y'$.
        \item If $Y'=\lambda_{i'}$ and the vertex $v_i^{-1}G_1$ belongs to the intersection, then $v^{-1}_iG_1=v^{-1}_{i'}G_1$, which again implies that $i={i'}$ as $v_i$ and $v_{i'}$ do not coincide a.a.s.
        \item If $Y'=\Lambda_j$ and the vertex $r_iG_2$ belongs to the intersection, then $R_jr_iG_2=G_2$, and therefore $R_jr_i=t_j^{-1}w_js_i^{-1}v_i\in G_2$. This would imply that $w_js_i^{-1}=1$, which is a.a.s.~not the case as the translation length of $w_j$ on $\hat\Gamma_1$ is a.a.s.~bigger than that of $s_i^{-1}$.
        \item Finally, if $Y'=\Lambda_j$ and the vertex $v_i^{-1}G_1$ belongs to the intersection, then $R_j^{-1}r_i^{-1}G_1=G_1$, and we get a contradiction as above.\qedhere
        \end{itemize} 
        \end{claimproof}

         We now turn to geometric separation. More precisely, fixing a constant $0<\varepsilon < 1$ sufficiently small,  we will bound the diameter of $\mc N_{2K+2E}(Y) \cap Y'$  by a constant $M_1$ that is on the order of $\varepsilon \Delta n$, where $\Delta$ is the maximum of the drifts of $v_1,\dots, v_{k_1}$ on $\Gamma_2$ and of $w_1,\dots, w_{k_2}$ on $\Gamma_1$.

        Suppose first that $Y'=g\Lambda_j$ for some $g\in G_1*G_2$ and some $1\le j\le k_2$. Then the intersection of any common vertex space with exactly one of $Y$ and $Y'$ is an edge, while its intersection with the other is a geodesic whose length a.a.s.~grows linearly in $n$.  It follows that the diameter of $\mc N_{2K+2E}(Y) \cap Y'$ is bounded by a constant depending only on $E$ and $K$; in particular, this constant is independent of $n$.  Now suppose $Y'=g\lambda_{i'}$ for some $g\in G_1*G_2$ and some $1\le i'\le k_1$.  In this case, we again consider the intersection of $Y$ and $Y'$ with a common vertex space, but now the intersections are either both an edge or both a geodesic segment whose length is a.a.s.~growing with $n$.  For vertex spaces in which the latter holds, the geodesic segments are translates of $\gamma_{v_i}$ and $\gamma_{v_{i'}}$, where we allow the possibility that $i={i'}$.  By \Cref{prop:selfmatch_gamma} and \Cref{prop:GammaGammaMatch_new}, if these are distinct translates of $\gamma_{v_i}$ and $\gamma_{v_{i'}}$, then they do not have an $(\varepsilon\Delta n, 2K+2E)$--match a.a.s.  Therefore the diameter of $\mc N_{2K+2E}(Y) \cap Y'$ is on the order of $\varepsilon \Delta n$.  In either case, we see that $\mc Y$ is $M_1$--geometrically separated  where $M_1$ is on the order of $\varepsilon \Delta n$. This verifies \Cref{hyp:metric}.

        We now verify \Cref{hyp:spinning} and \Cref{hyp:complete} simultaneously. The free product $G_1*G_2$ acts transitively on $X$ and cofinitely on $\mc Y$ by construction, and each $H_Y$ acts geometrically on $Y$. Furthermore, as in the proof of \Cref{claim:spinning_for_random}, we can choose $\varepsilon$ small enough to ensure that both \Cref{hyp:spinning}.\eqref{I:spinning_bound} and \Cref{hyp:complete}.\eqref{I:def_of_widetildeL} hold. Regarding \Cref{hyp:spinning}.\eqref{I:equivariance}, as in the proof of \Cref{thm:RandomSubgroupIsSpinning}, it is enough to check that $\Stab_{G_1*G_2}(Y)=H_Y$. Hence assume without loss of generality that $Y=\lambda_i$ for $1\le i\le k_1$, and let $r_i=s_i^{-1}v_i$. As the argument will not depend on $i$, we drop all indices. By \cite[Lemma 6.5]{DGO}, we have $\Stab_{G_1*G_2}(Y)\le \mc E(r)$, so it suffices to prove the following claim.
        \begin{claim}
            $\mc E(r)=\langle r\rangle$.
        \end{claim}
        
        \begin{claimproof}
            Again by \cite[Lemma 6.5]{DGO}, the group $\mc E(r)$ is also the setwise stabiliser of the axis $p(\lambda)$, since the action on the Bass-Serre tree $T$ is acylindrical. Let $g\in \mc E(r)$.
            Up to multiplying $g$ by a power of $r$, we can assume that $g$ fixes $\Gamma_2$ and therefore lies in $G_2$. Since $g$ preserves the axis $p(\lambda)$, we must have $g\Gamma_1\in \{\Gamma_1,v^{-1}\Gamma_1\}$.  If $g\Gamma_1=\Gamma_1$ then $g\in G_1\cap G_2=\{1\}$. Otherwise, $g\Gamma_2=v^{-1}\Gamma_2$, so $g=v^{-1}$ as they both lie in $G_2$. However, in this case $g$ must be an order 2 element that acts as a reflection on $p(\lambda)$, which is impossible since $v$ has infinite order.
        \end{claimproof}

        Finally, we prove that the quotient $H$ is again a acylindrically hyperbolic (relative) HHG by checking the requirements of \Cref{cor:preserveAH}. Let $z_1,z_2$ be random walks of length $n$ on $\hat \Gamma_1$ with respect to $\mu_1$ such that $\{w_1,\ldots, w_{k_2},z_1,z_2\}$ are pairwise independent. Let $\alpha_{z_1}$ and $\alpha_{z_2}$ be the quasiaxes for $z_1$ and $z_2$, respectively, defined as in \Cref{constr_alpha}, and let $f,g$ be conjugates of $z_1,z_2$, respectively, such that the quasiaxes for $f$ and $g$ both contain $1\in \hat\Gamma_1$. Notice that the ideal endpoints of $f$ and $g$ in $\partial \hat\Gamma_1$ are a.a.s.~disjoint by \Cref{claim:fg_indep}; moreover, since the Gromov boundary of $\hat \Gamma_1$ embeds in that of $X$, $f$ and $g$ are independent loxodromic isometries of $X$ as well. As in the proof of \Cref{thm:RandomSubgroupIsSpinning}, it  now suffices to check the following:
        \begin{claim}
            Denote the vertex $1\in \hat\Gamma_2$ by $x_0$. Then 
        $\sup_{Y\in \mc Y}\sup_{m\in \mathbb{Z}}\dist^\pi_Y(x_0,f^mx_0)< L/80$ a.a.s., and similarly for $g$.
        \end{claim}
        
        \begin{claimproof}
            Since $f$ preserves $G_1$, the projection of $\{f^mx_0\}_{m\in\Z}$ to any $X$-geodesic $\eta$ is a point if $\eta\cap \hat\Gamma_1=\emptyset$, and otherwise it coincides with the projection of $\{f^mx_0\}_{m\in\Z}$ to $\eta_1\coloneq \eta\cap \hat\Gamma_1$. Hence it is enough to show that
        $$\max_{i=1,\ldots, k_2}\sup_{g\in G_1}\sup_{m\in \mathbb{Z}}\dist^\pi_{g\gamma_{w_i}}(x_0,f^mx_0)< L/80.$$
        Arguing as in \Cref{claim:shortproj_for_random}, if the above was not true then $\alpha_{z_1}$ and some the quasiaxis $\alpha_{w_i}$ for $w_i$ (defined as in \Cref{constr_alpha}) would a.a.s.~have a match of length on the order of $\chi \Delta n$ for some $0<\chi<1$, contradicting \Cref{prop:GammaGammaMatch_new}.
        \end{claimproof} 
        \noindent The proof of \Cref{prop:commonquot_checkhyp} is now complete.
   \end{proof}
   \noindent \Cref{thm:commonquot} now follows from \Cref{thm:quotientishhg} and \Cref{prop:commonquot_checkhyp}.

We shall now prove \Cref{cor:common_rel_hyp_quotient_intro} about the existence of common quotients of relatively hyperbolic groups. The result will be implied by the following more explicit and technical result.

\begin{cor}\label{cor:relhyp_common_quotient}
    If $(G_1,\mathcal P_1)$ and $(G_2,\mathcal P_2)$ are non-elementary relatively hyperbolic groups with infinite, finitely generated peripheral subgroups, then there exists a relatively hyperbolic group $(H,\mathcal Q)$ and surjections $G_1\twoheadrightarrow H$ and $G_2\twoheadrightarrow H$ such that each $P\in \mathcal P_1\cup \mathcal P_2$ is commensurable to some $Q\in \mathcal Q$ and each $Q\in \mathcal Q$ is commensurable to some $P\in P_1\cup P_2$.  Moreover, if $\mathcal E(G_1)=\mathcal E(G_2)=\{1\}$, then $\mathcal Q=\mathcal P_1\cup \mathcal P_2$. 
\end{cor}

\begin{proof}
    Consider the groups $G_i'= G_i/\mathcal E(G_i)$, for $i=1,2$, which are relatively hyperbolic groups with peripheral structure $\mathcal P_i'=\{P/(\mathcal E(G_i)\cap P)\mid P\in \mathcal P_i\}$, again by \cite[Lemma 4.4]{AMO_SQ_univ}.  Note that if $\mathcal E(G_i)=\{1\}$, then $G_i'=G_i$ and $\mathcal P_i'=\mathcal P_i$.
    
    Let $\mathcal A_i$ be finite, symmetric generating sets for $G_i'$.  As described in the proof of \Cref{cor:relhyp}, there are relative HHG structures $(G_i', \mathfrak S_i)$, where $\mathfrak S_i=\{S_i\}\cup G_i'\mathcal P_i'$.  Moreover, $G\coloneq G_1'*G_2'$ is hyperbolic relative to $\mathcal P_1'\cup\mathcal P_2'$ and, as a relative HHG, has index set $\mathfrak S=\{S\}\cup G\mathcal P_1' \cup G\mathcal P_2'$. 
    Notice also that $\mathcal E(G)=\{1\}$, as every finite normal subgroup would be contained in one of the free factors.   Following the proof of \Cref{thm:commonquot}, the quotient $H=G_1'*G_2'/K$, where $K=\llangle s_i^{-1}v_i, t_j^{-1}w_j\mid s_i\in \mc{A}_1,t_j\in \mc{A}_2, 1\leq i\leq k_1,1\leq j\leq k_2\rrangle$
    is an acylindrically hyperbolic relative HHG with the property that $G_i$ surjects onto $H$ for $i=1,2$.  The relative HHG structure of $H$ is $\mathfrak S/K$.  Using the same arguments as in the proof of \Cref{cor:relhyp}, $(H,\mathfrak S/K)$ is a rank 1 relative HHG and therefore a non-elementary hyperbolic group relative to $\mathcal Q=\mathcal P_1'\cup\mathcal P_2'$.
\end{proof}

\begin{proof}[Proof of \Cref{cor:common_rel_hyp_quotient_intro}]
Let $(G_1,\mathcal P_1)$ be a non-elementary relatively hyperbolic group with infinite, finitely generated peripheral subgroups.  Take $(G_2,\mathcal P_2)$ to be any hyperbolic group with property (T) such as those constructed in \cite{Generalised_triangle}, where we set $\mathcal P_2=\emptyset$.  \Cref{cor:relhyp_common_quotient} yields a non-elementary relatively hyperbolic common quotient $(H,\mathcal Q)$ with property (T) such that each $Q\in \mathcal Q$ is commensurable to some $P\in\mathcal P_1$.
\end{proof}

\appendix
\section{Quotients of HHG by finite kernels}\label{sec:appendix}

In this appendix, we show that a quotient of an HHG by a finite subgroup is again an HHG. In particular, if $G$ is an acylindrically hyperbolic HHG and $\mc E(G)$ is its maximal finite normal subgroup, then $G/\mc E(G)$ is again an acylindrically hyperbolic HHG, and its maximal finite normal subgroup is trivial. In other words, up to commensurability one can always assume that an acylindrically hyperbolic HHG has trivial maximal finite normal subgroup, circumventing all fine algebraic issues that the latter might cause.
\begin{thm}\label{thm:HHG/E}
    Let $(G,\mf S)$ be a (relative) HHG, and let $K\unlhd G$ be a finite normal subgroup. Then $G/K$ is a (relative) HHG, and if $G$ is  acylindrically hyperbolic, then so is $G/K$.
\end{thm}

\begin{proof} Let $\mc X$ be a Cayley graph for $G$. Consider a (relative) HHG structure $(\mc X, \mf S)$ for $G$, with (relative) HHG constant $E$. Since $G$ acts cofinitely on $\mf S$, there are finitely many isometry types of coordinate spaces, so we can assume that $E$ is larger than the diameter of any bounded domain. We shall prove that $G/K$ admits a (relative) HHG structure where:
\begin{itemize}
    \item the index set is $\ofS\coloneq \mf S'/K$, where  
    $$\mf S'=\{U\in \mf S\mid \exists V\nest U \text{ with } \mc C V \mbox{ unbounded}\};$$
    \item two domains $[V], [W]\in \ofS$ are orthogonal (resp. nested) if they admit orthogonal (resp. nested) representatives; and
    \item for every $W\in \mf S'$, $\mc C W$ and $\mc C [W]$ are uniformly quasi-isometric.  
\end{itemize}
We break the proof into a sequence of steps.

 \par\medskip
\textbf{Step 0: Removing unnecessary bounded domains.} 
We first check that $(G,\mf S')$ is  a (relative) HHG structure. By inspection of \Cref{defn:HHS}, removing the bounded domains without unbounded nested domains can only affect the existence of containers and the validity of the large link axiom, so we must check that both still hold in $\mf S'$.

    \begin{itemize}
\item \textbf{Containers:} Let $V,W\in \mf S'$ be such that $V\propnest W$, and there exists $U\in \mf S'$ which is nested in $W$ and orthogonal to $V$. By the containers axiom in $\mf S$, there exists $T\in \mf S$ such that $T\propnest W$ and $T$ contains all $U$ as above. Since $U\in \mf S'$, it contains an unbounded domain, and in turn this means that $T\in \mf S'$. Hence $T$ is the container for $V$ inside $W$ in $\mf S'$.
\item \textbf{Large links:} The passing up axiom \ref{axiom:passing_up} in $\mf S$ already involved only unbounded domains, since we assumed that the HHG constant $E$ was larger than the diameter of any bounded domain. Hence the passing up axiom holds in $\mf S'$ as well, and this implies the large link axiom by \Cref{lem:passingupisenough}.
\end{itemize}

 \par\medskip
    \textbf{Step 1: The kernel fixes every unbounded domain.} Let $\wfS\subseteq \mf S'$ be the collection of unbounded domains. In this step we show that the $K$-action on $\wfS$ is trivial. Fix $k\in K-\{1\}$, and let $W\in \wfS$. If $W=S$ then clearly $kW=S$, so suppose $W\neq S$.  By assumption, $\mc CW$ is unbounded. For every $g,h\in G$,
    $$\dist_{kW}(h, gh)=\dist_{W}(k^{-1}h, k^{-1}gh).$$
    The triangle inequality and the fact that $K$ is normal in $G$ yield that
    \begin{align*}
        \left|\dist_{W}(k^{-1}h, k^{-1}gh)-\dist_{W}(h, gh)\right|&\le \dist_{W}(h, k^{-1}h)+\dist_{W}(gh, k^{-1}gh)\\
        &\le \max_{k'\in K}\left(\dist_{h^{-1}W}(1, k')+\dist_{(gh)^{-1}W}(1, k')\right)\\
        &\le 2\sup_{V\in \mf S}\max_{k'\in K}\dist_{V}(1, k').
    \end{align*}
    The quantity $M\coloneq 2\sup_{V\in \mf S}\max_{k'\in K}\dist_{V}(1, k')$ is bounded in terms of $\diam_{\mc X}K$, since coordinate projections are uniformly coarsely Lipschitz. Therefore,  distances in $W$ and in $kW$ must agree up to an error of $M$. 
    
    We now show that $kW=W$, by excluding every other possibility.
    \begin{itemize}
        \item If $W\propnest kW$ and the order of $k$ is $m$, then $W\propnest kW\propnest k^2W\ldots \propnest k^m W=W$, a contradiction. The same argument excludes that $kW\propnest W$. 
        \item If $kW\orth W$, then since $\mc C W$ is unbounded, the partial realization axiom~\ref{axiom:partial_realisation} can be used to find an element $g\in G$ such that $\dist_{kW}(1,g)\le E$, while $\dist_{W}(1,g)\ge M+E+1$. This would contradict the fact that $|\dist_W(1,g)-\dist_{kW}(1,g)|\le M$.
        \item Finally, suppose $W\trans kW$. Since $\mc C W$ is unbounded, we can find elements $g,h\in G$ such that $$\min\{\dist_W(\rho^{kW}_W, h), \dist_W(\rho^{kW}_W, gh), \dist_W(h, gh)\}\ge M+3E+1.$$ 
        However the consistency axiom~\ref{axiom:consistency} would give that 
        $$\dist_{kW}(h, gh)\le \dist_{kW}(h, \rho^{W}_{kW})+\dist_{kW}(\rho^{W}_{kW}, gh)+\diam\rho^{W}_{kW}\le 3E, $$
        again contradicting that $|\dist_W(h,gh)-\dist_{kW}(h,gh)|\le M$.
    \end{itemize}

    \par\medskip\textbf{Step 2: The kernel acts with uniformly bounded orbits.} 
    We now prove that there exists a constant $C$ such that, for every $W\in \wfS$ and $p\in \mc C W$, $\diam_{W} (K\cdot p)\le C$. 
    
    As noted above, $\sup_{W\in \wfS}\diam\pi_W(K)$ is bounded in terms of $\diam_{\mc X}K$, since coordinate projections are coarsely Lipschitz. Furthermore, since the latter are also uniformly coarsely surjective, it is enough to uniformly bound $\sup_{g\in G}\sup_{W\in \wfS}\diam\pi_W(Kg)$. To this extent, notice that
    $$\sup_{g\in G}\sup_{W\in \wfS}\diam\pi_W(Kg)=\sup_{g\in G}\sup_{W\in \wfS}\diam\pi_{g^{-1}W}(K)=\sup_{W\in \wfS}\diam\pi_W(K), $$
    where we used that every $g\in G$ normalizes $K$.

    \par\medskip\textbf{Step 3: HHG structure of the quotient.}  We shall now prove that $G/K$ has the following HHG structure. Let $\mc Y=\mc X/K$. Let $\ofS=\mf S'/K$. Given a domain $W\in \mf S'$, let $[W]$ be its image in $\ofS$. Two domains $[U], [V]\in \ofS$ are orthogonal (resp. nested) if they admit orthogonal (resp. nested) representatives $U\in [U]$ and $V\in [V]$. Set 
    $$\mc C [W]=\left(\bigcup_{W\in [W]}\mc C W\right)/K,$$ and for any representative let $q_W\colon \mc C W\to \mc C [W]$ be the quotient projection. This map is $1$-Lipschitz. Moreover, it is a uniform quality quasi-isometry, because $K$-orbits are uniformly bounded in $\mc C W$.
    
    Given $[W]$, if $\mc C W$ is bounded then so is $\mc C [W]$, and projections can be defined arbitrarily. Otherwise there is a single representative $W\in [W]$, and given $[x]\in \mc Y$ let $$\pi_{[W]}([x])=q_W\left(\bigcup_{x\in [x]} \pi_W(x)\right).$$ Similarly, for every $[V]\in \ofS$ which is properly nested in, or transverse to, $[W]$, set 
    $$\rho^{[V]}_{[W]}=q_W\left(\bigcup_{V\in [V]} \rho^V_W\right).$$ 
    
    We now check that $(\mc Y, \ofS)$ satisfies the axioms of a HHG structure. 
    \begin{itemize}
        \item \textbf{Projections~\ref{axiom:projections}.} If $[W]\in \ofS$, then $\pi_{[W]}$ is an $E$-coarsely onto, $E$-coarse map, as the quotient projection $q_W\colon \mc C W\to \mc C[W]$ is $1$-Lipschitz and surjective. Moreover, given $[x],[y]\in \mc Y$, let $x\in [x]$ and $y\in [y]$ realize the distance between $[x]$ and $[y]$, so that \[\dist_{[W]}([x],[y])\le \dist_{W}(x,y)\le E\dist_{\mc X}(x,y)+E=E\dist_{\mc Y}([x],[y])+E.\]
        
        \item \textbf{Nesting~\ref{axiom:nesting}.} By construction, the unique maximal element of $\ofS$ is $[S]$. Moreover, whenever $[V]\propnest [W]$ and $\mc C W$ is unbounded, we have that $\diam\rho^{[V]}_{[W]}\le \diam \bigcup_{k\in K}\rho^{kV}_{W},$ and the latter is uniformly bounded since $K$-orbits in $\mc C W$ are uniformly bounded.
        
        \item \textbf{Finite complexity~\ref{axiom:finite_complexity}.} By how nesting is defined, every chain $[U_1]\propnest\ldots\propnest [U_k]$ in $\ofS$ lifts to a chain $U_1\propnest \ldots\propnest U_k$ in $\mf S'$, whose length is uniformly bounded.
        
        \item \textbf{Orthogonality~\ref{axiom:orthogonal}.} We first prove that $\nest$ and $\orth$ are pairwise exclusive. Indeed, suppose that $[V],[W]\in \ofS$ admit representatives $V,V'\in [V]$ and $W,W'\in [W]$ such that $V\orth W$ and $W'\propnest V'$. Up to the $K$-action, we can assume that $V=V'$. But then, since $W\in \mf S'$, there exists $U\nest W$ with $\mc C U$ unbounded.  But $\mc CU$ is thus fixed by $K$ and hence nested in $W'$, as well. This contradicts the fact that, in $\mf S'$, orthogonal domains have no common nested domain.
        
        Moving to the second requirement, let $[U],[V],[W]\in \ofS$ be such that $[U]\propnest [V]$ and $[V]\orth [W]$. As above, we can find representatives $U\propnest V$ and $V\orth W$, so $U\orth W$ and therefore $[U]\orth [W]$.

        \item \textbf{Containers~\ref{axiom:containers}.} Let $[V]\propnest [W]$, and let $[V']\nest [V]$ be such that $\mc C V'$ is unbounded. Let $\mc U=\{[U]\propnest [W]\mid [U]\orth[V]\}$, and notice that every element of $\mc U$ is orthogonal to $[V']$ as well. Fix a representative $W\in [W]$, and for every $[U]\in \mc U$ let $U\in [U]$ be nested in $W$. Since $V'$ is the unique representative of $[V']$, it must be nested in $W$ and orthogonal to every $U$. Thus the container axiom for $\mf S'$ produces a domain $T\propnest W$ which contains every $U$. Hence $[T]\propnest [W]$ and contains all $[U]\in \mc U$, so it is a container for $[V]$ inside $[W]$.
        
        \item \textbf{Transversality~\ref{axiom:transversality}.} Uniform boundedness of projections follows as for the nesting axiom.
        
        \item \textbf{Consistency, hyperbolicity, bounded geodesic image, and large links~\ref{axiom:consistency}-\ref{axiom:hyperbolicity}-\ref{axiom:bounded_geodesic_image}-\ref{axiom:large_link_lemma}} Since $q_W$ is a uniform quasi-isometry for every $W\in \ofS$, all three axioms hold because they hold in $(\mc X, \mf S')$ and (relative) projections are defined via the original (relative) projections.
        
        \item \textbf{Partial realisation~\ref{axiom:partial_realisation}.} Let $[V_1],\ldots,[V_k]\in \ofS$ be pairwise orthogonal, and for every $i$ let $p_i\in \mc C [V_i]$. For every $i$ let $r_i\in q_{V_i}^{-1}(p_i)$, and let $x\in \mc X$ realize the collection $\{r_i\}_{i=1,\ldots, k}$. It is straightforward to see that $[x]$ realizes $\{p_i\}_{i=1,\ldots, k}$ in $\mc Y$.
        
        \item \textbf{Uniqueness~\ref{axiom:uniqueness}.} Let $[x],[y]\in\mc Y$, and let $r>0$ be a constant such that $\dist_{[W]}([x],[y])\le r$ for all $[W]\in\ofS$. Then any $x\in [x]$ and $y\in [y]$ have uniformly close projections to all $W\in \mf S'$ (again, because every $q_W$ is a uniform quasi-isometry). The uniqueness axiom for $(\mc X, \mf S')$ then yields that $x$ and $y$ are uniformly close in $\mc X$, and therefore $\dist_{\mc Y}([x],[y])\le \dist_{\mc X}(x,y)$ is uniformly bounded.     
        
        \item \textbf{HHG structure.} Since $G$ acts metrically properly and coboundedly on $\mc X$, then so does $G/K$ on $\mc Y$. Moreover, the cofinite $G$-action on $\mf S'$ induces a cofinite $G/K$-action on $\ofS$, so $(\mc Y, \ofS)$ is a HHG structure for $G/K$.
    \end{itemize}
    For the ``furthermore'' part of the statement, suppose that $\mc C S$ is unbounded and $G$ is not virtually cyclic. Then $\mc C [S]=\mc C S/K$ is unbounded as well, since $K$ acts with uniformly bounded orbits, and $G/K$ is not virtually cyclic as it is quasi-isometric to $G$.  Hence $G/K$ is an acylindrically hyperbolic HHG by \Cref{lem:acyl_action_on_CS}.
\end{proof}

\section{List of  constants}\label{tableofconstants}
\noindent All constants depending on $M_0$ are bounded linearly in $M_0$. 
\par\medskip
\makebox[\textwidth][c]{\small{
\begin{tabular}{c|c|c}
\textbf{Name} & \textbf{From} & \textbf{Description}  \\\hline

$\delta$ & \Cref{hyp:metric} & Hyperbolicity constant of $X$\\\hline

$\Phi(\lambda,c,\delta)$ & \Cref{lem:stability} & Morse constant for $(\lambda,c)$-quasigeodesics\\\hline

$K$ & \Cref{hyp:metric} & Quasiconvexity constant of $Y\in \mc Y$\\\hline

$J(K,\delta)$ & \Cref{lem:lipschitzproj} & Lipschitz constant of $\pi_Y$\\\hline

$\Omega(\delta, K)$ & \Cref{lem:NppGivesQgeo} & nearest point path is $(1,\Omega)$-quasigeodesic\\\hline

$\hat\delta(\delta, K)$ &\Cref{lem:hyp_constant_cone_off} & hyperbolicity constant of $\hat X$\\\hline

$D(\delta,K)$ &\Cref{lem:DboundSpriano} & $[x,y]^X\subseteq \mc N_D([x,y]^{\hat X})$ \\\hline 

$M_0$ & \Cref{hyp:metric} & $\forall Y\neq Y'\in \mc Y, \,\diam(Y \cap \mc N_{2K+2\delta}(Y'))\le M_0$\\\hline

$M(t) =M(\delta, K,M_0,t)$ & \Cref{lem:geomsep_for_qc} & $\diam(Y \cap \mc N_{t}(Y'))\le M(t)$ \\\hline

$B(\delta, K, M_0)$ & \Cref{lem:boundedproj_for_qc} & $\diam_Y(U)\le B$ \\\hline

$C(\delta, K,M_0)$ & \Cref{lem:strongBGI} & Strong BGI: $\dist^\pi_Y(x,y)\ge C \Rightarrow v_Y\in [x,y]^{\hat X}$ \\\hline

$\theta(\delta, K, M_0)$ & \Cref{P:hyp-projcplx} & $\mc Y$ satisfies projection axioms wrt $\theta$ \\\hline

$R$ & \Cref{hyp:spinning} & $G$-action on $X$ is $R$-cobounded \\\hline

$\Theta(\delta, K,M_0)$ & \Cref{cor:proj_complex_with_points} & $\hat X$ satisfies projection axioms wrt $\Theta$\\\hline

$\widetilde\Theta(\delta, K,M_0,R)$ &\Cref{hyp:spinning} & projection constant with a $G$-action\\\hline

$\Zhe=33\widetilde\Theta$ & \Cref{def:projcomplex} & In $\mc P=\mc P_\Zhe(\mc Y)$, $W\in\link(W')$ iff $\forall Y\,\dist_Y(W,W')\le \Zhe$\\\hline

$L$ & \Cref{hyp:spinning} & Spinning: $\dist^\pi_Y(x,h_Yx)\ge L$ $\forall x\in \hat X-\{v_Y\}$\\\hline

$L_{hyp}(\widetilde \Theta)$ & \Cref{rem:l_hyp} & If $L>L_{hyp}$, $\mc P/N$ is hyperbolic 
\\\hline

$\ol L(\delta, K,M_0,R)\ge L_{hyp}$ & \Cref{eq:bound_on_l_spinning} & If $L>\ol L$, $\hat X/N$ is hyperbolic 
\\\hline

$\tau(\delta,K,M_0,R,L)$ & \Cref{cor:tau} & $\dist_X(x,nx)\ge \tau$ $\forall n\in N-\{1\}$. \\\hline

$E$ & \Cref{defn:HHS} & relative HHG constant of $(G,\mf S)$\\\hline

$A(K,E)$ & \Cref{eqn:A} & in $\hat X$, $\mc N_A(\rho^U_S)$ is coned off for every $U\propnest S$\\\hline

$\widetilde{L}(E,K,M_0)$ & \Cref{eq:tildeL} & if $L>\widetilde{L}$ then $G/N$ is a relative HHG\\\hline

$L'$ & \Cref{eq:bound_on_L'} & Spinning constant in $X'$\\\hline

$\aleph$ &\Cref{rem:YtoU_is_bounded} & $\forall y\in Y$, $\diam(\pi_U(H_Y\cdot y))\le \aleph$ \\\hline

$\beth(\aleph,E)$ &\Cref{prop:ProjsBdd} & bound on $\diam\pi_{\ol U}(\ol x)$ and $\diam\rho_{\ol U}^{\ol V}$ \\\hline

$\Psi$ &\Cref{thm:quotientishhg}.\ref{axiom:uniqueness} & $H_Y$--action on $Y$ is $\Psi$-cobounded\\\hline

$\Delta$ &\Cref{prop:drift} & minimal drift of random walks\\\hline
\end{tabular}
}
}

\bibliography{RandomHHGQuotients}{}
\bibliographystyle{amsalpha}
\end{document}